\documentclass[reqno,11pt]{amsart}
\usepackage[utf8]{inputenc}
\usepackage[margin=1in]{geometry}
\usepackage{amscd,amsfonts,amsmath,amssymb,amsthm}
\usepackage{bbm,dsfont,centernot,mathtools}
\usepackage{color,mathrsfs,stackrel}
\usepackage[dvipsnames]{xcolor}
\usepackage[colorlinks=true,linkcolor=RoyalBlue,citecolor=Red]{hyperref}

\newtheorem{theorem}{Theorem}[section]
\newtheorem{lemma}[theorem]{Lemma}
\newtheorem{proposition}[theorem]{Proposition}
\newtheorem{corollary}[theorem]{Corollary}

\theoremstyle{definition}
\newtheorem{definition}[theorem]{Definition}
\newtheorem{remark}[theorem]{Remark}

\numberwithin{equation}{section}

\allowdisplaybreaks

\def\loc {\mathop {\rm loc}\nolimits}
\def\skw {\mathop {\rm skew}\nolimits}
\def\sym {\mathop {\rm sym}\nolimits}
\def\spn {\mathop {\rm span}\nolimits}
\def\supp {\mathop {\rm supp}\nolimits}

\def\Div {\mathop {\rm Div}\nolimits}
\def\Curl {\mathop {\rm Curl}\nolimits}

\def\dist {\mathop {\rm dist}\nolimits}
\def\det {\mathop {\rm det}\nolimits}

\def\de {{\rm d}}
\def\e {{\rm e}}
\def\R {\mathbb R}
\def\N {\mathbb N}
\def\C {\mathbb C}

\title[A fractional approach to strain-gradient plasticity]{A fractional approach to strain-gradient plasticity: beyond core-radius of discrete dislocations}

\author[S. Almi]{Stefano Almi}
\address[Stefano Almi]{Department of Mathematics and Applications ``R.~Caccioppoli'', University of Naples Federico II, Via Cintia, Monte S. Angelo, 80126 Napoli, Italy.}
\email{stefano.almi@unina.it}

\author[M. Caponi]{Maicol Caponi}
\address[Maicol Caponi]{Department of Mathematics and Applications ``R.~Caccioppoli'', University of Naples Federico II, Via Cintia, Monte S. Angelo, 80126 Napoli, Italy.}
\email{maicol.caponi@unina.it}

\author[M. Friedrich]{Manuel Friedrich} 
\address[Manuel Friedrich]{Department of Mathematics, Friedrich-Alexander Universit\"at Erlangen-N\"urnberg. Cauerstr.~11,
D-91058 Erlangen, Germany}
\email{manuel.friedrich@fau.de}

\author[F. Solombrino]{Francesco Solombrino}
\address[Francesco Solombrino]{Department of Mathematics and Applications ``R.~Caccioppoli'', University of Naples Federico II, Via Cintia, Monte S. Angelo, 80126 Napoli, Italy.}
\email{francesco.solombrino@unina.it}

\date{}

\makeatletter
\@namedef{subjclassname@2020}{%
 \textup{2020} Mathematics Subject Classification}
\makeatother

\begin{document}




\begin{abstract}
We derive a strain-gradient theory for plasticity as the $\Gamma$-limit of discrete dislocation fractional energies, without the introduction of a core-radius. By using the finite horizon fractional gradient introduced by Bellido, Cueto, and Mora-Corral~\cite{BCMC}, we consider a nonlocal model of semi-discrete dislocations, in which the stored elastic energy is computed via the fractional gradient of order $1-\alpha$. As $\alpha$ goes to $0$, we show that suitably rescaled energies $\Gamma$-converge to the macroscopic strain-gradient model of Garroni, Leoni, and Ponsiglione~\cite{GLP}.
\end{abstract}

\maketitle


\section{Introduction}

The derivation of macroscopic plasticity from dislocation models is crucial for mathematical and mechanical purposes, see, e.g.,~\cite{CeLe05, CoGaMu11, CoGaOr15, dLGaPo12, GLP, GSM, Gi17, Gr97, MuScZe14, Po07,SZ}. A well-established approach starts from a mesoscopic and semi-discrete model where the dislocations are modeled by an additional constraint on the deformation. In the planar setting, this corresponds to the assumption that close to an edge dislocation in position~$x_{0}$ with Burgers vector~$\xi$ the strain field~$\beta$ satisfies 
\begin{equation}\label{e:intro1}
\Curl \beta = \xi \delta_{x_{0}}
\end{equation}
see, e.g.,~\cite{ArOr}. In particular, the presence of dislocations prevents $\beta$ from having a global gradient structure. In this scenario, $\beta$ is referred to as an {\em incompatible strain field}. Away from dislocations, the energetics is of elastic type and $\beta$ locally takes the form of a deformation gradient. 

As noticed, e.g., in~\cite{CeLe05}, the constraint~\eqref{e:intro1} implies that $\beta$ is not square integrable close to the dislocation in $x_{0}$. Therefore, elastic energies with quadratic growth, such as those considered in linearized models, cannot capture the behavior of strains fulfilling~\eqref{e:intro1}. As a remedy, regularizations of the energy have to be considered. The probably most common one in the recent literature is the so called {\em core-radius approach}~\cite{BBS, HirLot}, where the energy is computed on a reference configuration after cutting out an $\varepsilon$-ball around each dislocation. In such a perforated domain, the strain field is $L^2$-integrable, but the curl constraint~\eqref{e:intro1} has no clear meaning anymore, and has to be reformulated as an integral-circulation-type condition. Following this approach, the asymptotic behavior for vanishing core-radius parameter~$\varepsilon$ has been investigated under suitable scaling of the energy, both in the linear~\cite{CDLM, GLP, Gi20} and in the nonlinear setting~\cite{ARS23, CGM23, FPP, Gi17, MuScZe14, SZ}. In the most relevant scaling, where the number of dislocations scales as $|\log \varepsilon|$ (the order of the \emph{self-energy} related to one dislocation), the limit energy accounts for the competition between a bulk elastic energy and a plastic term with linear growth defined on a Radon measure~$\mu$. The latter expresses the density of dislocations and is related to the strain $\beta$ by the identity $\Curl \beta = \mu$. An alternative approach with an energy defined on the whole reference configuration has been proposed, for instance, in~\cite{CoGaOr15}. There, a regularization of the dislocation density $\mu$ by standard mollification with compactly supported mollifiers is considered, resulting in square integrable incompatible fields through the curl constraint on $\beta$. In the limit of vanishing regularization, the same dislocation model as in~\cite{GLP} is recovered.

In the present paper, we take the perspective of considering a \emph{nonlocal model} for dislocations, avoiding the core-radius approach. Hereby, we naturally bridge to nonlocal elasticity models depending on fractional gradients, which have been attracting ever increasing interest in recent years, see e.g.~\cite{BCBM3, BCMC, BCMC2, CKS, KrSch, ShSp1, ShSp2}. In a fractional linear elastic model, the stored energy is expressed as
\begin{equation}
\label{e:intro2}
\frac{1}{2} \int_{\Omega} \mathbb{C} \nabla^{1-\alpha} u : \nabla^{1-\alpha} u \, {\rm d} x\,,
\end{equation}
where $\mathbb{C}$ denotes an elasticity tensor and the fractional gradient $\nabla^{1-\alpha}u$ is defined by convolution with a singular potential of Riesz type. As pointed out in~\cite{CKS}, when dealing with elastic models, it is reasonable to postulate that such a singular potential has an integration domain depending on a ball of size $\rho>0$ while keeping the same singularity as the Riesz one. This corresponds to a \emph{finite horizon} of interaction among the particles, as used for instance in peridynamics~\cite{Silling}, and amounts to consider 
\begin{equation*}
\nabla^{1-\alpha}_{\rho} u\coloneqq \mathcal I^\alpha_\rho ( \nabla u) = Q^{1-\alpha}_{\rho} * \nabla u
\end{equation*}
in place of $\nabla^{1-\alpha} u$ in~\eqref{e:intro2}, for a suitable Riesz-type kernel $\mathcal I^\alpha_\rho$ supported in a ball of size~$\rho$. The precise form of~$Q^{1-\alpha}_{\rho}$ is recalled in~\eqref{eq:Qsd} in Section~\ref{s:preliminaries}.

The above approach naturally extends to incompatible fields satisfying the constraints introduced in~\eqref{e:intro1} if one considers a linear elastic energy on $\mathcal I^\alpha_{\rho_\alpha} \beta$, where the fractional parameter $\alpha$ is close to $0$ and the horizon parameter $\rho_{\alpha}$ has to be suitably scaled with~$\alpha$. At this stage, we notice that the adoption of a fractional regularization allows us to keep the differential constraint~\eqref{e:intro1} in a differential form, without the necessity of passing to a circulation-type condition. Equivalently, one may rewrite the strain energy in terms of the auxiliary variable $\hat{\beta}_{\alpha} = \mathcal I^\alpha_{\rho_\alpha} \beta$, corresponding to a soft regularization of the curl constraint~\eqref{e:intro1} with singular kernels in place of mollifiers, see Remark~\ref{rmk:reformulation} for details. Hence, our result described below generalizes the perspective of mollification in~\cite{CoGaOr15} to the case of singular convolution operators (yet, restricted to two dimensions). To our view, this is both of mathematical interest and relevant for applications to fractional elasticity~\cite{Silling,Silling3,Silling2}. 

The main result of the paper concerns the asymptotic analysis of the model when $\alpha \to 0$ for a suitable behavior of the horizon $\rho_{\alpha}$ and a suitable scaling of the energy. We work under the assumption of well-separated dislocations, also expressed in terms of $\rho_\alpha$, cf.~Section~\ref{s:main} and in particular Remark~\ref{newremark}. The limiting model is obtained in terms of $\Gamma$-convergence. In the most relevant scenario, the so called {\em critical regime} (see Theorem~\ref{t:critical}), we recover the linear model of strain-gradient plasticity derived in~\cite{GLP}: the self-energy term reads as 
\begin{equation*}
\int_{\Omega} \varphi \bigg( \frac{{\rm d}\mu}{{\rm d}|\mu|}\bigg) \, {\rm d}|\mu|,
\end{equation*}
where the convex and $1$-homogeneous function~$\varphi$ is obtained by a relaxation procedure of a suitable cell formula (cf.~\eqref{eq:ad-psi} and~\eqref{eq:ad-hatpsi}). The core of our proof consists in deriving a quadratic bound from below for the cell formula, in the form given in Proposition~\ref{prop:cell-formula} and Remark~\ref{rem:Ia-scaling}. Here, we explicitly exploit the structure of the singular Riesz-type potential to compute the asymptotic behavior of the energy. This is done in two steps: first in Lemma~\ref{lem:lower-est} we estimate the approximate energy in terms of a special field $\eta_\xi$. Then, in Lemmas~\ref{prop:Ia-eta-con}--\ref{lemma 4.7} we make use of the explicit expression of $\eta_\xi$ to derive a formula for the asymptotic limit. Although not necessary for our proof, as a byproduct, by introducing a suitable multiplicative factor in our model, we are also able to show that $\varphi$ exactly coincides with the density in~\cite{GLP}, see Remark~\ref{rem:Ia-scaling}. Hence, on the one hand, we recover the same result as in~\cite{GLP}. On the other hand, the cell formula~\eqref{eq:ad-psi} does not require neither excision of a core-radius nor the modification of the curl constraint~\eqref{e:intro1}, which is kept in its original differential form. As a consequence, for compactness and the derivation of the $\Gamma$-liminf inequality no additional modifications of the strain field~$\beta$ need to be performed in order to extract information from~\eqref{e:intro1}. We also point out that, while the asymptotics for the cell formula can be derived also for the classical Riesz potential $\mathcal I^\alpha$ (see Remark~\ref{rem:classic-Riesz}), the use of the finite horizon $\rho$ is essential for compactness, as explained in Remark~\ref{rem:classic-Riesz2}.

\paragraph*{\bf Outlook}
The results contained in the present work deal with a strain-gradient plasticity model obtained as the limit of a fractional regularization of a linear semi-discrete dislocation model under the well-separateness assumption. The extension of our analysis to more general dislocation distributions (see, e.g.,~\cite{Gi20}) and to nonlinear theories~\cite{Gi17, MuScZe14, SZ} is a natural research line that will be subject to future investigation. Finally, we would like to mention a related but different fractional approach to the emergence of topological defects and Ginzburg-Landau energies that has been recently developed in~\cite{ABSS}. 

\paragraph*{\bf Plan of the paper}
In Section~\ref{s:preliminaries} we recall some basic notation and preliminary facts. We present the fractional dislocation model in Section~\ref{s:main}, together with the statements of the asymptotic results in the critical, subcritical, and supercritical regimes (see Theorems~\ref{t:critical},~\ref{t:subcritical}, and~\ref{t:supercritical}, respectively). Section~\ref{s:cell} is devoted to the derivation of the cell formula for the critical and the subcritical regime. Finally, in Sections~\ref{s:proofs}--\ref{s:proofs2} we discuss the proofs of our main results. 


\section{Preliminaries}\label{s:preliminaries}


\subsection{Notation}
The space of $2\times 2$ matrices with real entries is denoted by $\R^{2\times 2}$. Given two matrices $A_1,A_2\in \R^{2\times 2}$, their scalar product is denoted by $A_1:A_2$, and the induced norm of $A\in\R^{2\times 2}$ by $|A|$. The subspace of symmetric matrices is denoted by $\R^{2\times 2}_{\sym}$ and the subspace of skew-symmetric matrices by $\R^{2\times 2}_{\skw}$. Given $A\in\R^{2\times 2}$, we denote by $A^{\sym}\coloneqq\frac{1}{2}(A+A^T)\in\R^{2\times 2}_{\sym}$ its symmetric part. We use $SO(2)$ to denote the special orthogonal group in $\R^2$, consisting of all matrices $A\in\R^{2\times 2}$ satisfying $A^{-1}=A^T$ and $\det A=1$. The identity matrix is denoted by $I \in \R^{2 \times 2}$. 

For all $r>0$ and $x\in\R^2$, we denote by $B_r(x)$ the open ball of radius $r$ and center $x$. If $x=0$, we simply write $B_r$, and we set $\mathbb S^1\coloneqq \partial B_1$. Given a measurable set $E\subseteq\R^2$, we denote by $|E|$ the $2$-dimensional Lebesgue measure of $E$, and by $\chi_E \colon \R^2 \to \{ 0 ,1 \}$ we denote the corresponding characteristic function. Given $\rho>0$ and a set $E\subseteq\R^2$, we define 
\begin{align}\label{neigh}
E_\rho\coloneqq\{x\in\R^2:\dist(x,E)<\rho\}.
\end{align}
Given an open subset $\Omega$ of $\R^2$, we denote by $\mathcal M(\Omega;\R^k)$ the space of vector-valued Radon measures on $\Omega$. Moreover, the set of all distributions on $\Omega$, namely the continuous dual space of $C_c^\infty (\Omega;\R^k)$, endowed with the strong dual topology, is denoted by $\mathcal D'(\Omega;\R^k)$. We denote by $\mathcal S(\R^2;\R^k)$ the space of Schwartz functions $\phi\colon\R^2\to\R^k$. We adopt standard notation for Lebesgue spaces on measurable subsets $E\subseteq\mathbb R^2$ and Sobolev spaces on open subsets $\Omega\subseteq\mathbb R^2$. According to the context, we use $\|\cdot\|_{L^p(E)}$ to denote the norm in $L^p(E;\R^k)$ for all $1\le p\le\infty$ and $k\in\N$. A similar convention is also used to denote the norms in Sobolev spaces. The boundary values of a Sobolev function are always intended in the sense of traces. Since we often need to consider functions defined in the whole space, when needed we always assume that a function $f\in L^p(E;\R^k)$ is defined on $\R^2$, by setting $f=0$ outside $E$. Given
$$
 \beta=\begin{pmatrix}
 \beta_{11} &\beta_{12}\\
 \beta_{21} &\beta_{22}
\end{pmatrix}\in L^1_{\loc}(\R^2;\R^{2\times 2}),$$
we define the curl in the sense of distributions as
$$
 \Curl\beta\coloneqq 
\begin{pmatrix}
\partial_1\beta_{12}-\partial_2\beta_{11}\\
\partial_1\beta_{22}-\partial_2\beta_{21}
\end{pmatrix}\in\mathcal D'(\R^2;\R^2).$$
Let us define
$$J\coloneqq \begin{pmatrix}
0& -1\\
1& 0
\end{pmatrix}\in SO(2).$$
Then, we have
\begin{align}\label{curlicurl}
\Curl\beta=\Div(\beta J)\quad\text{in $\mathcal D'(\R^2;\R^2)$}.
\end{align}
%

\subsection{A notion of fractional gradient and Riesz potential with finite horizon}

We focus on $\R^2$ and introduce the notion of fractional gradient $\nabla^s_\rho$ and Riesz potential $\mathcal I^\alpha_\rho$ of order $s,\alpha\in (0,1)$, both with finite horizon $\rho>0$. We start by recalling the definition of the usual Riesz potential $\mathcal I^\alpha$ in $\R^2$.

\begin{definition}[Riesz potential]
Let $\alpha\in (0,2)$ and $f\colon\R^2\to\R$ be a measurable function satisfying
\begin{equation}\label{eq:Ia-hyp}
\int_{\R^2}\frac{|f(y)|}{(1+|y|)^{2-\alpha}}\,\de y<\infty.
\end{equation}
The {\it $\alpha$-Riesz potential} of $f$ is the function $\mathcal I^\alpha f\colon\R^2\to\R$ defined as
\begin{equation*}
\mathcal I^\alpha f(x)\coloneqq \frac{1}{\gamma_\alpha}\int_{\R^2} \frac{f(y)}{|x-y|^{2-\alpha}}\,\de y\quad\text{for a.e.\ $x\in\R^2$},\qquad \text{with}\quad \gamma_\alpha\coloneqq \frac{\pi 2^\alpha\Gamma(\frac{\alpha}{2})}{\Gamma(\frac{2-\alpha}{2})},
\end{equation*}
where $\Gamma$ denotes the Gamma function.
\end{definition}

\begin{remark}\label{rem:gamma-a}
$\mathcal I^\alpha f$ is well-defined for all measurable functions $f$ satisfying~\eqref{eq:Ia-hyp}, see Proposition~\ref{prop:Riesz} in the Appendix. Moreover, we have
$$\gamma_1=2\pi,\qquad\lim_{\alpha\to 0}\alpha\gamma_\alpha=2\pi,\qquad \lim_{\alpha\to 2}\frac{\gamma_\alpha}{2-\alpha}=2\pi.$$
\end{remark}

We now define the fractional gradient with finite horizon $\nabla^s_\rho$ for all $s\in (0,1)$ (degree of differentiability) and $\rho>0$ (horizon). Following~\cite{BCMC,BCMC2}, we fix a function $\overline w\colon [0,\infty)\to \R$ which satisfies
\begin{itemize}
\item[$(i)$] $\overline w\in C^\infty([0,\infty))$ and $\supp\overline w\subset [0,1)$,
\item[$(ii)$] $0\le \overline w\le 1$ in $[0,\infty)$ and $\overline w= 1$ on $\left[0,\frac{1}{2}\right]$,
\item[$(iii)$] $\overline w$ is nonincreasing on $[0,\infty)$, i.e., $\overline w(s)\le \overline w(t)$ for all $0\le t\le s$.
\end{itemize}
For all $\rho>0$, we define $\overline w_\rho\colon [0,\infty)\to \R$ and $w_\rho\colon \R^2\to \R$ as
\begin{align}\label{barwdef}
\overline w_\rho(t)\coloneqq \overline w\left(\frac{t}{\rho}\right)\quad\text{for all $t\in[0,\infty)$},\qquad w_\rho(x)\coloneqq \overline w_\rho(|x|)\quad\text{for all $x\in\R^2$}.
\end{align}

\begin{definition}[Fractional gradient with finite horizon]
Let $s\in (0,1)$, $\rho>0$, and $\phi\in C^\infty(\R^2)$. The {\it $s$-fractional gradient with finite horizon $\rho$} of $\phi$ is the function $\nabla_\rho^s \phi\colon \R^2\to\R^2$ defined as
$$\nabla_\rho^s\phi(x)\coloneqq \frac{1+s}{\gamma_{1-s}}\int_{\R^2}\frac{\phi(x)-\phi(y)}{|x-y|}\frac{x-y}{|x-y|}\frac{ w_\rho(x-y)}{|x-y|^{1+s}}\,\de y\quad\text{for all $x\in\R^2$}.$$
\end{definition}

\begin{remark}
The function $\nabla^s_\rho \phi$ is well-defined in $\R^2$ and
$\supp(\nabla^s_\rho\phi) \subset(\supp \phi)_\rho$ (recall~\eqref{neigh}). Moreover, if $\phi\in C_c^\infty(\R^2)$ and if we replace the function $w_\rho$ with $1$ in the definition above (which, roughly speaking, corresponds to the case $\rho=\infty$), we obtain the Riesz $s$-fractional gradient $$\nabla^s\phi(x)\coloneqq \frac{1+s}{\gamma_{1-s}}\int_{\R^2}\frac{\phi(x)-\phi(y)}{|x-y|}\frac{x-y}{|x-y|}\frac{1}{|x-y|^{1+s}}\,\de y\quad\text{for all $x\in\R^2$}.$$
We refer to~\cite{CS,ShSp1,ShSp2} for detailed literature on this operator. 
\end{remark}

We recall that the Riesz $s$-fractional gradient $\nabla^s$ can be written in terms of a local gradient via the $(1-s)$-Riesz potential. More precisely, for all $s\in (0,1)$ and $\phi\in C_c^\infty(\R^2)$ we have 
\begin{equation}\label{eq:Riesz-gradient}
\nabla^s\phi=\nabla(\mathcal I^{1-s}\phi)=\mathcal I^{1-s}(\nabla\phi),
\end{equation}
see~\cite[Proposition~2.2]{CS} and~\cite[Theorem~1.2]{ShSp1}. A similar property also applies to $\nabla^s_\rho$, as shown in~\cite{BCMC,BCMC2,CKS}. To formulate this result, we need some additional notation. For all $s\in (0,1)$ and $\rho>0$, we define the functions $\overline q^s_\rho\colon [0,\infty)\to\R$ and $q^s_\rho\colon\R^2\to \R$ as
$$\overline q^s_\rho(t)\coloneqq(1+s)t^{1+s}\int_t^\infty \frac{\overline w_\rho(r)}{r^{2+s}}\,\de r\quad\text{for all $t\in[0,\infty)$},\qquad q^s_\rho(x)\coloneqq \overline q^s_\rho(|x|)\quad\text{for all $x\in\R^2$}.$$
Moreover, we set $Q^s_\rho\colon\R^2\setminus\{0\}\to [0,\infty)$ as
\begin{equation}\label{eq:Qsd}
Q^s_\rho(x)\coloneqq \frac{1}{\gamma_{1-s}}\frac{1}{|x|^{1+s}}q^s_\rho(x)=\frac{1+s}{\gamma_{1-s}}\int_{|x|}^\infty\frac{\overline w_\rho(r)}{r^{2+s}}\,\de r\quad\text{for all $x\in\R^2\setminus\{0\}$}.
\end{equation}
By~\cite[Lemma~4.2]{BCMC2}, for all $s\in (0,1)$ and $\rho>0$, we have
\begin{equation}
Q^s_\rho\in C^\infty(\R^2\setminus\{0\})\cap L^1(\R^2),\qquad\supp (Q^s_\rho)\subset B_\rho.\label{eq:Qsd-p}
\end{equation}

We have the following relation between the fractional gradient with finite horizon $\nabla^s_\rho$ and the (standard) gradient $\nabla$. 

\begin{proposition}[{\cite[Proposition~1 and Equation~(2.13)]{CKS}}]\label{prop:Psd}
Let $s\in (0,1)$ and $\rho>0$. 
\begin{itemize}
\item[$(i)$] For all $\phi\in C^\infty(\R^2)$ we have $Q_\rho^s \phi\in C^\infty(\R^2)$ and 
\begin{equation}\label{eq:nablasd-Isd}
\nabla(Q^s_\rho*\phi)=Q^s_\rho*\nabla \phi=\nabla_\rho^s \phi\quad\text{in $\R^2$}. 
\end{equation}
In particular, if $\phi\in\mathcal S(\R^2)$, then $\nabla_\rho^s \phi\in \mathcal S(\R^2)$, and, if $\phi\in C^\infty_c(\R^2)$, then $\nabla_\rho^s \phi\in C_c^\infty(\R^2)$. 
\item[$(ii)$] There exists a linear operator $P^s_\rho\colon\mathcal S(\R^2)\to\mathcal S(\R^2)$ such that for all $\phi\in \mathcal S(\R^2)$
$$P^s_\rho(Q^s_\rho*\phi)=Q_\rho^s*(P^s_\rho \phi)=\phi\quad\text{in $\R^2$}.$$
\end{itemize}
\end{proposition}

For more details on the operator $P_\rho^s$, we refer to~\cite{CKS}. We only point out that in the limit case $\rho=\infty$, $P_\rho^s$ is nothing else but the $\frac{1-s}{2}$-fractional Laplacian. 

Based on~\eqref{eq:Riesz-gradient} and Proposition~\ref{prop:Psd}, we introduce the Riesz potential with finite horizon. 

\begin{definition}[Riesz potential with finite horizon]\label{rieszhor}
Let $\alpha\in (0,1)$, $\rho>0$, and $f\in L^1_{\loc}(\R^2)$. The {\it $\alpha$-Riesz potential with finite horizon $\rho$} of $f$ is the function $\mathcal I^\alpha_\rho f\colon\R^2\to\R$ defined as
$$
\mathcal I^\alpha_\rho f(x)\coloneqq (Q^{1-\alpha}_\rho*f)(x)=\int_{B_\rho(x)}f(y)Q_\rho^{1-\alpha}(x-y)\,\de y\quad\text{for a.e.\ $x\in\R^2$}.
$$
\end{definition}

\begin{remark}
By~\eqref{eq:Qsd-p} the function $Q_\rho^{1-\alpha}$ is well-defined for all $\alpha\in (0,1)$ and it belongs to $L^1(\R^2)$. Therefore, $\mathcal I^\alpha_\rho f$ is well-defined a.e.\ in $\R^2$, see Lemma~\ref{lem:Iad-int}. We remark that $\mathcal I^\alpha_\rho$ can actually be defined for $\alpha \in (0,2)$ as~\eqref{eq:Qsd}--\eqref{eq:Qsd-p} are well-defined for $s \in (-1,1)$. Moreover, the definition of the Riesz potential with finite horizon is consistent with the classical one. Indeed, if $f\colon \R^2 \to\R$ satisfies~\eqref{eq:Ia-hyp} and we replace the function $\overline{w}_\rho$ in~\eqref{eq:Qsd} by $1$, we obtain the Riesz potential $\mathcal I^\alpha$. 
\end{remark}

For more information about the fractional gradient $\nabla^s_\rho$ and the Riesz potential $\mathcal I^\alpha_\rho$ with finite horizon $\rho>0$, we refer to~\cite{BCMC,BCMC2,CKS} and the Appendix below.


\section{Main Problem}\label{s:main}


In this section we introduce our model of semi-discrete dislocations and present the main results. 

\subsection{The model}

 Let $\Omega\subset\R^2$ be an open, bounded, simply connected set with Lipschitz boundary. We fix two linearly independent vectors $b_1,b_2\in\mathbb S^1$, and we define
$$\mathbb S\coloneqq\spn_{\mathbb Z}\{b_1,b_2\}.$$
Let $(\rho_\alpha)_{\alpha\in (0,1)}\subset (0,1)$ (horizons) and $(N_\alpha)_{\alpha\in (0,1)}\subset\N$ (numbers of dislocations) be such that 
\begin{align}
&\rho_\alpha\to 0\quad\text{as $\alpha\to0$},& &\alpha\log\rho_\alpha\to 0\quad\text{as $\alpha\to0$},\label{eq:a-delta}\\
&N_\alpha\to\infty\quad\text{as $\alpha\to 0$},& &N_\alpha\rho_\alpha^2\to 0\quad\text{as $\alpha\to 0$}.\label{eq:Na-delta}
\end{align} 
Assumptions~\eqref{eq:Na-delta} model the fact that the number of dislocations tends to infinity but the area of the region where the singular integral interacts with the dislocations asymptotically vanishes. 
On the other hand, assumption~\eqref{eq:a-delta} prevents $\rho_\alpha$ from going to zero too quickly, so that the Riesz potential $\mathcal I^\alpha_{\rho_\alpha}$ can keep track of the information due to the dislocations as $\alpha\to 0$ (see also Corollary~\ref{coro:Iad-conv} in the Appendix). 
For all $\alpha\in (0,1)$, we define
\begin{align*}
\mathcal X_\alpha\coloneqq\left\{\sum_{i=1}^M\xi_i\delta_{x_i}\in\mathcal M(\R^2;\R^2):M\in\mathbb N,\,\xi_i\in \mathbb S\setminus\{0\},\,B_{\rho_\alpha}(x_i)\subset\Omega,\,|x_i-x_k|\ge 2\rho_\alpha\text{ for $i\neq k$}\right\},
\end{align*}
and, for a given $\mu\in \mathcal X_\alpha$, we set
\begin{align*}
\mathcal A_\alpha(\mu)\coloneqq &\{\beta\in L^1(\Omega_{\rho_\alpha};\R^{2\times 2})\,:\,\Curl\beta=\mu\text{ in $\mathcal D'(\Omega_{\rho_\alpha};\R^2)$}\}.
\end{align*}
Due to the nonlocality induced by $\mathcal I^\alpha_{\rho_\alpha}$, we assume here well-separated dislocations, namely that the distance between any pair of dislocation points is at least $2\rho_\alpha$ and that $\beta$ is defined in the neighborhood $\Omega_{\rho_\alpha}$ of $\Omega$, see~\eqref{neigh}. Let $\C\colon \R^{2\times 2}\to \R^{2\times 2}$ be a fourth-order elasticity tensor, satisfying the following assumptions:
\begin{itemize}
\item[(C1)] $\C F=\C F^{\sym}\in\R^{2\times 2}_{\sym}$ for all $F\in\R^{2\times 2}$;
\item[(C2)] $\C F_1:F_2=F_1:\C F_2$ for all $F_1,F_2\in\R^{2\times 2}$;
\item[(C3)] there exist $0<\nu_1\le \nu_2$ such that 
$$\nu_1|F^{\sym}|^2\le \C F:F\le \nu_2|F^{\sym}|^2\quad\text{for all $F\in \R^{2\times 2}$}.$$
\end{itemize}
Given $\mu\in\mathcal X_\alpha$ and $\beta\in\mathcal A_\alpha(\mu)$, we define the energy
\begin{equation}\label{eq:Ea}
\mathcal E_\alpha(\mu,\beta)\coloneqq \frac{1}{2}\int_\Omega\C\mathcal I^\alpha_{\rho_\alpha}\beta(x):\mathcal I^\alpha_{\rho_\alpha}\beta(x)\,\de x.
\end{equation}

\begin{remark}[Comparison to core-radius approach and elastic model]\label{newremark}

(i) When the fractional order $\alpha$ depends asymptotically on an \emph{atomic scale} $\varepsilon$ by
\begin{equation}\label{epsalpha}
\alpha(\varepsilon)\sim\frac{1}{|\log\varepsilon|}\quad\text{as $\varepsilon\to 0$},
\end{equation}
our model is closely related to the one in~\cite{GLP} based on standard gradients. In particular, in our setting the horizon $\rho_\alpha$ plays the same role of the hard core-radius $\rho_\varepsilon$ in~\cite{GLP}. The assumption~\cite[Section~2.1(i)]{GLP} corresponds to~\eqref{eq:a-delta} (see also the equivalent formulation (i')) and assumption~\cite[Section~2.1(ii)]{GLP} is exactly~\eqref{eq:Na-delta}.

(ii) As we detail below in Remark~\ref{rem:zeta}(ii), for all $\mu\in\mathcal X_\alpha$ there exists $\zeta\in\mathcal A_\alpha(\mu)$ such that $\mathcal E_\alpha(\mu,\zeta)<\infty$. This is a key difference to the model with standard gradients, where in presence of dislocations the energy is always infinite, unless the elastic energy is restricted to the domain outside the so-called core region surrounding the dislocations. 

(iii) If we consider the special case $\mu=0$ and a regular function $\beta\in\mathcal A_\alpha(0)$, then $\beta=\nabla v$ in $\Omega_{\rho_\alpha}$ for some regular function $v$. In this case, thanks to~\eqref{eq:nablasd-Isd}, the energy~\eqref{eq:Ea} reduces to
\begin{equation}\label{eq:Ea2}
\mathcal E_\alpha(0,\beta)=\mathcal E_\alpha(0,\nabla v)=\frac{1}{2}\int_\Omega\C\nabla_{\rho_\alpha}^{1-\alpha} v(x):\nabla_{\rho_\alpha}^{1-\alpha} v(x)\,\de x,
\end{equation}
which can be interpreted as the analogue of a linearized elasticity model in the context of nonlocal gradients with finite horizon. In this sense,~\eqref{eq:Ea} is the generalization of~\eqref{eq:Ea2} to the case of incompatible fields.
\end{remark}

\begin{remark}[Energy regimes]\label{rem:scaling} 
Note that $N_\alpha$ does not enter explicitly in the model, but it plays a key role for the different scaling regimes that we now discuss. The number $N_\alpha$ represents the typical number of dislocations, i.e., in the above model we expect to have measures of the form $\mu=\sum_{i=1}^{N_\alpha}\xi_i\delta_{x_i}\in \mathcal X_\alpha$. The Riesz potential with finite horizon $\mathcal I^\alpha_{\rho_\alpha}$ introduces a nonlocality of length scale $\rho_\alpha$ into the model. Similarly to~\cite{GLP}, this motivates to decompose the energy $\mathcal E_\alpha(\mu,\beta)$ for $\beta\in \mathcal A_\alpha(\mu)$ into the sum of two terms: the \emph{self-energy}
$$
\mathcal E_\alpha^{\rm self}(\mu,\beta)\coloneqq\frac{1}{2}\int_{\bigcup_{i=1}^{N_\alpha} B_{\rho_\alpha}(x_i)}\C\mathcal I^\alpha_{\rho_\alpha}\beta(x):\mathcal I^\alpha_{\rho_\alpha}\beta(x)\,\de x,
$$
and the \emph{interaction energy}
$$
\mathcal E_\alpha^{\rm inter}(\mu,\beta)\coloneqq\frac{1}{2}\int_{\Omega\setminus \bigcup_{i=1}^{N_\alpha} B_{\rho_\alpha}(x_i)}\C\mathcal I^\alpha_{\rho_\alpha}\beta(x):\mathcal I^\alpha_{\rho_\alpha}\beta(x)\,\de x.
$$ 
As shown in Proposition~\ref{prop:cell-formula}, for $\mu\coloneqq\xi\delta_0$, we have
\begin{align*}
\inf_{\beta\in \mathcal A_\alpha(\mu)}\frac{1}{2}\int_{B_{\rho_\alpha}}\C\mathcal I^\alpha_{\rho_\alpha}\beta(x):\mathcal I^\alpha_{\rho_\alpha}\beta(x)\,\de x\sim \frac{1}{\alpha}\quad\text{as $\alpha\to 0$}.
\end{align*}
Therefore, we expect that 
$$
\mathcal E_\alpha^{\rm self}(\mu,\beta)\sim \frac{N_\alpha}{\alpha}\quad\text{as $\alpha\to 0$}.
$$
Moreover, since the interaction energy is defined outside the singularities, we expect its behavior to be the same as in the core-radius model, i.e.,
$$
\mathcal E_\alpha^{\rm inter}(\mu,\beta)\sim N_\alpha^2\quad\text{as $\alpha\to 0$},
$$
see~\cite[Section~2.2]{GLP}. Therefore, depending on the regime $N_\alpha\ll \frac{1}{\alpha}$, $N_\alpha\gg \frac{1}{\alpha}$, or $N_\alpha\sim \frac{1}{\alpha}$ as $\alpha\to 0$, the self-energy is dominant (subcritical regime), the interaction energy is dominant (supercritical regime), or they are both of the same order (critical regime). With the choice in~\eqref{epsalpha}, we recover the same scaling factors and regimes of~\cite{GLP}. 
\end{remark}

Motivated by Remark~\ref{rem:scaling}, we consider three regimes for the behavior of $N_\alpha$ with respect to $\alpha\to 0$. We define $\mathcal F_\alpha, \mathcal F^{\rm sub}_\alpha, \mathcal F^{\rm super}_\alpha\colon \mathcal M(\R^2;\R^2)\times L^1_{\loc}(\R^2;\R^{2\times 2})\to [0,\infty]$ as 
\begin{itemize}
\item[(1)] {\it Critical regime} ($N_\alpha\sim \frac{1}{\alpha}$ as $\alpha\to 0$): 
\begin{equation}\label{Falphacrit}
 \mathcal F_\alpha(\mu,\beta)\coloneqq 
\begin{cases}
(2\alpha)^2\mathcal E_\alpha(\mu,\beta)&\text{if $\mu\in\mathcal X_\alpha$ and $\beta\in\mathcal A_\alpha(\mu)$},\\
\infty&\text{otherwise}.
\end{cases}
\end{equation} 
\item[(2)] {\it Subcritical regime} ($N_\alpha\ll \frac{1}{\alpha}$ as $\alpha\to 0$): 
\begin{equation}\label{Falphasub}
\mathcal F^{\rm sub}_\alpha(\mu,\beta)\coloneqq 
\begin{cases}
\frac{2\alpha}{N_\alpha}\mathcal E_\alpha(\mu,\beta)&\text{if $\mu\in\mathcal X_\alpha$ and $\beta\in\mathcal A_\alpha(\mu)$},\\
\infty&\text{otherwise}.
\end{cases}
\end{equation}
\item[(3)] {\it Supercritical regime} ($N_\alpha\gg\frac{1}{\alpha}$ as $\alpha\to 0$):
\begin{equation}\label{Falphasuper}
 \mathcal F^{\rm super}_\alpha(\mu,\beta)\coloneqq 
\begin{cases}
\frac{1}{N_\alpha^2}\mathcal E_\alpha(\mu,\beta)&\text{if $\mu\in\mathcal X_\alpha$ and $\beta\in\mathcal A_\alpha(\mu)$},\\
\infty&\text{otherwise}.
\end{cases}
\end{equation}
\end{itemize}
Our goal is to study the $\Gamma$-limit of $\mathcal F_\alpha$, $\mathcal F^{\rm sub}_\alpha$, and $\mathcal F^{\rm super}_\alpha$. In order to recover the exact same limits as in~\cite{GLP}, the above energies are rescaled accordingly with prefactors $2$, see Remark~\ref{a new remark}(ii) below.

 To formulate the asymptotic energy, we need to introduce the {\it self-energy density}. First, let $\xi\in\R^2$, $\alpha\in (0,1)$, and $\rho\in (0,1)$ be fixed. We set
\begin{align}\label{Axirho}
\mathcal A(\xi,\rho)\coloneqq \left\{\beta\in L^1(B_{2\rho};\R^{2\times 2}): \Curl\beta=\xi\delta_0\text{ in $\mathcal D'(B_{2\rho};\R^2)$}\right\},
\end{align}
and, since $\mathcal A(\xi,\rho)\neq \emptyset$ by Remark~\ref{rem:zeta}(i) below, we can consider the infimum problem 
\begin{equation}\label{eq:ad-psi}
\Psi(\xi,\alpha,\rho)\coloneqq\inf_{\beta\in\mathcal A(\xi,\rho)}\frac{1}{2}\int_{B_\rho}\C\mathcal I^\alpha_\rho\beta(x):\mathcal I^\alpha_\rho\beta(x)\,\de x.
\end{equation}
Then, for fixed $\rho\in (0,1)$ we define the asymptotic energy $\psi\colon\R^2\to [0,\infty)$ as
\begin{align}\label{eq:psi}
\psi(\xi)&\coloneqq \lim_{\alpha\to 0}2\alpha\Psi (\xi,\alpha,\rho)\quad\text{for all $\xi\in\R^2$}.
\end{align}
In Section~\ref{s:cell} below, see~\eqref{eq:unif-conv2}, we will show that $\psi$ is well-defined, i.e., the limit in~\eqref{eq:psi} exists and it is independent of $\rho\in (0,1)$. We will also get that 
\begin{align}\label{eq:psi-prop-neu}
\psi(\xi)\ge c|\xi|^2 
\end{align}
for some $c>0$, see~\eqref{eq:psi-prop} below. In particular, this implies that $\Psi(\xi,\alpha,\rho)$ scales like $\alpha^{-1}$, see the scaling discussed in Remark~\ref{rem:scaling} and also Proposition~\ref{prop:cell-formula} below. Finally, we define the self-energy density $\varphi\colon\R^2\to [0,\infty)$ through the following relaxation procedure
\begin{equation}\label{eq:varphi}
\varphi(\xi)\coloneqq\inf\left\{\sum_{k=1}^M\lambda_k\psi(\xi_k)\,:\,\sum_{k=1}^M\lambda_k\xi_k=\xi,\,M\in\N,\,\lambda_k\ge 0,\,\xi_k\in\mathbb S\right\}\quad\text{for all $\xi\in\R^2$}.
\end{equation}
\begin{remark}\label{a new remark} 
(i) As in~\cite{GLP}, the {\it self-energy density} $\varphi$ is a positively 1-homogeneous convex function, and it depends only on the elasticity tensor $\mathbb C$ and the class of admissible Burgers vectors $\{b_1,b_2\}$. By~\eqref{eq:psi-prop-neu}, the infimum is actually a minimum. 

(ii) In view of the particular choice of $2\alpha$ in~\eqref{eq:psi}, we have that $\varphi$ coincides with the self-energy of the core-radius model, see Proposition~\ref{prop:cell-formula} and Remark~\ref{rem:Ia-scaling} below. This is the reason why in the definition of $\mathcal F^{\rm sub}_\alpha$ the scaling factor is chosen as $2\alpha$. Eventually, the prefactor $(2\alpha)^2$ for $\mathcal F_\alpha$ ensures that all three energies $\mathcal F_\alpha$, $\mathcal F^{\rm sub}_\alpha$, $\mathcal F^{\rm super}_\alpha$ are consistent for $N_\alpha = \frac{1}{2\alpha}$. 
\end{remark}

\subsection{Main results}
From now on, we consider a fixed sequence $(\alpha_j)_j$ converging to $0$. For notational convience, we write $(\rho_j)_j$ for the corresponding sequence of horizons $(\rho_{\alpha_j})_j$ and similarly we write $(N_j)_j$ in place of $(N_{\alpha_j})_{j}$. Then~\eqref{eq:a-delta} and~\eqref{eq:Na-delta} read as
\begin{align}\label{all convergences}
\rho_j\to 0, \quad \quad N_j\to\infty, \quad \quad \alpha_j\log\rho_j\to 0, \qquad N_j\rho_j^2\to 0\qquad \text{as $j\to\infty$}.\end{align}
In the critical regime, the limit energy $\mathcal F\colon\mathcal M(\Omega;\R^2)\times L^2(\Omega;\R^{2\times 2})\to [0,\infty]$ takes the form
\begin{equation}\label{eq:F-limit}
\mathcal F(\mu,\beta)\coloneqq
\begin{cases}
\displaystyle\frac{1}{2}\int_\Omega\C\beta(x):\beta(x)\,\de x+\int_\Omega\varphi\left(\frac{\de\mu}{\de|\mu|}\right)\de|\mu| &\text{if $\Curl\beta=\mu$},\\
\displaystyle\infty&\text{otherwise},
\end{cases}
\end{equation}
and our main $\Gamma$-convergence result is the following.
\begin{theorem}[Critical regime]\label{t:critical}
Let $(\alpha_j)_j\subset (0,1)$ and $(\rho_j)_j \subset (0,1)$ be satisfying~\eqref{all convergences} for the choice 
\begin{align}
N_j = \frac{1}{2\alpha_j}. \label{eq:a-da-j-1}
\end{align} 
\begin{itemize}
\item[$(i)$] {\bf Compactness}. Let $(\mu_j,\beta_j)_j\subset\mathcal M(\R^2;\R^2)\times L^1_{\loc}(\R^2;\R^{2\times 2})$ be such that 
\begin{equation}\label{eq:Fa-bound}
 \sup_{j\in\N} \mathcal F_{\alpha_j}(\mu_j,\beta_j)<\infty.
\end{equation}
There exist a (not relabeled) subsequence of $(\alpha_j)_j$, a sequence of skew-symmetric matrices $(S_j)_j\subset \R^{2\times 2}_{\skw}$, and $(\mu,\beta)\in (\mathcal M(\Omega;\R^2)\cap H^{-1}(\Omega;\R^2))\times L^2(\Omega;\R^{2\times 2})$ with $\Curl\beta=\mu$ in $\mathcal D'(\Omega;\R^2)$ such that 
\begin{align}
2\alpha_j\mu_j\xrightharpoonup{*}\mu&\quad\text{in $\mathcal M(\Omega;\R^2)$ as $j\to\infty$},\label{eq:con-cri-1}\\
2\alpha_j(\mathcal I^{\alpha_j}_{\rho_j}\beta_j-S_j)\rightharpoonup \beta&\quad\text{in $L^2(\Omega;\R^{2\times 2})$ as $j\to\infty$}.\label{eq:con-cri-2}
\end{align}

\item[$(ii)$] {\bf $\Gamma$-liminf inequality}. Let $(\mu_j,\beta_j,S_j)_j\subset\mathcal M(\R^2;\R^2)\times L^1_{\loc}(\R^2;\R^{2\times 2})\times\R^{2\times 2}_{\skw}$ and $(\mu,\beta)\in\mathcal M(\Omega;\R^2)\times L^2(\Omega;\R^{2\times 2})$ be satisfying~\eqref{eq:con-cri-1}--\eqref{eq:con-cri-2}. Then,
\begin{align}\label{eq:gamma-liminf}
\liminf_{j\to\infty}\mathcal F_{\alpha_j}(\mu_j,\beta_j)\ge \mathcal F(\mu,\beta).
\end{align}

\item[$(iii)$] {\bf $\Gamma$-limsup inequality}. For all $(\mu,\beta)\in (\mathcal M(\Omega;\R^2)\cap H^{-1}(\Omega;\R^2))\times L^2(\Omega;\R^{2\times 2})$ with $\Curl\beta=\mu$ in $\mathcal D'(\Omega;\R^2)$ there is a sequence $(\mu_j,\beta_j)_j\subset\mathcal M(\R^2;\R^2)\times L^1_{\loc}(\R^2;\R^{2\times 2})$ satisfying $(\mu_j,\beta_j)\in\mathcal X_{\alpha_j}\times \mathcal A_{\alpha_j}(\mu_j)$ for all $j\in\N$,~\eqref{eq:con-cri-1}--\eqref{eq:con-cri-2} $($with $S_j=0)$, and
\begin{equation}\label{eq:gamma-limsup}
\limsup_{j\to\infty}\mathcal F_{\alpha_j}(\mu_j,\beta_j)\le \mathcal F(\mu,\beta).
\end{equation}
\end{itemize}
\end{theorem}

In the subcritical regime, the self-energy is predominant as observed in Remark~\ref{rem:scaling}. Therefore, at the limit as $\alpha\to 0$ we expect that $\mu$ and $\beta$ are no longer related. Indeed, the limit functional $\mathcal F^{\rm sub}\colon\mathcal M(\Omega;\R^2)\times L^2(\Omega;\R^{2\times 2})\to [0,\infty]$ takes the form
\begin{align}\label{subcritlimit}
\mathcal F^{\rm sub}(\mu,\beta)\coloneqq
\begin{cases}
\displaystyle\frac{1}{2}\int_\Omega\C\beta(x):\beta(x)\,\de x+\int_\Omega\varphi\left(\frac{\de\mu}{\de|\mu|}\right)\de|\mu| &\text{if $\Curl\beta=0$},\\
\displaystyle\infty&\text{otherwise}.
\end{cases}
\end{align}
More precisely, we obtain the following $\Gamma$-limit result.

\begin{theorem}[Subcritical regime]\label{t:subcritical}
Let $(\alpha_j)_j\subset (0,1)$, $(\rho_j)_j\subset (0,1)$, and $(N_j)_j\subset\mathbb N$ be satisfying~\eqref{all convergences} and 
\begin{align}
N_j\alpha_j\to 0\quad\text{as $j\to\infty$}\label{eq:a-da-j-4}.
\end{align} 
\begin{itemize}
\item[$(i)$] {\bf Compactness}. Let $(\mu_j,\beta_j)_j\subset\mathcal M(\R^2;\R^2)\times L^1_{\loc}(\R^2;\R^{2\times 2})$ be such that 
\begin{equation*}
 \sup_{j\in\N}\mathcal F^{\rm sub}_{\alpha_j}(\mu_j,\beta_j)<\infty.
\end{equation*}
There exist a (not relabeled) subsequence of $(\alpha_j)_j$, a sequence of skew-symmetric matrices $(S_j)_j\subset \R^{2\times 2}_{\skw}$, and $(\mu,\beta)\in\mathcal M(\Omega;\R^2)\times L^2(\Omega;\R^{2\times 2})$ with $\Curl\beta=0$ in $\mathcal D'(\Omega;\R^2)$ such that 
\begin{align}
\frac{1}{N_j}\mu_j\xrightharpoonup{*}\mu&\quad\text{in $\mathcal M(\Omega;\R^2)$ as $j\to\infty$},\label{eq:con-sub-1}\\
\frac{\sqrt{2\alpha_j}}{\sqrt{N_j}}(\mathcal I^{\alpha_j}_{\rho_j}\beta_j-S_j)\rightharpoonup \beta&\quad\text{in $L^2(\Omega;\R^{2\times 2})$ as $j\to\infty$}.\label{eq:con-sub-2}
\end{align}

\item[$(ii)$] {\bf $\Gamma$-liminf inequality}. Let $(\mu_j,\beta_j,S_j)_j\subset\mathcal M(\R^2;\R^2)\times L^1_{\loc}(\R^2;\R^{2\times 2})\times\R^{2\times 2}_{\skw}$ and $(\mu,\beta)\in\mathcal M(\Omega;\R^2)\times L^2(\Omega;\R^{2\times 2})$ be satisfying~\eqref{eq:con-sub-1}--\eqref{eq:con-sub-2}. Then,
\begin{equation*}
\liminf_{j\to\infty}\mathcal F^{\rm sub}_{\alpha_j}(\mu_j,\beta_j)\ge \mathcal F^{\rm sub}(\mu,\beta).
\end{equation*}

\item[$(iii)$] {\bf $\Gamma$-limsup inequality}. For all $(\mu,\beta)\in\mathcal M(\Omega;\R^2)\times L^2(\Omega;\R^{2\times 2})$ with $\Curl\beta=0$ in $\mathcal D'(\Omega;\R^2)$ there is a sequence $(\mu_j,\beta_j)_j\subset\mathcal M(\R^2;\R^2)\times L^1_{\loc}(\R^2;\R^{2\times 2})$ satisfying $(\mu_j,\beta_j)\in\mathcal X_{\alpha_j}\times \mathcal A_{\alpha_j}(\mu_j)$ for all $j\in\N$,~\eqref{eq:con-sub-1}--\eqref{eq:con-sub-2} $($with $S_j=0)$, and
\begin{equation}\label{eq:gamma-limsup-sub}
\limsup_{j\to\infty}\mathcal F^{\rm sub}_{\alpha_j}(\mu_j,\beta_j)\le \mathcal F^{\rm sub}(\mu,\beta).
\end{equation}
\end{itemize}
\end{theorem}

In the supercritical regime, the interaction energy is dominant as observed in Remark~\ref{rem:scaling}. Hence, in the limit $\alpha\to 0$, the self-energy term disappears. Therefore, as $\alpha\to 0$, $\mathcal F_\alpha$ $\Gamma$-converges to the functional $\mathcal F^{\rm super}\colon L^2(\Omega;\R^{2\times 2}_{\sym})\to [0,\infty]$ defined as
\begin{align}\label{supercritlimit}
\mathcal F^{\rm super}(\beta)=
\frac{1}{2}\int_\Omega\C\beta(x):\beta(x)\,\de x\quad\text{for all $\beta\in L^2(\Omega;\R^{2\times 2}_{\sym})$.}
\end{align}
More precisely, we have the following result. 

\begin{theorem}[Supercritical regime]\label{t:supercritical}
Let $(\alpha_j)_j\subset (0,1)$, $(\rho_j)_j\subset (0,1)$, and $(N_j)_j\subset\mathbb N$ be satisfying~\eqref{all convergences} and 
\begin{align}
&N_j\alpha_j\to \infty \quad\text{as $j\to\infty$}\label{eq:a-da-j-6}.
\end{align} 
\begin{itemize}
\item[$(i)$] {\bf Compactness}. Let $(\mu_j,\beta_j)_j\subset\mathcal M(\R^2;\R^2)\times L^1_{\loc}(\R^2;\R^{2\times 2})$ be such that 
\begin{equation*}
 \sup_{j\in\N}\mathcal F^{\rm super}_{\alpha_j}(\mu_j,\beta_j)<\infty.
\end{equation*}
There exist a (not relabeled) subsequence of $(\alpha_j)_j$ and $\beta\in L^2(\Omega;\R^{2\times 2}_{\sym})$ such that 
\begin{align}
\frac{1}{N_j}\mathcal I^{\alpha_j}_{\rho_j}\beta_j^{\sym}\rightharpoonup \beta\quad\text{in $L^2(\Omega;\R^{2\times 2}_{\sym})$ as $j\to\infty$}.\label{eq:con-sup}
\end{align}

\item[$(ii)$] {\bf $\Gamma$-liminf inequality}. Let $(\mu_j,\beta_j)_j\subset\mathcal M(\R^2;\R^2)\times L^1_{\loc}(\R^2;\R^{2\times 2})$ and $\beta\in L^2(\Omega;\R^{2\times 2}_{\sym})$ be satisfying~\eqref{eq:con-sup}. Then
\begin{equation*}
\liminf_{j\to\infty}\mathcal F^{\rm super}_{\alpha_j}(\mu_j,\beta_j)\ge \mathcal F^{\rm super}(\beta).
\end{equation*}

\item[$(iii)$] {\bf $\Gamma$-limsup inequality}. For all $\beta\in L^2(\Omega;\R^{2\times 2}_{\sym})$ there is a sequence $(\mu_j,\beta_j)_j\subset\mathcal M(\R^2;\R^2)\times L^1_{\loc}(\R^2;\R^{2\times 2})$ satisfying $(\mu_j,\beta_j)\in\mathcal X_{\alpha_j}\times \mathcal A_{\alpha_j}(\mu_j)$ for all $j\in\N$,~\eqref{eq:con-sup}, and
\begin{equation}\label{eq:gamma-limsup-sup}
\limsup_{j\to\infty}\mathcal F^{\rm super}_{\alpha_j}(\mu_j,\beta_j)\le \mathcal F^{\rm super}(\beta).
\end{equation}
\end{itemize}
\end{theorem} 

\begin{remark}\label{rmk:reformulation}
By Lemma~\ref{lem:curl-Iadxi}, we deduce that for all $\mu=\sum_{i=1}^M\xi_i\delta_{x_i}\in\mathcal X_\alpha$ and $\beta\in\mathcal A_\alpha(\mu)$ we have
$$\Curl\mathcal I^\alpha_{\rho_\alpha}\beta=\sum_{i=1}^M\xi_iQ^{1-\alpha}_{\rho_\alpha}(\,\cdot\,-x_i)\quad\text{in $\mathcal D'(\Omega;\R^2)$}.$$ 
Therefore, we could study, in a completely equivalent way, the $\Gamma$-limit as $\alpha\to 0$ of the functionals $\hat{\mathcal E}_\alpha \colon L^1(\Omega;\R^2)\times L^2(\Omega;\R^{2\times 2})\to [0,\infty]$, defined through the energy
\begin{equation*}
\hat{\mathcal E}_\alpha(\hat\mu,\hat\beta)\coloneqq\frac{1}{2}\int_\Omega \C\hat\beta(x):\hat\beta(x)\,\de x\quad\text{for $\hat\mu\in\hat{\mathcal X}_\alpha$ and $\hat\beta\in \hat{\mathcal A}(\hat\mu)$},
\end{equation*}
where
\begin{align*}
\hat{\mathcal X}_\alpha\coloneqq\left\{\sum_{i=1}^M\xi_iQ^{1-\alpha}_{\rho_\alpha}(\,\cdot\,-x_i):M\in\mathbb N,\,\xi_i\in \mathbb S\setminus\{0\},\,B_{\rho_\alpha}(x_i)\subset\Omega,\,|x_i-x_k|\ge 2\rho_\alpha\text{ for $i\neq k$}\right\},
\end{align*}
and
$$\hat{\mathcal A}(\hat\mu)\coloneqq\left\{\hat\beta\in L^2(\Omega;\R^{2\times 2}): \Curl\hat\beta=\hat\mu\text{ in $\mathcal D'(\Omega;\R^2)$}\right\}\quad\text{for $\hat \mu\in \hat{\mathcal X}_\alpha$}.$$
As explained in the introduction, this is in analogy with~\cite{CoGaOr15}, where the authors considered a model of discrete dislocations in which the curl constraint is regularized via mollifiers. We point out that, compared to~\cite{CoGaOr15}, we consider singular convolution operators. 
\end{remark}


\section{Cell formula for the self-energy}\label{s:cell}

In this section, we compute the asymptotic self-energy stored in a neighborhood of a dislocation. As in~\cite{GLP}, this is done by defining a self-energy through a cell problem. Recall the definitions of the cell formulas $\Psi$ and $\psi$ in~\eqref{eq:ad-psi} and~\eqref{eq:psi}, respectively. Our main goal is to show that $\psi$ is well-defined and, as a byproduct, that it coincides with the cell formula for the core-radius model of~\cite{GLP}. The main statement of this section and its proof are contained in Section~\ref{ascel}. In Sections~\ref{auxi-proof}--\ref{auxi-proof3} we then give the proof of some auxiliary statements.

Before we start, we construct elements $\mathcal A_\alpha(\mu)$ for $\mu\in \mathcal X_\alpha$, which in particular provides useful competitors for the minimization problem~\eqref{eq:ad-psi}. 

\begin{remark}[Competitors]\label{rem:zeta}
(i) Let us consider $\mu = \xi\delta_{0} $. We introduce the function 
\begin{equation*}
\zeta(x)\coloneqq \frac{\xi}{2\pi}\otimes \frac{Jx}{|x|^2}\quad \text{for $x\in\R^2\setminus\{0\}$}.
\end{equation*}
Then, $\zeta\in C^\infty(\R^2\setminus\{0\};\R^{2\times 2})\cap L^p_{\loc}(\R^2;\R^{2\times 2})$ for every $p\in [1,2)$ and $\Curl\zeta=\xi\delta_{0}$ in $\mathcal D'(\R^2;\R^2)$. In fact, $\Curl\zeta=0$ in $\R^2\setminus\{0\}$, and for all $\Phi\in C_c^\infty(\R^2;\R^2)$ we have by~\eqref{curlicurl} and the divergence formula
\begin{align*}
\langle\Curl \zeta,\Phi\rangle_{\mathcal D'(\R^2)}&=-\int_{\R^2}\zeta(x)J:\nabla\Phi(x)\,\de x =-\lim_{\varepsilon\to 0}\int_{\R^2\setminus B_\varepsilon}\zeta(x)J:\nabla\Phi(x)\,\de x\\
&=-\lim_{\varepsilon\to 0}\int_{\R^2\setminus B_\varepsilon}\Div ((\zeta(x)J)^T\Phi(x))\,\de x \\&=\lim_{\varepsilon\to 0}\frac{1}{2\pi\varepsilon}\int_{\partial B_\varepsilon}\xi\cdot\Phi(x)\,\de x=\xi\cdot\Phi(0).
\end{align*}
This shows that $\zeta\in \mathcal A(\xi,\rho)$ defined in~\eqref{Axirho} for all $\rho>0$, and thus $\mathcal A(\xi,\rho) \neq \emptyset$. Moreover, Lemma~\ref{lem:Iad-int} implies $\mathcal I^\alpha_{\rho}\zeta\in L^2_{\rm loc}(\R^2;\R^{2\times 2})$ for each $\alpha \in (0,1)$. 
 
(ii) For more general measures $\overline{\mu} = \sum_{i=1}^M\xi_i\delta_{x_i} \in \mathcal X_\alpha$, we define $\overline{\zeta} = \sum_{i=1}^M \zeta_i(\cdot - x_i)$, where $\zeta_i$ is as in (i) with $\xi_i$ in place of $\xi$. Then, $\overline{\zeta} \in \mathcal{A}_\alpha(\overline{\mu}) \neq \emptyset$ and we obtain $\mathcal I^\alpha_{ \rho}\overline{\zeta}\in L^2(\Omega;\R^{2\times 2})$. In particular, we have $\mathcal E_\alpha(\overline{\mu},\overline{\zeta})<\infty$. 
\end{remark}

\subsection{Asymptotic cell formula}\label{ascel}

In this subsection we show that formula~\eqref{eq:psi} is well-defined and we compute it, similarly to~\cite{GLP}. To this aim, given $\xi\in\R^2$, we consider a distributional solution $\eta_\xi\colon\R^2\to \R^{2\times 2}$ to
\begin{equation}\label{eq:eta-sys}
\begin{cases}
\Curl \eta_\xi=\xi\delta_0&\text{in $\R^2$},\\
\Div\C \eta_\xi=0&\text{in $\R^2$}.
\end{cases}
\end{equation}
The function $\eta_\xi$ is smooth in $\R^2\setminus\{0\}$ and in polar coordinates takes the form
\begin{equation}\label{eq:eta-decay}
\eta_\xi(r\cos\theta,r\sin\theta)=\frac{1}{r}\Gamma_\xi(\theta)\quad\text{for all $r\in(0,\infty)$ and $\theta\in[0,2\pi)$},
\end{equation}
where the function $\Gamma_\xi$ depends on the elasticity tensor $\C$, is linear in $\xi$, and satisfies the bound
\begin{equation}\label{eq:eta-est}
|\Gamma_{\xi}(\theta)|\le K|\xi|\quad\text{for all $\theta\in[0,2\pi)$}
\end{equation}
for a constant $K=K(\nu_1,\nu_2)>0$. We refer to~\cite{BBS} and~\cite{GLP} for an exhaustive treatment of the function $\eta_\xi$. We have the following result. 

\begin{proposition}\label{prop:cell-formula}
For all $\xi\in\R^2$ and $\rho\in (0,1)$ we have
\begin{equation}\label{eq:point-conv}
\psi(\xi)\coloneqq \lim_{\alpha\to 0}2\alpha\Psi(\xi,\alpha,\rho)=\lim_{\alpha\to 0}\alpha\int_{B_\rho}\C\mathcal I^\alpha_\rho\eta_\xi(x):\mathcal I^\alpha_\rho\eta_\xi(x)\,\de x= \frac{1}{2}\int_0^{2\pi}\C \Gamma_\xi(\theta):\Gamma_\xi(\theta)\,\de \theta.
\end{equation}
In particular, $\psi(\xi) \ge c |\xi|^2$, see~\eqref{eq:psi-prop} below. Moreover, if $(\rho_\alpha)_{\alpha\in (0,1)}\subset (0,1)$ is such that 
$$\alpha\log\rho_\alpha\to 0\quad\text{as $\alpha\to 0$},$$
then
\begin{equation}\label{eq:unif-conv}
\lim_{\alpha\to 0}\sup_{\xi\in\mathbb S^1}\left|2\alpha\Psi(\xi,\alpha,\rho_\alpha)-\psi(\xi)\right|=0.
\end{equation}
\end{proposition}

\begin{remark}\label{rem:Ia-scaling} 
Thanks to Proposition~\ref{prop:cell-formula}, we derive that our limit energy $\psi$ is independent of $\rho\in (0,1)$. Moreover, in view of~\eqref{eq:point-conv}, $\psi$ coincides with the limit energy defined for the core-radius model in~\cite[Equation~(36)]{GLP}. This identity is the reason why in Section~\ref{s:main} we defined the energies $\mathcal F_\alpha$ and $\psi$ by using $2\alpha$ instead of $\alpha$. 
\end{remark}

By~\eqref{eq:point-conv}, it follows that 
\begin{equation}\label{eq:psi-prop}
\psi(\lambda\xi)=\lambda^2\psi(\xi)\quad\text{for all $\xi\in\R^2$ and $\lambda \in \R$},\qquad \inf_{\xi\in\mathbb S^1}\psi(\xi)=\min_{\xi\in\mathbb S^1}\psi(\xi)=c>0.
\end{equation}
Indeed, if $\psi(\xi)=0$ for some $\xi\in\mathbb S^1$, then, in view of~\eqref{eq:point-conv} and (C3), we find $\Gamma_\xi^{\sym}=0$ on $[0,2\pi)$. Hence, $\eta_\xi^{\sym}=0$ in $\R^2\setminus\{0\}$, thanks to~\eqref{eq:eta-decay}. Since $\Curl\eta_\xi=0$ in $\R^2\setminus\{0\}$, by applying Korn's inequality locally in $\R^2\setminus\{0\}$, we derive that $\eta_\xi$ is constant, which contradicts $\Curl\eta_\xi=\xi\delta_0$ in $\R^2$. 
This implies~\eqref{eq:psi-prop}, and thus~\eqref{eq:psi-prop-neu} holds. If $(\rho_\alpha)_{\alpha\in (0,1)}\subset (0,1)$ is such that 
$$\alpha\log\rho_\alpha\to 0\quad\text{as $\alpha\to 0$},$$
then by~\eqref{eq:unif-conv} and~\eqref{eq:psi-prop} we derive
\begin{equation}\label{eq:unif-conv2}
\lim_{\alpha\to 0}2\alpha\Psi(\xi,\alpha,\rho_\alpha)=\psi(\xi)\quad\text{for all $\xi\in\R^2$}.
\end{equation}
The proof of Proposition~\ref{prop:cell-formula} requires several intermediate results. We state the relevant estimates and show how they imply Proposition~\ref{prop:cell-formula}. Their proofs are deferred to Sections~\ref{auxi-proof}--\ref{auxi-proof3}. First of all, along the lines of~\cite[Lemma~5]{GLP}, we derive the following estimates for the cell formula. 

\begin{lemma}\label{lem:lower-est} 
Let $\xi\in\R^2$, $\alpha\in \left(0,\frac{1}{2}\right)$, and $\rho\in (0,1)$ be fixed. There exists a constant $C>0$, independent of $\xi$, $\alpha$, and $\rho$, such that 
\begin{align}\label{eq:lower-est}
\Psi(\xi,\alpha,\rho)\le \frac{1}{2}\int_{B_\rho}\C \mathcal I^\alpha_\rho\eta_\xi(x):\mathcal I^\alpha_\rho\eta_\xi(x)\,\de x\le\Psi (\xi,\alpha,\rho)+C|\xi|^2.
\end{align}
\end{lemma}

Note that we need a restriction on the range of $\alpha$ to ensure that the constant in the statement is independent of $\alpha$. In view of Lemma~\ref{lem:lower-est}, in order to prove~\eqref{eq:point-conv}, it suffices to show that for all $\rho\in (0,1)$ there exists
$$
\lim_{\alpha\to 0}\alpha\int_{B_\rho}\C\mathcal I^\alpha_\rho\eta_\xi(x):\mathcal I^\alpha_\rho\eta_\xi(x)\,\de x= \frac{1}{2}\int_0^{2\pi}\C \Gamma_\xi(\theta):\Gamma_\xi(\theta)\,\de \theta.
$$
In view of~\eqref{eq:eta-decay}--\eqref{eq:eta-est}, we get that $\eta_\xi$ satisfies~\eqref{eq:Ia-hyp}. Therefore, its Riesz potential $\mathcal I^\alpha\eta_\xi$ is well-defined by Proposition~\ref{prop:Riesz}. We first prove that for all $\xi\in\R^2$ and $\rho\in (0,1)$
\begin{align*}
\lim_{\alpha\to 0}\alpha\int_{B_\rho}\C \mathcal I^\alpha\eta_\xi(x):\mathcal I^\alpha\eta_\xi(x)\,\de x=\frac{1}{2}\int_0^{2\pi}\C \Gamma_\xi(\theta):\Gamma_\xi(\theta)\,\de \theta, 
\end{align*}
and then we use some uniform estimates to pass from $\mathcal I^\alpha\eta_\xi$ to $\mathcal I^\alpha_\rho\eta_\xi$. 

\begin{lemma}\label{prop:Ia-eta-con}
Let $\xi\in\R^2$ and $\rho\in (0,1)$ be fixed. We have
\begin{align}\label{eq:Ia-eta-core-point}
\lim_{\alpha\to 0}\alpha\int_{B_\rho}\C \mathcal I^\alpha\eta_\xi(x):\mathcal I^\alpha\eta_\xi(x)\,\de x=\frac{1}{2}\int_0^{2\pi}\C \Gamma_\xi(\theta):\Gamma_\xi(\theta)\,\de \theta.
\end{align}
Moreover, if $(\rho_\alpha)_{\alpha\in (0,1)}\subset (0,1)$ is such that 
$$\alpha\log\rho_\alpha\to 0\quad\text{as $\alpha\to 0$},$$
then
\begin{equation}\label{eq:Ia-eta-core-unif}
\lim_{\alpha\to 0}\sup_{\xi\in\mathbb S^1}\left|\alpha\int_{B_{\rho_\alpha}}\C \mathcal I^\alpha\eta_\xi(x):\mathcal I^\alpha\eta_\xi(x)\,\de x-\frac{1}{2}\int_0^{2\pi}\C \Gamma_\xi(\theta):\Gamma_\xi(\theta)\,\de \theta \right|=0.
\end{equation}
\end{lemma}

Next, we control the difference of $\mathcal I^\alpha_\varrho\eta_\xi(\omega)$ and $\mathcal I^\alpha\eta_\xi(\omega)$ as $\alpha\to 0$, uniformly in $(\xi,\varrho, \omega)\in\mathbb S^1\times (1,\infty)\times\mathbb S^1$. Later, this will be applied for $\varrho = \frac{\rho}{|x|}$, for $\rho\in (0,1)$ and $x\in B_\rho \setminus \{0\}$.

\begin{lemma}\label{lemma 4.7}
Let $\xi\in\R^2$, $\alpha\in \left(0,\frac{1}{2}\right)$, $\varrho\in (1,\infty)$, and $\omega\in\mathbb S^1$ be fixed. There exists a constant $C>0$, independent of $\xi$, $\alpha$, $\varrho$, and $\omega$, such that 
\begin{align}\label{eq:Iad-Ia-est}
|\mathcal I^\alpha_{\varrho}\eta_\xi(\omega)-\mathcal I^\alpha\eta_\xi(\omega)|\le C|\xi|\alpha.
\end{align}
\end{lemma}

The proofs of the previous three statements are deferred to Sections~\ref{auxi-proof}--\ref{auxi-proof3}, respectively. As the previous lemma is only formulated for $\omega \in \mathbb{S}^1$, we need to make use of the following 
scaling arguments. Using~\eqref{eq:eta-decay}, one can check that for all $\xi\in\R^2$ and $\alpha\in (0,1)$ we have
\begin{equation}\label{eq:Ia-eta-scal}
\mathcal I^\alpha\eta_\xi(x)=\frac{1}{|x|^{1-\alpha}}\mathcal I^\alpha\eta_\xi\left(\frac{x}{|x|}\right)\quad\text{for all $x\in\R^2\setminus\{0\}$}.
\end{equation}
Moreover, for $\rho \in (0,1)$ it holds
\begin{equation}\label{eq:Ia-eta-scal2}
\mathcal I^\alpha_\rho\eta_\xi(x)=\frac{1}{|x|^{1-\alpha}}\mathcal I^\alpha_\frac{\rho}{|x|}\eta_\xi\left(\frac{x}{|x|}\right)\quad\text{for all $x\in\R^2\setminus\{0\}$}.
\end{equation}
To see this, in view of~\eqref{barwdef},~\eqref{eq:Qsd}, and ~\eqref{eq:eta-decay}, for all $x\in \R^2\setminus\{0\}$ and $\lambda>0$ we calculate 
\begin{align*}
\mathcal I^\alpha_\rho\eta_\xi(\lambda x)&=(Q^{1-\alpha}_\rho*\eta_\xi)(\lambda x)=\int_{\R^2}\eta_\xi(y)Q^{1-\alpha}_\rho(\lambda x-y)\,\de y\\
&=\lambda^2\int_{\R^2}\eta_\xi(\lambda z)Q^{1-\alpha}_\rho(\lambda x-\lambda z)\,\de z=\lambda\int_{\R^2}\eta_\xi(z)Q^{1-\alpha}_\rho(\lambda x-\lambda z)\,\de z\\
&=\lambda\frac{2-\alpha}{\gamma_{\alpha}}\int_{\R^2}\eta_\xi(z)\int_{\lambda|x-z|}^\infty \frac{\overline w_\rho(r)}{r^{3-\alpha}}\,\de r\,\de z\\
&=\frac{1}{\lambda^{1-\alpha}}\frac{2-\alpha}{\gamma_{\alpha}}\int_{\R^2}\eta_\xi(z)\int_{|x-z|}^\infty \frac{\overline w_\rho(\lambda s)}{s^{3-\alpha}}\,\de s\,\de z\\
&=\frac{1}{\lambda^{1-\alpha}}\frac{2-\alpha}{\gamma_{\alpha}}\int_{\R^2}\eta_\xi(z)\int_{|x-z|}^\infty \frac{\overline w_\frac{\rho}{\lambda}(s)}{s^{3-\alpha}}\,\de s \,\de z=\frac{1}{\lambda^{1-\alpha}}\mathcal I^\alpha_\frac{\rho}{\lambda}\eta_\xi(x).
\end{align*}
As a final preparation, we note that there exists a constant $C>0$, independent of $\xi$, $\alpha$, $\rho$, and $\omega$, such that for all $\xi\in\R^2$, $\alpha\in (0,\frac{1}{2})$, $\rho\in (0,1)$, and $\omega\in \mathbb S^1$ it holds 
\begin{align}\label{eq:Ia-eta-est}
|\mathcal I^\alpha\eta_\xi(\omega)|&\le C|\xi|, \qquad |\mathcal I^\alpha_\rho\eta_\xi(\omega)|\le C|\xi|.
\end{align}
In fact, by using~\eqref{eq:eta-decay}--\eqref{eq:eta-est}, Proposition~\ref{prop:Riesz-comp}, and $\gamma_1=2\pi$ we derive that
\begin{align*}
|\mathcal I^\alpha\eta_\xi(\omega)|&\le \frac{K|\xi|}{\gamma_{\alpha}}\int_{\R^2}\frac{1}{|y||\omega-y|^{2-\alpha}}\,\de y=\frac{2\pi K|\xi|}{\gamma_{1+\alpha}}\quad\text{for all $\xi\in\R^2$, $\alpha\in (0,1)$, and $\omega\in\mathbb S^1$}.
\end{align*}
This yields the first part of~\eqref{eq:Ia-eta-est}. A similar argument, using~\eqref{eq:Qsd-est}, yields the second part of~\eqref{eq:Ia-eta-est}. In particular, by~\eqref{eq:Ia-eta-scal} and~\eqref{eq:Ia-eta-scal2}, we derive that for all $\xi\in\R^2$, $\alpha\in (0,\frac{1}{2} )$, $\rho \in (0,1)$, and $x\in \R^2$ it holds 
\begin{align}\label{eq:Ia-eta-est2}
|\mathcal I^\alpha\eta_\xi(x)|&\le \frac{C|\xi|}{|x|^{1-\alpha}}, \qquad |\mathcal I^\alpha_\rho\eta_\xi(x)|\le \frac{C|\xi|}{|x|^{1-\alpha}}.
\end{align}

We can now finally prove Proposition~\ref{prop:cell-formula}.

\begin{proof}[Proof of Proposition~\ref{prop:cell-formula}]
Let $\xi\in\R^2$ be fixed. By~\eqref{eq:Ia-eta-scal2}, for all $\alpha\in (0,1)$ and $\rho\in (0,1)$ we can write
\begin{align*}
\int_{B_\rho}\C \mathcal I^\alpha_\rho\eta_\xi(x):\mathcal I^\alpha_\rho\eta_\xi(x)\,\de x&=\int_{B_\rho}\frac{1}{|x|^{2-2\alpha}}\C \mathcal I^\alpha_\frac{\rho}{|x|}\eta_\xi\left(\frac{x}{|x|}\right):\mathcal I^\alpha_\frac{\rho}{|x|}\eta_\xi\left(\frac{x}{|x|}\right)\,\de x.
\end{align*}
By~\eqref{eq:Iad-Ia-est},~\eqref{eq:Ia-eta-est}, and $(\rm C3)$ for all $\alpha\in (0,1)$, $\rho\in (0,1)$, and $x\in B_\rho\setminus\{0\}$ we have
\begin{align*}
\left|\C \mathcal I^\alpha_\frac{\rho}{|x|}\eta_\xi\left(\frac{x}{|x|}\right):\mathcal I^\alpha_\frac{\rho}{|x|}\eta_\xi\left(\frac{x}{|x|}\right)-\C \mathcal I^\alpha\eta_\xi\left(\frac{x}{|x|}\right):\mathcal I^\alpha\eta_\xi\left(\frac{x}{|x|}\right)\right|\le 2\nu_2C^2|\xi|^2\alpha,
\end{align*}
for a constant $C>0$ independent of $\xi$, $\alpha$, $\rho$, and $x$. Therefore, using~\eqref{eq:Ia-eta-scal}--\eqref{eq:Ia-eta-scal2}, for all $\alpha\in (0,1)$, and $\rho\in (0,1)$ we have
\begin{align*}
&\left|\alpha\int_{B_\rho}\C \mathcal I^\alpha_\rho\eta_\xi(x):\mathcal I^\alpha_\rho\eta_\xi(x)\,\de x- \frac{1}{2}\int_0^{2\pi}\C \Gamma_\xi(\theta):\Gamma_\xi(\theta)\,\de \theta \right|\\
&\le \left|\alpha\int_{B_\rho}\C \mathcal I^\alpha\eta_\xi(x):\mathcal I^\alpha\eta_\xi(x)\,\de x- \frac{1}{2}\int_0^{2\pi}\C \Gamma_\xi(\theta):\Gamma_\xi(\theta)\,\de \theta \right|+2\nu_2C^2|\xi|^2\alpha^2\int_{B_\rho}\frac{1}{|x|^{2-2\alpha}}\,\de x\\
&= \left|\alpha\int_{B_\rho}\C \mathcal I^\alpha\eta_\xi(x):\mathcal I^\alpha\eta_\xi(x)\,\de x-\frac{1}{2}\int_0^{2\pi}\C \Gamma_\xi(\theta):\Gamma_\xi(\theta)\,\de \theta \right|+2\pi \nu_2C^2\rho^{2\alpha}|\xi|^2\alpha.
\end{align*}
This implies both~\eqref{eq:point-conv} and~\eqref{eq:unif-conv}, thanks to Lemma~\ref{lem:lower-est} and Lemma~\ref{prop:Ia-eta-con}.
\end{proof}

\begin{remark}\label{rem:classic-Riesz}

(i) Without giving details, we mention that the cell formula can also be defined by introducing an additional parameter $r>0$ as
\begin{align*}
&\mathcal A(\xi,\rho,r)\coloneqq \left\{\beta\in L^1(B_{r+\rho};\R^{2\times 2}): \Curl\beta=\xi\delta_0\text{ in $\mathcal D'(B_{r+\rho};\R^2)$}\right\},\\
&\Psi(\xi,\alpha,\rho,r)\coloneqq\inf_{\beta\in\mathcal A(\xi,\rho,r)}\frac{1}{2}\int_{B_r}\C\mathcal I^\alpha_\rho\beta(x):\mathcal I^\alpha_\rho\beta(x)\,\de x.
\end{align*}
In this case, we still have $\psi(\xi)=\lim_{\alpha\to 0}2\alpha \Psi(\xi,\alpha,\rho,r)$ for all $\rho \in (0,1)$ and $r>0$. 

(ii) An analogous result as in Proposition~\ref{prop:cell-formula} can be shown for the classical Riesz potential $\mathcal I^\alpha$. More precisely, given 
\begin{align}\label{New equa}
&\mathcal A^{\rm Riesz}(\xi)\coloneqq \left\{\beta\in L^1_{\loc}(\R^2;\R^2): \Curl\beta=\xi\delta_0\text{ in $\mathcal D'(\R^2;\R^2)$}\right\}, \notag \\
&\Psi^{\rm Riesz}(\xi,\alpha,r)\coloneqq\inf_{\beta\in\mathcal A^{\rm Riesz}(\xi)}\frac{1}{2}\int_{B_r}\C\mathcal I^\alpha\beta(x):\mathcal I^\alpha\beta(x)\,\de x,
\end{align}
we have $\psi(\xi)=\lim_{\alpha\to 0}2\alpha \Psi^{\rm Riesz}(\xi,\alpha,r)$ for all $r>0$. Formally, this follows by performing the arguments of Proposition~\ref{prop:cell-formula} in the case $\rho=\infty$. This also explains why in this case the condition $\Curl\beta=\xi\delta_0$ is required on the \emph{entire} $\R^2$.
 \end{remark}

\subsection{Proof of Lemma~\ref{lem:lower-est}}\label{auxi-proof}

We start by recalling the Korn's inequality for incompatible fields of~\cite{GLP}, which will be used in the following.

\begin{proposition}[{\cite[Theorem~11]{GLP}}]\label{prop:inc-Korn}
Let $\Omega\subset\R^2$ be an open, bounded, simply connected set with Lipschitz boundary. There exists a constant $C=C(\Omega)>0$ with the following property: for all $\beta\in L^2(\Omega;\R^{2\times 2})$ with $\mu\coloneqq\Curl\beta\in\mathcal M(\Omega;\R^2)$, there is a skew-symmetric matrix $S\in \R^{2\times 2}_{\skw}$ satisfying
$$\|\beta-S\|_{L^2(\Omega)}\le C(\|\beta^{\sym}\|_{L^2(\Omega)}+|\mu|(\Omega)).$$
\end{proposition}

For the proof of Lemma~\ref{lem:lower-est}, we need a further preparation. By Lemma~\ref{lem:curl-Iadxi}, for all $\beta\in \mathcal A(\xi,\rho)$ we have
$$\hat\beta\coloneqq \mathcal I^\alpha_\rho\beta\in L^1(B_\rho;\R^{2\times 2}),\qquad \Curl\hat\beta=\xi Q^{1-\alpha}_\rho\text{ in $\mathcal D'(B_\rho;\R^2)$}.$$
Therefore, we can also consider the set
$$\hat{\mathcal A}(\xi,\alpha,\rho)\coloneqq \left\{\hat\beta\in L^2(B_\rho;\R^{2\times 2}): \Curl\hat\beta=\xi Q^{1-\alpha}_\rho\text{ in $\mathcal D'(B_\rho;\R^2)$}\right\},$$
and the infimum problem
\begin{equation}\label{eq:ad-hatpsi}
 \hat\Psi (\xi,\alpha,\rho)\coloneqq\inf_{\hat\beta\in\hat{\mathcal A}(\xi,\alpha,\rho)}\frac{1}{2}\int_{B_\rho}\C\hat\beta(x):\hat\beta(x)\,\de x.
\end{equation}
Comparing with~\eqref{eq:ad-psi}, we directly get the relation 
\begin{equation}\label{eq:psi-hatpsi-auch}
 \hat\Psi(\xi,\alpha,\rho)\le\Psi(\xi,\alpha,\rho).
\end{equation}
Now, we want to show that the two cell formulas actually coincide. 

\begin{lemma}
For all $\xi\in\R^2$, $\alpha\in(0,1)$, and $\rho\in (0,1)$ we have
\begin{equation}\label{eq:psi-hatpsi}
\Psi(\xi,\alpha,\rho)= \hat\Psi (\xi,\alpha,\rho)=\min_{\hat\beta\in\hat{\mathcal A}(\xi,\alpha,\rho)}\frac{1}{2}\int_{B_\rho}\C\hat\beta(x):\hat\beta(x)\,\de x.
\end{equation}
\end{lemma}

\begin{proof}
We first show that the minimum in ~\eqref{eq:ad-hatpsi} is attained. By~\eqref{eq:psi-hatpsi-auch} and $\Psi(\xi,\alpha,\rho) < \infty$, see Remark~\ref{rem:zeta}(i), we can consider a minimizing sequence $(\hat\beta_n)_n\subset\hat{\mathcal A}(\xi,\alpha,\rho)$ for $\hat{\Psi}(\xi,\alpha,\rho)$. Since by~\eqref{eq:Qsd-est} 
$$\|\Curl\hat\beta_n\|_{L^1(B_\rho)}=|\xi|\int_{B_\rho}Q^{1-\alpha}_\rho(x)\,\de x\le \frac{2\pi\rho^\alpha|\xi|}{\alpha\gamma_\alpha},$$
by employing Proposition~\ref{prop:inc-Korn} and (C1)--(C3) there exist a constant $C>0$ independent of $n$, and a sequence $(S_n)_n\subset \mathbb R^{2\times 2}_{\skw}$ such that 
$${\|\hat\beta_n-S_n\|^2_{L^2(B_\rho)} \le C\left(\int_{B_\rho}\C\hat\beta_n(x):\hat\beta_n(x)\,\de x+\|\Curl\hat\beta_n\|_{L^1(B_\rho)}^2\right)\le C.}$$
Hence, there exists $\hat\beta_\infty\in L^2(B_\rho;\R^{2\times 2})$ such that, up to a not relabeled subsequence,
$$\hat\beta_n-S_n\rightharpoonup\hat\beta_\infty\quad\text{in $L^2(B_\rho;\R^{2\times 2})$ as $n\to\infty$}.$$
Moreover, for all $\Phi\in C_c^\infty(B_\rho;\R^2)$ we have
$$\langle \Curl\hat\beta_\infty,\Phi\rangle_{\mathcal D'(B_\rho)}=\lim_{n\to\infty}\langle\Curl\hat\beta_n,\Phi\rangle_{\mathcal D'(B_\rho)}=\xi\cdot\int_{B_\rho}\Phi(x)Q^{1-\alpha}_\rho(x)\,\de x,$$
which gives that $\hat\beta_\infty\in \hat{\mathcal A}(\xi,\alpha,\rho)$. Finally, we have
\begin{align*}
\frac{1}{2}\int_{B_\rho}\C\hat\beta_\infty(x):\hat\beta_\infty(x)\,\de x&\le \liminf_{n \to\infty}\frac{1}{2}\int_{B_\rho}\C(\hat\beta_n(x)-S_n):(\hat\beta_n(x)-S_n)\,\de x\\
&=\lim_{n\to\infty}\frac{1}{2}\int_{B_\rho}\C\hat\beta_n(x):\hat\beta_n(x)\,\de x=\inf_{\hat\beta\in\hat{\mathcal A}(\xi,\alpha,\rho)}\frac{1}{2}\int_{B_\rho}\C\hat\beta(x):\hat\beta(x)\,\de x. 
\end{align*}
Hence, the minimum in the minimization problem~\eqref{eq:ad-hatpsi} is attained. 

We now show the first identity in~\eqref{eq:psi-hatpsi}. Let $\hat\beta\in \hat{\mathcal A}(\xi,\alpha,\rho)$ and $\zeta$ be the function defined in Remark~\ref{rem:zeta}(i) with $\mu\coloneqq\xi\delta_0$. Lemma~\ref{lem:curl-Iadxi} and Remark~\ref{rem:zeta}(i) imply 
$$\Curl(\hat\beta-\mathcal I^\alpha_\rho\zeta)=0\quad\text{in $\mathcal D'(B_\rho;\R^2)$},\qquad\hat\beta-\mathcal I^\alpha_\rho\zeta\in L^2(B_\rho;\R^{2\times 2}).$$
Thus, there exists $v\in H^1(B_\rho;\R^2)$ such that 
\begin{equation*}
\hat\beta-\mathcal I^\alpha_\rho\zeta=\nabla v\quad\text{in $B_\rho$}.
\end{equation*}
Let $(\Phi_n)_n\subset C_c^\infty(\R^2;\R^2)$ be such that $\Phi_n\to v$ in $H^1(B_\rho;\R^2)$ as $n\to\infty$. We consider the linear operator $P_\rho^{1-\alpha}\colon\mathcal S(\R^2;\R^2)\to\mathcal S(\R^2;\R^2)$ given by Proposition~\ref{prop:Psd}. We define
$$\beta_n\coloneqq \zeta+\nabla (P_\rho^{1-\alpha}\Phi_n)\quad\text{in $\R^2$}.$$
Hence, $\beta_n\in \mathcal A(\xi,\rho)$ for all $n\in\N$. By Proposition~\ref{prop:Psd} and~\eqref{eq:ad-psi} we get
\begin{align*}
\Psi(\xi,\alpha,\rho)&\le \liminf_{n\to\infty}\frac{1}{2}\int_{B_\rho}\C\mathcal I^\alpha_\rho\beta_n(x):\mathcal I^\alpha_\rho\beta_n(x)\,\de x\\
&=\lim_{n\to\infty}\frac{1}{2}\int_{B_\rho}\C(\mathcal I^\alpha_\rho\zeta(x)+\nabla\Phi_n(x)):(\mathcal I^\alpha_\rho\zeta(x)+\nabla\Phi_n(x))\,\de x\\
&=\frac{1}{2}\int_{B_\rho}\C(\mathcal I^\alpha_\rho\zeta(x)+\nabla v(x)):(\mathcal I^\alpha_\rho\zeta(x)+\nabla v(x))\,\de x=\frac{1}{2}\int_{B_\rho}\C\hat\beta(x):\hat\beta(x)\,\de x.
\end{align*}
This along with~\eqref{eq:psi-hatpsi-auch} implies the first identity in~\eqref{eq:psi-hatpsi}.
\end{proof}

\begin{remark}\label{important remark}
Since the energy in~\eqref{eq:ad-hatpsi} is quadratic, we immediately get 
$$\Psi(\lambda\xi,\alpha,\rho) =\lambda^2\Psi (\xi,\alpha,\rho) \quad\text{for all $\xi\in\R^2$ and $\lambda \in \R$}. $$
\end{remark} 

\begin{proof}[Proof of Lemma~\ref{lem:lower-est}]
We fix $\xi\in\R^2$, $\alpha\in (0,\frac{1}{2})$, and $\rho\in (0,1)$. The first inequality in~\eqref{eq:lower-est} is trivially achieved by observing that $\eta_\xi\in\mathcal A(\xi,\rho)$, see~\eqref{eq:eta-sys}. 

Let us prove the second inequality. To this end, let $\hat\beta\in\hat{\mathcal A}(\xi,\alpha,\rho)$ be a minimizer of~\eqref{eq:ad-hatpsi}. We define the annular sets
$$
C_k\coloneqq B_{\rho 2^{-k+1}}\setminus B_{\rho 2^{-k}}\quad\text{for all $k\in\N$}.
$$
Notice that $\bigcup_{k\in\N} C_k= B_\rho\setminus\{0\}$. We define 
\begin{equation}\label{eq:over-k}
\overline k\coloneqq\min\left\{k\in\N:\frac{1}{2}\int_{C_k}\C \hat \beta(x):\hat \beta(x)\,\de x\le \frac{1}{2}\int_{C_k}\C\mathcal I^\alpha_\rho\eta_\xi(x):\mathcal I^\alpha_\rho\eta_\xi(x)\,\de x\right\}. 
\end{equation}
Without restriction, we can assume that the set is nonempty and thus $\overline{k}$ exists, as otherwise we would have
\begin{align*}
\frac{1}{2}\int_{B_\rho}\C\mathcal I^\alpha_\rho\eta_\xi(x):\mathcal I^\alpha_\rho\eta_\xi(x)\,\de x&\le \frac{1}{2}\int_{B_\rho}\C \hat\beta(x):\hat\beta(x)\,\de x = \hat\Psi (\xi,\alpha,\rho) =\Psi(\xi,\alpha,\rho),
\end{align*}
which gives~\eqref{eq:lower-est}. By~\eqref{eq:Ia-eta-est2} there exists a constant $C>0$ (independent of $\xi$, $\alpha$, and $\rho$) such that 
\begin{align}
 \int_{C_{\overline k}} |\mathcal I^\alpha_\rho\eta_\xi(x)|^2 \,\de x & 
\le C^2|\xi|^2\int_{C_{\overline k}} \frac{1}{|x|^{2-2\alpha}} \,\de x = 
 2\pi C^2|\xi|^2\int_{\rho2^{-\overline k}}^{\rho2^{-\overline k+1}}\frac{1}{r^{1-2\alpha}}\,\de r \nonumber\\
&=2\pi C^2|\xi|^2\frac{\rho^{2\alpha}}{2^{2\alpha\overline k}}\frac{2^{2\alpha}-1}{2\alpha} \le C|\xi|^2.\label{eq:Ck-est}
\end{align}
Since $\Curl (\hat\beta-\mathcal I^\alpha_\rho\eta_\xi)=0$ in $\mathcal D'(B_\rho;\R^2)$ by Lemma~\ref{lem:curl-Iadxi} and $\hat\beta-\mathcal I^\alpha_\rho\eta_\xi\in L^2(B_\rho;\R^{2\times 2})$ in view of~\eqref{eq:Ia-eta-est2}, there exists a function $v_1\in H^1(B_\rho;\R^2)$ such that 
\begin{align}\label{dont remove label}
\hat\beta-\mathcal I^\alpha_\rho\eta_\xi=\nabla v_1\quad\text{in $B_\rho$}.
\end{align}
Thanks to~\eqref{eq:over-k},~\eqref{eq:Ck-est}, and (C1)--(C3) we can find a constant $C>0$ (independent of $\xi$, $\alpha$, and $\rho$) such that 
$$\int_{C_{\overline k}}|(\nabla v_1(x))^{\sym}|^2\,\de x\le C|\xi|^2.$$
By Poincarè's and Korn's inequalities, there exist $C>0$ (independent of $\xi$, $\alpha$, and $\rho$), $S\in\R^{2\times 2}_{\skw}$, and $m\in \R^2$ such that 
\begin{align}\label{dont remove label2}
\int_{C_{\overline k}}|\nabla v_1(x)-S|^2\,\de x\le C|\xi|^2,\qquad \int_{C_{\overline k}}|v_1(x)-Sx-m|^2\,\de x\le \rho^2 2^{-2\overline k}C|\xi|^2.
\end{align}
Let $\phi\in C_c^\infty(\R^2)$ be such that 
\begin{align}
0\le \phi\le 1\text{ in $\R^2$},\quad|\nabla\phi|\le \frac{C}{\rho 2^{-\overline k}}\text{ in $\R^2$},\quad\phi=1\text{ in $B_{\rho2^{-\overline k}}$},\quad\phi=0\text{ in $\R^2\setminus B_{\rho2^{-\overline k+1}}$}.\label{eq:over-phi2}
\end{align}
We define the function
$$
v_2(x)\coloneqq \phi(x)(v_1(x)-Sx-m)\quad\text{for $x\in B_\rho$}.
$$
By construction $v_2\in H^1_0(B_\rho;\R^2)$. Thanks to~\eqref{eq:over-k}--\eqref{eq:over-phi2}, we can find a constant $C>0$ (independent of $\xi$, $\alpha$, and $\rho$) such that
\begin{align}
&\frac{1}{2}\int_{B_\rho}\C(\mathcal I^\alpha_\rho\eta_\xi(x)+\nabla v_2(x)):(\mathcal I^\alpha_\rho\eta_\xi(x)+\nabla v_2(x))\,\de x\nonumber\\
&=\frac{1}{2}\int_{B_{\rho 2^{-\overline k}}}\C\hat\beta(x):\hat\beta(x)\,\de x+\frac{1}{2}\int_{B_\rho\setminus B_{\rho2^{-\overline k+1}}}\C\mathcal I^\alpha_\rho\eta_\xi(x):\mathcal I^\alpha_\rho\eta_\xi(x)\,\de x\nonumber\\
&\quad+\frac{1}{2}\int_{C_{\overline k}}\C(\mathcal I^\alpha_\rho\eta_\xi(x)+\nabla v_2(x)):(\mathcal I^\alpha_\rho\eta_\xi(x)+\nabla v_2(x))\,\de x\nonumber\\
&\le \frac{1}{2}\int_{B_{\rho }}\C\hat\beta(x):\hat\beta(x)\,\de x +\frac{1}{2}\int_{C_{\overline k}}\C(\mathcal I^\alpha_\rho\eta_\xi(x)+\nabla v_2(x)):(\mathcal I^\alpha_\rho\eta_\xi(x)+\nabla v_2(x))\,\de x\nonumber\\
&\le\Psi (\xi,\alpha,\rho)+ \nu_2\int_{C_{\overline k}}|\mathcal I^\alpha_\rho\eta_\xi(x)|^2\,\de x +\nu_2\int_{C_{\overline k}}|\nabla v_2(x)|^2\,\de x\nonumber\\
&\le\Psi (\xi,\alpha,\rho)+C \nu_2 |\xi|^2+2\nu_2\int_{C_{\overline k}}|\nabla\phi(x)|^2|v_1(x)-Sx-m|^2\,\de x+2\nu_2\int_{C_{\overline k}}|\phi(x)|^2|\nabla v_1(x)-S|^2\,\de x\nonumber\\
&\le\Psi (\xi,\alpha,\rho)+C|\xi|^2\label{eq:lower-est-1}.
\end{align}
Finally, since $\Div\C\eta_\xi=0$ in $\mathcal D'(\R^2;\R^2)$, see~\eqref{eq:eta-sys}, for all $\Phi\in C_c^\infty(B_\rho;\R^2)$ we have
\begin{align*}
\int_{B_\rho}\C \mathcal I^\alpha_\rho\eta_\xi(x):\nabla \Phi(x)\,\de x=\int_{B_{2\rho}}\C\eta_\xi(x):\mathcal I^\alpha_\rho\nabla \Phi(x)\,\de x=\int_{B_{2\rho}}\C\eta_\xi(x):\nabla \mathcal I^\alpha_\rho\Phi(x)\,\de x=0,
\end{align*}
thanks to Fubini's theorem, Proposition~\ref{prop:Psd}, and the fact that $\mathcal I^\alpha_\rho\Phi\in C_c^\infty(B_{2\rho};\R^2)$, see~\eqref{eq:Qsd-p}. Since $v_2\in H^1_0(B_\rho;\R^2)$, by a density argument we conclude that
\begin{align*}
\int_{B_\rho}\C \mathcal I^\alpha_\rho\eta_\xi(x):\nabla v_2(x)\,\de x=0,
\end{align*}
which gives
\begin{equation}\label{eq:lower-est-2}
\frac{1}{2}\int_{B_\rho}\C \mathcal I^\alpha_\rho\eta_\xi(x):\mathcal I^\alpha_\rho\eta_\xi(x)\,\de x\le \frac{1}{2}\int_{B_\rho}\C(\mathcal I^\alpha_\rho\eta_\xi(x)+\nabla v_2(x)):(\mathcal I^\alpha_\rho\eta_\xi(x)+\nabla v_2(x))\,\de x.
\end{equation}
By combining~\eqref{eq:lower-est-1} and~\eqref{eq:lower-est-2} we derive~\eqref{eq:lower-est}.
\end{proof} 
 
\subsection{Proof of Lemma~\ref{prop:Ia-eta-con}}\label{auxi-proof2}

In order to prove Lemma~\ref{prop:Ia-eta-con}, we need the following auxiliary lemma, which states that for all $\omega\in\mathbb S^1$ the Riesz potential $\mathcal I^\alpha\eta_\xi(\omega)$ converges to $\eta_\xi(\omega)$ as $\alpha\to 0$, uniformly with respect to $\xi\in\mathbb S^1$.

\begin{lemma}\label{lemma 4.6}
Let $\omega\in\mathbb S^1$ be fixed. We have
\begin{align}\label{eq:Ia-eta-unif}
\lim_{\alpha\to 0}\sup_{\xi\in\mathbb S^1}|\mathcal I^\alpha\eta_\xi(\omega)-\eta_\xi(\omega)|=0.
\end{align}
 In particular, for all $\xi\in\R^2$ we derive
\begin{align}\label{eq:Ia-eta-point}
\lim_{\alpha\to 0}\mathcal I^\alpha\eta_\xi(\omega)=\eta_\xi(\omega).
\end{align}
\end{lemma}

\begin{proof} 
Let $\omega\in\mathbb S^1$ be fixed and let $\varepsilon >0$. Since $\eta_{e_i}$ is continuous around $\omega$ for each $i\in\{1,2\}$, there is $\sigma=\sigma(\varepsilon)\in (0,\frac12)$ such that 
$$|\eta_{e_i}(y)-\eta_{e_i}(\omega)|<\frac{\varepsilon}{2}\quad\text{for all $y\in B_\sigma(\omega)$ and $i\in\{1,2\}$}.$$
For all $\xi\in\mathbb S^1$ we can find $\theta\in [0,2\pi)$ such that 
$$\xi=\cos\theta e_1+\sin\theta e_2.$$
Since the system $\eqref{eq:eta-sys}$ is linear in $\xi$, we derive for each $\xi \in \mathbb S^1$ 
\begin{align}\label{stetigkeit}
|\eta_\xi(y)-\eta_\xi(\omega)|\le |\cos\theta||\eta_{e_1}(y)-\eta_{e_1}(\omega)|+|\sin\theta||\eta_{e_2}(y)-\eta_{e_2}(\omega)|<\varepsilon\quad\text{for all $y\in B_\sigma(\omega)$}.
\end{align}
We fix $\xi\in\mathbb S^1$ and $\alpha\in (0,1)$, and we define
\begin{align*}
\mathcal I^\alpha\eta_\xi (\omega)&=g_1^\alpha(\omega) +g_2^\alpha(\omega) +g_3^\alpha(\omega)\\ 
& \coloneqq\frac{1}{\gamma_{\alpha}}\int_{B_\sigma(\omega)}\frac{\eta_\xi(y)}{|\omega-y|^{2-\alpha}}\,\de y+\frac{1}{\gamma_{\alpha}}\int_{B_\sigma}\frac{\eta_\xi(y)}{|\omega-y|^{2-\alpha}}\,\de y+\frac{1}{\gamma_{\alpha}}\int_{\R^2\setminus (B_\sigma(\omega)\cup B_\sigma)}\frac{\eta_\xi(y)}{|\omega-y|^{2-\alpha}}\,\de y.
\end{align*}
Since
$$|\omega-y|\ge 1-\sigma\quad\text{for all $y\in B_\sigma$},$$
by~\eqref{eq:eta-decay}--\eqref{eq:eta-est} we derive
$$|g_2^\alpha(\omega)|\le \frac{K}{\gamma_{\alpha}(1-\sigma)^{2-\alpha}}\int_{B_\sigma}\frac{1}{|y|}\,\de y=\frac{2\pi \sigma K}{\gamma_{\alpha}(1-\sigma)^{2-\alpha}}.$$
Moreover, thanks to 
$$|\omega-y|\ge\sigma=\sigma|y+\omega-y|\ge \sigma|y|-\sigma|\omega-y|\quad\text{for all $y\in \R^2\setminus B_\sigma(\omega)$},$$
together with~\eqref{eq:eta-decay}--\eqref{eq:eta-est} we get
$$|g_3^\alpha(\omega)|\le \frac{(1+\sigma)^{2-\alpha}K}{\gamma_{\alpha} \sigma^{2-\alpha}}\int_{\R^2\setminus B_\sigma}\frac{1}{|y|^{3-\alpha}}\,\de y=\frac{2\pi(1+\sigma)^{2-\alpha}K}{(1-\alpha)\gamma_{\alpha} \sigma^{3-2\alpha}}.$$
Finally, we have
$$\eta_\xi(\omega)=\frac{1}{\gamma_{\alpha}}\int_{B_\sigma(\omega)}\frac{\eta_\xi(\omega)}{|\omega-y|^{2-\alpha}}\,\de y+\eta_\xi(\omega)\left(1-\frac{2\pi \sigma^\alpha}{\alpha\gamma_{\alpha}}\right).$$
Therefore, again by~\eqref{eq:eta-decay}--\eqref{eq:eta-est} and by~\eqref{stetigkeit} we get 
\begin{align*}
|g_1^\alpha(\omega) -\eta_\xi(\omega)|&\le |\eta_\xi(\omega)|\left|1-\frac{2\pi \sigma^\alpha}{\alpha\gamma_{\alpha}}\right|+\frac{1}{\gamma_{\alpha}}\int_{B_\sigma(\omega)}\frac{|\eta_\xi(y)-\eta_\xi(\omega)|}{|\omega-y|^{2-\alpha}}\,\de y\le K\left|1-\frac{2\pi \sigma^\alpha}{\alpha\gamma_{\alpha}}\right|+\varepsilon\frac{2\pi \sigma^\alpha}{\alpha\gamma_{\alpha}},
\end{align*}
which yields
$$\sup_{\xi\in\mathbb S^1}|\mathcal I^\alpha\eta_\xi(\omega)-\eta_\xi(\omega)|\le K\left|1-\frac{2\pi \sigma^\alpha}{\alpha\gamma_{\alpha}}\right|+\varepsilon\frac{2\pi \sigma^\alpha}{\alpha\gamma_{\alpha}}+\frac{2\pi \sigma K}{\gamma_{\alpha}(1-\sigma)^{2-\alpha}}+\frac{2\pi(1+\sigma)^{2-\alpha}K}{(1-\alpha)\gamma_{\alpha} \sigma^{3-2\alpha}}.$$
Thanks to Remark~\ref{rem:gamma-a}, we derive
$$\limsup_{\alpha\to 0}\sup_{\xi\in\mathbb S^1}|\mathcal I^\alpha\eta_\xi(\omega)-\eta_\xi(\omega)|\le \varepsilon.$$
As $\varepsilon >0$ was arbitrary, we have~\eqref{eq:Ia-eta-unif}. Let us now show~\eqref{eq:Ia-eta-point}. If $\xi=0$, then $\eta_\xi=\mathcal I^\alpha\eta_\xi=0$. On the other hand, if $\xi\in\R^2\setminus\{0\}$, by~\eqref{eq:Ia-eta-unif} and the fact that the system $\eqref{eq:eta-sys}$ is linear in $\xi$, we get
$$|\mathcal I^\alpha\eta_\xi(\omega)-\eta_\xi(\omega)|=|\xi|\left|\mathcal I^\alpha\eta_{\frac{\xi}{|\xi|}}(\omega)-\eta_{\frac{\xi}{|\xi|}}(\omega)\right|\to 0\quad\text{as $\alpha\to 0$}.$$
This proves~\eqref{eq:Ia-eta-point}.
\end{proof}

\begin{proof}[Proof of Lemma~\ref{prop:Ia-eta-con}]
We fix $\xi\in\R^2$ and $\rho\in (0,1)$. By a change of variables, Fubini's theorem, and~\eqref{eq:Ia-eta-scal}, for all $\alpha\in (0,1)$, we have
\begin{align}\label{help1}
\alpha\int_{B_\rho}\C \mathcal I^\alpha\eta_\xi(x):\mathcal I^\alpha\eta_\xi(x)\,\de x&=\alpha\int_{B_\rho}\frac{1}{|x|^{2-2\alpha}}\C \mathcal I^\alpha\eta_\xi\left(\frac{x}{|x|}\right):\mathcal I^\alpha\eta_\xi\left(\frac{x}{|x|}\right)\,\de x \notag\\
&=\alpha\int_0^\rho\frac{1}{r^{1-2\alpha}}\,\de r\int_0^{2\pi}\C \mathcal I^\alpha\eta_\xi(\cos\theta,\sin\theta):\mathcal I^\alpha\eta_\xi(\cos\theta,\sin\theta)\,\de \theta.
\end{align}
By~\eqref{eq:eta-decay},~\eqref{eq:Ia-eta-est},~\eqref{eq:Ia-eta-point}, and the dominated convergence theorem we conclude that 
\begin{align}\label{eq:angular-term}
&\lim_{\alpha\to 0}\int_0^{2\pi}\C \mathcal I^\alpha\eta_\xi(\cos\theta,\sin\theta):\mathcal I^\alpha\eta_\xi(\cos\theta,\sin\theta)\,\de\theta\nonumber\\
&=\int_0^{2\pi}\C \eta_\xi(\cos\theta,\sin\theta):\eta_\xi(\cos\theta,\sin\theta)\,\de \theta=\int_0^{2\pi}\C \Gamma_\xi(\theta):\Gamma_\xi(\theta)\,\de \theta.
\end{align}
Moreover, 
\begin{align}\label{help2}
\lim_{\alpha\to 0}\alpha\int_0^\rho\frac{1}{r^{1-2\alpha}}\de r=\lim_{\alpha\to 0}\frac{\rho^{2\alpha}}{2}=\frac{1}{2}.
\end{align} This along with~\eqref{help1}--\eqref{eq:angular-term} shows~\eqref{eq:Ia-eta-core-point}. Finally, let $(\rho_\alpha)_{\alpha\in (0,1)}\subset (0,1)$ be such that 
$\alpha\log\rho_\alpha\to 0$ as $\alpha\to 0$, i.e., $\rho_\alpha^\alpha \to 1$ as $\alpha\to 0$. We fix $\xi\in\mathbb S^1$ and $\alpha\in (0,1)$. Then, by~\eqref{eq:eta-decay}--\eqref{eq:eta-est},~\eqref{eq:Ia-eta-est},~\eqref{help1}, and~\eqref{help2} we have
\begin{align*}
&\left|\alpha\int_{B_{\rho_\alpha}}\C \mathcal I^\alpha\eta_\xi(x):\mathcal I^\alpha\eta_\xi(x)\,\de x-\frac{1}{2}\int_0^{2\pi}\C \Gamma_\xi(\theta):\Gamma_\xi(\theta)\,\de \theta\right|\\
&=\left|\frac{\rho_\alpha^{2\alpha}}{2}\int_0^{2\pi}\C \mathcal I^\alpha\eta_\xi(\cos\theta,\sin\theta):\mathcal I^\alpha\eta_\xi(\cos\theta,\sin\theta)\,\de \theta-\frac{1}{2}\int_0^{2\pi}\mathbb C\eta_\xi(\cos\theta,\sin\theta):\eta_\xi(\cos\theta,\sin\theta)\,\de \theta\right|\\
&\le \rho_\alpha^{2\alpha}\nu_2\max\{C,K\}\int_0^{2\pi}|\mathcal I^\alpha\eta_\xi(\cos\theta,\sin\theta)-\eta_\xi(\cos\theta,\sin\theta)|\,\de \theta+(1-\rho_\alpha^{2\alpha})\pi \nu_2K^2\\
&\le \rho_\alpha^{2\alpha}\nu_2\max\{C,K\}\int_0^{2\pi}\sup_{\xi\in\mathbb S^1}|\mathcal I^\alpha\eta_\xi(\cos\theta,\sin\theta)-\eta_\xi(\cos\theta,\sin\theta)|\,\de \theta+ (1-\rho_\alpha^{2\alpha})\pi \nu_2K^2. 
\end{align*}
Therefore, by~\eqref{eq:Ia-eta-est},~\eqref{eq:Ia-eta-unif}, and the dominated convergence theorem we obtain~\eqref{eq:Ia-eta-core-unif}.
\end{proof} 


\subsection{Proof of Lemma~\ref{lemma 4.7}}\label{auxi-proof3}

This short subsection is devoted to the proof of Lemma~\ref{lemma 4.7}. 

\begin{proof}[Proof of Lemma~\ref{lemma 4.7}]
Let $\xi\in\R^2$, $\alpha\in \left(0,\frac{1}{2}\right)$, $\varrho\in (1,\infty)$, and $\omega\in \mathbb S^1$ be fixed. By~\eqref{eq:Qsd} and~\eqref{eq:eta-decay}--\eqref{eq:eta-est} we have
\begin{align*}
|\mathcal I^\alpha_\varrho\eta_\xi(\omega)-\mathcal I^\alpha\eta_\xi(\omega)|&=\frac{2-\alpha}{\gamma_{\alpha}}\left|\int_{\R^2}\eta_\xi(y)\int_{|\omega-y|}^\infty \frac{1-\overline w_\varrho(r)}{r^{3-\alpha}}\,\de r\,\de y\right|\\
&\le\frac{(2-\alpha)K|\xi|}{\gamma_{\alpha}}\int_{\R^2}\frac{1}{|y|}\int_{|\omega-y|}^\infty \frac{1-\overline w_\varrho(r)}{r^{3-\alpha}}\,\de r\,\de y.
\end{align*}
Since $1-\overline w_\varrho(r)=0$ for all $r\in [0,\frac{\varrho}{2} ]$, see~\eqref{barwdef}, we derive
 \begin{align}\label{eq:Iad-w1}
|\mathcal I^\alpha_\varrho\eta_\xi(\omega)-\mathcal I^\alpha\eta_\xi(\omega)|&\le\frac{(2-\alpha)K|\xi|}{\gamma_{\alpha}}\int_{\R^2}\frac{1}{|y|}\int_{\max\{|\omega-y|,\frac{\varrho}{2}\}}^\infty \frac{1}{r^{3-\alpha}}\,\de r\,\de y\notag\\
&=\frac{K|\xi|}{\gamma_{\alpha}}\int_{\R^2}\frac{1}{|y|}\frac{1}{\max\{|\omega-y|,\frac{\varrho}{2}\}^{2-\alpha}}\,\de y\notag\\
&=\frac{2^{2-\alpha}K|\xi|}{\gamma_{\alpha}\varrho^{2-\alpha}}\int_{B_\frac{\varrho}{2}(\omega)}\frac{1}{|y|}\,\de y+\frac{K|\xi|}{\gamma_{\alpha}}\int_{\R^2\setminus B_\frac{\varrho}{2}(\omega)}\frac{1}{|y||\omega-y|^{2-\alpha}}\,\de y.
\end{align}
As $\varrho\in (1,\infty)$, we have $B_\frac{\varrho}{2}(\omega)\subset B_{\frac{3\varrho}{2}}$. Therefore,
\begin{align}\label{eq: do nooot remove!}
\frac{2^{2-\alpha}K|\xi|}{\gamma_{\alpha}\varrho^{2-\alpha}}\int_{B_\frac{\varrho}{2}(\omega)}\frac{1}{|y|}\,\de y&\le \frac{2^{2-\alpha}K|\xi|}{\gamma_{\alpha}\varrho^{2-\alpha}}\int_{B_\frac{3\varrho}{2}}\frac{1}{|y|}\,\de y=\frac{2^{2-\alpha}3\pi K|\xi|}{\gamma_{\alpha}\varrho^{1-\alpha}}\le \frac{12\pi K}{\alpha\gamma_{\alpha}}|\xi|\alpha.
\end{align}
Moreover, we have
$$|\omega-y|\ge \frac{\varrho}{2}=\frac{\varrho}{2}|y+\omega-y|\ge \frac{\varrho}{2}|y|-\frac{\varrho}{2}|\omega-y|\quad\text{for all $y\in \R^2\setminus B_\frac{\varrho}{2}(\omega)$},$$
which gives
$$|\omega-y|\ge\frac{\varrho}{2+\varrho}|y|\ge \frac{1}{3}|y|\quad\text{for all $y\in\R^2\setminus B_\frac{\varrho}{2}(\omega)$}.$$
Hence,
\begin{align}\label{eq:Iad-w2}
&\frac{K|\xi|}{\gamma_{\alpha}}\int_{\R^2\setminus B_\frac{\varrho}{2}(\omega)}\frac{1}{|y||\omega-y|^{2-\alpha}}\,\de y\nonumber\\
&=\frac{K|\xi|}{\gamma_{\alpha}}\int_{B_\varrho\setminus B_\frac{\varrho}{2}(\omega)}\frac{1}{|y||\omega-y|^{2-\alpha}}\,\de y+\frac{K|\xi|}{\gamma_{\alpha}}\int_{\R^2\setminus (B_\varrho\cup B_\frac{\varrho}{2}(\omega))}\frac{1}{|y||\omega-y|^{2-\alpha}}\,\de y\nonumber\\
&\le \frac{2^{2-\alpha}K|\xi|}{\gamma_{\alpha}\varrho^{2-\alpha}}\int_{B_\varrho}\frac{1}{|y|}\,\de y+\frac{3^{2-\alpha}K|\xi|}{\gamma_{\alpha}}\int_{\R^2\setminus B_\varrho}\frac{1}{|y|^{3-\alpha}}\,\de y\le \frac{8\pi K}{\alpha\gamma_{\alpha}}|\xi|\alpha+\frac{18\pi K}{(1-\alpha)\alpha\gamma_{\alpha}}|\xi|\alpha.
\end{align}
By~\eqref{eq:Iad-w1}--\eqref{eq:Iad-w2} and Remark~\ref{rem:gamma-a} we derive~\eqref{eq:Iad-Ia-est}. This concludes the proof. 
\end{proof}


\section{The critical regime}\label{s:proofs}

In this section we prove our $\Gamma$-limit result in the critical regime ($N_\alpha\sim \frac{1}{\alpha}$ as $\alpha\to 0$), that is Theorem~\ref{t:critical}. Recall~\eqref{Falphacrit} and~\eqref{eq:F-limit}. Let $(\alpha_j)_j$ and $(\rho_j)_j$ be two sequences satisfying~\eqref{all convergences}, and let $N_j = \frac{1}{2\alpha_j}$, see~\eqref{eq:a-da-j-1}. 
 
The proof of Theorem~\ref{t:critical} is divided into three parts. First, we show a compactness result, which justifies the topology for the $\Gamma$-limit. Then we prove the $\Gamma$-liminf inequality and finally the $\Gamma$-limsup inequality.

\begin{proof}[Proof of Theorem~\ref{t:critical}(i)]
By~\eqref{eq:Fa-bound} we deduce that $\mu_j\in\mathcal X_{\alpha_j}$ and $\beta_j\in\mathcal A_{\alpha_j}(\mu_j)$ for all $j\in\N$. In particular,
\begin{equation*}
\mu_j=\sum_{i=1}^{M_j}\xi_{i,j}\delta_{x_{i,j}},\qquad\Curl\beta_j=\mu_j\quad\text{in $\mathcal D'(\Omega_{\rho_j};\R^2)$},
\end{equation*}
where $\xi_{i,j}\in \mathbb S\setminus\{0\}$, $B_{\rho_{j}}(x_{i,j})\subset\Omega$, and $|x_{i,j}-x_{k,j}|\ge 2\rho_{j}$ for $i\neq k$. We first show that the sequence $(2\alpha_j\mu_j)_j\subset\mathcal M(\Omega;\R^2)$ is uniformly bounded. For all $j\in\N$, we have
\begin{align}
C&\ge 2\alpha_j^2\int_\Omega \C\mathcal I^{\alpha_j}_{\rho_j}\beta_j(x):\mathcal I^{\alpha_j}_{\rho_j}\beta_j(x)\,\de x\nonumber\\
&\ge 2\alpha_j^2\sum_{i=1}^{M_j}\int_{B_{\rho_j}} \C\mathcal I^{\alpha_j}_{\rho_j}\beta_j(x+x_{i,j}):\mathcal I^{\alpha_j}_{\rho_j}\beta_j(x+x_{i,j})\,\de x.\label{eq:com1}
\end{align}
 Since the functions $\beta_j\in L^1(B_{2\rho_j}(x_{i,j});\R^{2\times 2})$ and $\Curl\beta_j=\xi_{i,j}\delta_{x_{i,j}}$ in $\mathcal D'(B_{2\rho_j}(x_{i,j}))$, we derive 
\begin{align}
 2\alpha_j^2\sum_{i=1}^{M_j}\int_{B_{\rho_j}} \C\mathcal I^{\alpha_j}_{\rho_j}\beta_j(x+x_{i,j}):\mathcal I^{\alpha_j}_{\rho_j}\beta_j(x+x_{i,j})\,\de x&\ge 4\alpha_j^2\sum_{i=1}^{M_j}\Psi(\xi_{i,j},\alpha_j,\rho_j)\nonumber\\
&= 4\alpha_j^2\sum_{i=1}^{M_j}|\xi_{i,j}|^2\Psi\left(\frac{\xi_{i,j}}{|\xi_{i,j}|},\alpha_j,\rho_j\right),\label{eq:beta-admissible}
\end{align}
 where $\Psi$ is the function defined in~\eqref{eq:ad-psi} and we used Remark~\ref{important remark}. Let us set 
$$\overline{c}\coloneqq \inf_{\xi\in\mathbb S^1} \psi(\xi),$$
where $\psi$ is the function defined in~\eqref{eq:psi}. We have $\overline{c}>0$ by~\eqref{eq:psi-prop-neu}. By Proposition~\ref{prop:cell-formula}, see~\eqref{eq:unif-conv}, we can find $j_0\in \N$ such that 
\begin{equation}
2\alpha_j\Psi\left(\frac{\xi_{i,j}}{|\xi_{i,j}|},\alpha_j,\rho_j\right)\ge \frac{\overline{c}}{2}\quad\text{for all $j\ge j_0$}.
\end{equation}
Moreover, since $\xi_{i,j}\in\mathbb S\setminus \lbrace 0 \rbrace $, there is $c>0$ such that 
\begin{equation}\label{eq:com4}
|\xi_{i,j}|\ge c\quad\text{for all $i\in\{1,\dots,M_j\}$ and $j\in\N$}.
\end{equation}
By combining~\eqref{eq:com1}--\eqref{eq:com4} we can find $C>0$ such that
\begin{equation}\label{eq:com5}
2\alpha_j|\mu_j|(\Omega)= 2\alpha_j\sum_{i=1}^{M_j}|\xi_{i,j}|\le \frac{1}{c}2\alpha_j\sum_{i=1}^{M_j}|\xi_{i,j}|^2\le C\quad\text{for all $j\in\N$}.
\end{equation}
Hence, there is $\mu\in\mathcal M(\Omega;\R^2)$ such that, up to a not relabeled subsequence,
\begin{equation*}
2\alpha_j\mu_j\xrightharpoonup{*}\mu\quad\text{in $\mathcal M(\Omega;\R^2)$}\text{ as $j\to\infty$}.
\end{equation*}

We now show that there is a sequence of skew-symmetric matrices $(S_j)_j\subset\R^{2\times 2}_{\skw}$ such that the sequence $(2\alpha_j(\mathcal I^{\alpha_j}_{\rho_j}\beta_j-S_j))_j$ is uniformly bounded in $L^2(\Omega;\R^{2\times 2})$. By Lemma~\ref{lem:curl-Iadxi} and Proposition~\ref{prop:inc-Korn}, there exists a constant $C>0$ (independent of $j$) and a sequence $(S_j)_j\subset \R^{2\times 2}_{\skw}$ satisfying
\begin{equation*}
\|\mathcal I^{\alpha_j}_{\rho_j}\beta_j -S_j\|_{L^2(\Omega)}\le C(\|\mathcal I^{\alpha_j}_{\rho_j}\beta_j^{\sym}\|_{L^2(\Omega)}+\|\Curl \mathcal I^{\alpha_j}_{\rho_j}\beta_j\|_{L^1(\Omega)})\quad\text{for all $j\in \N$}.
\end{equation*}
Here, we used $(\mathcal I^{\alpha_j}_{\rho_j}\beta_j)^{\sym} = \mathcal I^{\alpha_j}_{\rho_j}\beta_j^{\sym}$, which follows directly from Definition~\ref{rieszhor}. By (C1)--(C3) we have
\begin{align*}
\|2\alpha_j\mathcal I^{\alpha_j}_{\rho_j}\beta_j^{\sym}\|_{L^2(\Omega)}^2\le \frac{4\alpha_j^2}{\nu_1}\int_\Omega \C\mathcal I^{\alpha_j}_{\rho_j}\beta_j(x):\mathcal I^{\alpha_j}_{\rho_j}\beta_j(x)\,\de x\le C, 
\end{align*}
and by~\eqref{eq:com5} along with Lemma~\ref{lem:curl-Iadxi} and~\eqref{eq:Qsd-est} it holds that
\begin{align*}
\|2\alpha_j\Curl \mathcal I^{\alpha_j}_{\rho_j}\beta_j\|_{L^1(\Omega)}&=2\alpha_j\sum_{i=1}^{M_j}|\xi_{i,j}|\int_{B_{\rho_j}(x_j)} Q^{1-{\alpha_j}}_{\rho_j}(x-x_{i,j})\,\de x\\
&\le\frac{2\alpha_j}{\gamma_{{\alpha_j}}}\sum_{i=1}^{M_j}|\xi_{i,j}|\int_{B_{\rho_j}}\frac{1}{|x|^{2-\alpha_j}}\,\de x=\frac{2\pi\rho_j^{\alpha_j}}{\alpha_j\gamma_{{\alpha_j}}}2\alpha_j\sum_{i=1}^{M_j}|\xi_{i,j}|\le C,
\end{align*}
where in the last step we also used Remark~\ref{rem:gamma-a}. 
Hence, we can find a constant $C>0$ such that
\begin{equation*}
\|2\alpha_j(\mathcal I^{\alpha_j}_{\rho_j}\beta_j -S_j)\|_{L^2(\Omega)}\le C\quad\text{for all $j\in \N$}.
\end{equation*}
Therefore, there is $\beta\in L^2(\Omega;\R^{2\times2})$ such that, up to a not relabeled subsequence,
\begin{equation*}
2\alpha_j(\mathcal I^{\alpha_j}_{\rho_j}\beta_j -S_j)\rightharpoonup\beta\quad\text{in $L^2(\Omega;\R^{2\times 2})$}\text{ as $j\to\infty$}.
\end{equation*}
It remains to prove that $\mu\in H^{-1}(\Omega;\R^2)$ and $\Curl\beta=\mu$ in $\mathcal D'(\Omega;\R^2)$. To this aim, we fix $\Phi\in C_c^\infty(\Omega;\R^2)$, and we observe that
\begin{align*}
\langle \Curl\beta,\Phi\rangle_{\mathcal D'(\Omega)}&=\lim_{j\to\infty}2\alpha_j\langle \Curl (\mathcal I^{\alpha_j}_{\rho_j}\beta_j-S_j),\Phi\rangle_{\mathcal D'(\Omega)}=\lim_{j\to\infty}2\alpha_j\langle \Curl \mathcal I^{\alpha_j}_{\rho_j} \beta_j,\Phi\rangle_{\mathcal D'(\Omega)}.
\end{align*} 
By Lemma~\ref{lem:curl-Iadxi} we have
\begin{align*}
2\alpha_j\langle \Curl \mathcal I^{\alpha_j}_{\rho_j} \beta_j,\Phi\rangle_{\mathcal D'(\Omega)}&=2\alpha_j\sum_{i=1}^{M_j}\xi_{i,j}\cdot\int_{\R^2}\Phi(x)Q^{1-{\alpha_j}}_{\rho_j}(x-x_{i,j})\,\de x\\
&=2\alpha_j\sum_{i=1}^{M_j}\xi_{i,j}\cdot\mathcal I^{\alpha_j}_{\rho_j}\Phi(x_{i,j})=\int_{\R^2}\mathcal I^{\alpha_j}_{\rho_j}\Phi(x)\cdot\de (2\alpha_j\mu_j).
\end{align*}
Since $\rho_j\to 0$ as $j\to\infty$, we derive that $\mathcal I^{\alpha_j}_{\rho_j}\Phi\in C_c^\infty(\Omega;\R^2)$ for $j$ sufficiently large, see~\eqref{eq:Qsd-p}. Moreover, $\mathcal I^{\alpha_j}_{\rho_j}\Phi\to \Phi$ in $C(\overline\Omega;\R^2)$ as $j\to\infty$, see Corollary~\ref{coro:Iad-conv}. Therefore,
\begin{align*}
\langle \Curl\beta,\Phi\rangle_{\mathcal D'(\Omega)}&=\lim_{j\to\infty}\int_\Omega\mathcal I^{\alpha_j}_{\rho_j}\Phi(x)\cdot\de(2\alpha_j\mu_j)=\int_\Omega\Phi(x)\cdot\de\mu,
\end{align*}
which gives $\Curl\beta=\mu$ in $\mathcal D'(\Omega;\R^2)$. Finally, by~\eqref{curlicurl} notice that
$${\langle \mu,\Phi\rangle_{\mathcal D'(\Omega)}=\langle \Curl\beta,\Phi\rangle_{\mathcal D'(\Omega)}=-\int_\Omega\beta(x)J:\nabla\Phi(x)\,\de x\quad\text{for all $\Phi\in C_c^\infty(\Omega;\R^2)$}.}$$
In particular, since $\beta\in L^2(\Omega;\R^{2\times 2})$ we find that $\mu\in H^{-1}(\Omega;\R^2)$. This concludes the proof. 
\end{proof}

\begin{remark}\label{rem:classic-Riesz2}
In the proof of the compactness argument, it is important to consider the Riesz potential with finite horizon, for this allows us to localize the dislocation energy. More precisely, in~\eqref{eq:beta-admissible} we used that $\beta_j$ is admissible for the asymptotic energy $\Psi(\xi_{i,j},\alpha_j,\rho_j)$, since $\Curl\beta_j=\xi_{i,j}\delta_{x_{i,j}}$ in $B_{2\rho_j}(x_{i,j})$. On the contrary, when we are dealing with the classical Riesz potential, a field $\beta$ should satisfy $\Curl\beta=\xi_{i,j}\delta_{x_{i,j}}$ (in the distributional sense) on the entire $\R^2$, see~\eqref{New equa}. This, however, is not compatible with the fact that $\Curl\beta_j=\sum_{i=1}^{M_j}\xi_{i,j}\delta_{x_{i,j}}$ in $\R^2$.
\end{remark}

We now prove the $\Gamma$-liminf inequality.

\begin{proof}[Proof of Theorem~\ref{t:critical}(ii)]
Without loss of generality, we may assume that 
\begin{align}\label{eq:Mj-bound0}
\exists\lim_{j\to\infty}\mathcal F_{\alpha_j}(\mu_j,\beta_j)<\infty, \qquad C\coloneqq\sup_j\mathcal F_{\alpha_j}(\mu_j,\beta_j)<\infty.
\end{align}
In particular, this yields $\mu_j=\sum_{i=1}^{M_j}\xi_{i,j}\delta_{x_{i,j}}\in\mathcal X_{\alpha_j}$ and $\beta_j\in\mathcal A_{\alpha_j}(\mu_j)$. By arguing as in the proof of Theorem~\ref{t:critical}(i), see~\eqref{eq:com4}--\eqref{eq:com5}, there is $C>0$ such that
\begin{align}\label{eq:Mj-bound}
\alpha_jM_j\le C\alpha_j\sum_{i=1}^{M_j}|\xi_{i,j}|\le C\alpha_j\sum_{i=1}^{M_j}|\xi_{i,j}|^2\le C\quad\text{for all $j\in\N$},
\end{align}
and $\Curl\beta=\mu$ in $\mathcal D'(\Omega;\R^2)$, as well as $\mu\in H^{-1}(\Omega;\R^2)$. We define
$$\Omega_j\coloneqq \bigcup_{i=1}^{M_j}B_{\rho_j}(x_{i,j})\subset\Omega\quad\text{for all $j\in\N$}.$$
By~\eqref{all convergences},~\eqref{eq:a-da-j-1}, and~\eqref{eq:Mj-bound} we derive that
$$|\Omega_j|=\pi M_j\rho_j^2\le C\frac{\rho_j^2}{\alpha_j}\to 0\quad\text{as $j\to\infty$}.$$
Therefore, we have
$$2\alpha_j(\mathcal I^{\alpha_j}_{\rho_j}\beta_j-S_j)\chi_{\Omega\setminus\Omega_j}\rightharpoonup \beta \quad\text{in $L^2(\Omega;\R^{2\times 2})$ as $j\to\infty$},$$
which yields
\begin{align}\label{eq:liminf-1}
\liminf_{j\to\infty}2\alpha_j^2\int_{\Omega\setminus\Omega_j}\C\mathcal I^{\alpha_j}_{\rho_j}\beta_j(x):\mathcal I^{\alpha_j}_{\rho_j}\beta_j(x)\,\de x\ge \frac{1}{2}\int_\Omega\C\beta(x):\beta(x)\,\de x.
\end{align}
Moreover, by~\eqref{eq:ad-psi},~\eqref{eq:varphi},~\eqref{eq:Mj-bound0}--\eqref{eq:Mj-bound} as well as~\eqref{eq:psi-prop} and Remark~\ref{important remark} we get 
\begin{align*}
2\alpha_j^2\int_{\Omega_j}\C\mathcal I^{\alpha_j}_{\rho_j}\beta_j(x):\mathcal I^{\alpha_j}_{\rho_j}\beta_j(x)\,\de x &=2\alpha_j^2\sum_{i=1}^{M_j}\int_{B_{\rho_j}(x_{i,j})} \hspace{-0.4cm}\C\mathcal I^{\alpha_j}_{\rho_j}\beta_j(x):\mathcal I^{\alpha_j}_{\rho_j}\beta_j(x)\,\de x \ge 4\alpha_j^2\sum_{i=1}^{M_j}\Psi(\xi_{i,j},\alpha_j,\rho_j)\\
&\ge 2\alpha_j\sum_{i=1}^{M_j}\psi(\xi_{i,j})-2\alpha_j\sum_{i=1}^{M_j}|\xi_{i,j}|^2\sup_{\xi\in\mathbb S^1}|2\alpha_j\Psi(\xi,\alpha_j,\rho_j)-\psi(\xi)|\\
&\ge 2\alpha_j\sum_{i=1}^{M_j}\varphi(\xi_{i,j})-C\sup_{\xi\in\mathbb S^1}|2\alpha_j\Psi(\xi,\alpha_j,\rho_j)-\psi(\xi)|\\
&=\int_\Omega \varphi(x) \, \de(2\alpha_j\mu_j) -C\sup_{\xi\in\mathbb S^1}|2\alpha_j\Psi(\xi,\alpha_j,\rho_j)-\psi(\xi)|.
\end{align*}
Since $\varphi$ is positively $1$-homogeneous and convex, see~\eqref{eq:varphi} and Remark~\ref{a new remark}(i), by Reshetnyak's lower semicontinuity theorem and Proposition~\ref{prop:cell-formula} we derive
\begin{align*}
\liminf_{j\to\infty}2\alpha_j^2\int_{\Omega_j}\C\mathcal I^{\alpha_j}_{\rho_j}\beta_j(x):\mathcal I^{\alpha_j}_{\rho_j}\beta_j(x)\,\de x\ge \liminf_{j\to\infty} \int_\Omega\varphi\left(\frac{\de(2\alpha_j\mu_j)}{\de|2\alpha_j\mu_j|}\right) \de|2\alpha_j\mu_j| \ge \int_\Omega\varphi\left(\frac{\de\mu}{\de|\mu|}\right)\,\de|\mu|.
\end{align*}
This combined with~\eqref{eq:liminf-1} concludes the proof of~\eqref{eq:gamma-liminf}.
\end{proof}

\begin{remark}\label{remark:forsub}
Due to relation~\eqref{eq:a-da-j-1}, the prefactor $2\alpha_j$ in~\eqref{eq:con-cri-1}--\eqref{eq:con-cri-2} can be replaced by both $1/N_j$ or $\sqrt{2\alpha_j}/\sqrt{N_j}$ without any change. In the same way, the prefactor $(2\alpha_j)^2$ in~\eqref{Falphacrit} can be replaced by $2\alpha_j / N_j$. By inspection of the previous proofs, we see that the arguments immediately imply Theorem~\ref{t:subcritical}(i),(ii), except for the property $\Curl\beta=0$ in $\mathcal D'(\Omega;\R^2)$.
\end{remark} 

To prove the $\Gamma$-limsup inequality, we need to approximate diffusive measures $\mu$ by  suitable sums of Dirac deltas. This is done in the following lemma, whose proof can be found in~\cite[Lemma~14]{GLP}.

\begin{lemma}\label{lem:mu-appr}
Let $\Omega$ be an open, bounded, simply connected set with Lipschitz boundary. Let $(N_j)_j\subset\N$ be such that $N_j\to \infty$ as $j\to\infty$. Let $\xi\coloneqq \sum_{i=1}^M\lambda_k\xi_k\in\R^2$ be such that $\lambda_k\ge 0$ and $\xi_k\in\mathbb S$. Let us define
\begin{equation}\label{eq:rj}
\Lambda\coloneqq \sum_{k=1}^M\lambda_k,\qquad r_j\coloneqq \frac{1}{2\sqrt{\Lambda N_j}}\quad\text{for all $j$}.
\end{equation}
There exists a sequence of measures $(\mu_j)_j\subset\mathcal M(\R^2;\R^2)$ such that 
\begin{align}\label{eq:muj-kj}
\mu_j= \sum_{k=1}^M\xi_k\mu_{k,j}\quad\text{for all $j$},\qquad \mu_{k,j}\coloneqq \sum_{i=1}^{M_{k,j}}\delta_{x_{i,k,j}}\quad\text{for all $k,j$},
\end{align}
with
$$
B_{r_j}(x_{i,k,j})\subset\Omega\quad\text{for all $i,k,j$},\qquad |x_{i,k,j}-x_{i',k',j}|\ge 2r_j\quad\text{if $(i',k')\neq (i,k)$, for all $j$},
$$
and satisfying
\begin{align}
&\frac{1}{N_j}|\mu_{k,j}|\xrightharpoonup{*} \lambda_k\chi_\Omega\,\de x\quad\text{in $\mathcal M(\R^2)$ as $j\to\infty$ for all $k=1,\ldots,M$},\label{eq:mukj}\\
&\frac{1}{N_j}\mu_j \xrightharpoonup{*}\xi\chi_\Omega\,\de x\quad\text{in $\mathcal M(\R^2;\R^2)$ as $j\to\infty$}.\label{eq:muj}
\end{align}
\end{lemma}

\begin{remark}\label{rem:mu-appr}
(i) For the sequel, it is convenient to use the following notation. Denoting the centers $(x_{i,k,j})_{i,k}$ by $(x_{i,j})_i$ with associated $(\xi_{i,j})_i \subset \lbrace \xi_1 ,\ldots, \xi_M \rbrace$ (see~\eqref{eq:muj-kj}), we can write 
$$
\mu_j=\sum_{k=1}^M\xi_k\left(\sum_{i=1}^{M_{k,j}}\delta_{x_{i,k,j}}\right)=\sum_{i=1}^{M_j}\xi_{i,j}\delta_{x_{i,j}},\qquad \text{where } \ M_j\coloneqq\sum_{k=1}^M M_{k,j}.
$$ 

(ii) The statement of~\cite[Lemma~14]{GLP} only says that
$$
\frac{1}{N_j}|\mu_{k,j}|\xrightharpoonup{*} \lambda_k\,\de x\quad\text{in $\mathcal M(\Omega)$ as $j\to\infty$ for all $k$},\qquad \frac{1}{N_j}\mu_j \xrightharpoonup{*}\xi\,\de x\quad\text{in $\mathcal M(\Omega;\R^2)$ as $j\to\infty$}.
$$
However, a careful inspection of the construction in its proof together with $|\partial\Omega|=0$ yield that actually~\eqref{eq:mukj}--\eqref{eq:muj} hold. 

(iii) Note that $\frac{\rho_j}{r_j}\to 0$ as $j\to\infty$ by~\eqref{all convergences}. Therefore, $\mu_j \in \mathcal X_{\alpha_j}$ for $j$ sufficiently large.

(iv) For later purposes, we note that there exists a constant $C>0$ (independent of $i$ and $j$) with
\begin{align}\label{eq inremark}
|\xi_{i,j}|\le C\quad\text{for all $i,j$},\qquad\frac{M_j}{N_j}\le C\quad\text{for all $j$}.
\end{align}
The latter follows from the boundedness of $\mu_j$ (see~\eqref{eq:muj}) along with the fact that $|\xi_{i,j}| \ge c>0$ for all $i,j$, see~\eqref{eq:com4}.
\end{remark}

We can finally prove the $\Gamma$-limsup inequality.

\begin{proof}[Proof of Theorem~\ref{t:critical}(iii)]
The proof is divided into three steps. We first assume that $\mu$ has the form $\mu = \xi\chi_A\,\de x$ for some $A \subset \Omega$ (Step 1). Then, we consider piecewise constant $\mu$ with respect to a partition of $\Omega$ (Step 2), and finally pass to the general case (Step 3). Step 1 will be further subdivided into several steps. 

{\bf Step~1.} Let $\mu\coloneqq \xi\chi_A\,\de x$, where $\xi\in\R^2$ and $A\subset\Omega$ is an open, bounded, simply connected set with Lipschitz boundary. Let $\beta\in L^2(\Omega;\R^{2\times 2})$ be such that $\Curl\beta=\mu$ in $\mathcal D'(\Omega;\R^2)$. Recalling~\eqref{eq:varphi}, we write $\xi=\sum_{k=1}^M\lambda_k\xi_k$, where $\xi_k\in\mathbb S$, $\lambda_k\ge 0$, such that 
\begin{align}\label{phiuse}
\varphi(\xi)=\sum_{k=1}^M\lambda_k\psi(\xi_k).
\end{align}
{\bf Step~1.1:} \emph{Regularization of $\beta$.} First, we regularize $\beta$ and extend it outside $\Omega$ in such a way that the condition on the curl is preserved. To this aim we fix $R>0$ such that $\Omega\subset\subset B_R$, and we consider the function $\tilde w\in H^1_0(B_R;\R^2)$ which solves
\begin{equation*}
\begin{cases}
\Delta \tilde w=\mu &\text{in $B_R$},\\
\tilde w=0 &\text{on $\partial B_R$}.
\end{cases}
\end{equation*}
Since $\mu = \xi\chi_A\in L^\infty(B_R;\R^2)$, we derive that $\tilde w\in W^{2,p}(B_R;\R^2)$ for all $p\in [1,\infty)$ by elliptic regularity. In particular, $\tilde w\in C^1(\overline B_R;\R^2)$ by Sobolev embedding theorems. We define $\tilde\beta\coloneqq \nabla \tilde wJ^T\in C(\overline B_R;\R^{2\times 2})$, and by~\eqref{curlicurl} we have
$$
\Curl(\beta-\tilde\beta)=\mu-\Delta \tilde w=\mu-\mu=0\quad\text{in $\mathcal D'(\Omega;\R^2)$}.
$$
Hence, there exists a function $v\in H^1(\Omega;\R^2)$ such that 
$$
\beta-\tilde\beta=\nabla v\quad\text{in $\Omega$}.
$$
Let $(\Phi_j)_j\subset C^\infty_c(\R^2;\R^2)$ be such that 
\begin{align}\label{tildebetaj0}
{\nabla\Phi_j\to\nabla v\quad\text{in $L^2(\Omega;\R^{2\times 2})$ as $j\to\infty$}.}
\end{align}
We consider the linear operator $P_{\rho_j}^{1-{\alpha_j}}\colon\mathcal S(\R^2;\R^2)\to \mathcal S(\R^2;\R^2)$ given by Proposition~\ref{prop:Psd}, and we define the regularized function 
\begin{align}\label{tildebetaj}
 \beta^{\rm reg}_j\coloneqq \tilde \beta+\nabla (P_{\rho_j}^{1-{\alpha_j}}\Phi_j)\quad\text{in $B_R$ for all $j\in\N$}.
\end{align}
By construction $\beta^{\rm reg}_j\in C({\overline B_R};\R^{2\times 2})$ and $\Curl\beta^{\rm reg}_j=\mu$ in $\mathcal D'(B_R;\R^2)$. 

{\bf Step~1.2:} \emph{Construction of $(\mu_j,\beta_j)$.} 
We now construct a sequence $(\mu_j,\beta_j)_j\subset\mathcal M(\R^2;\R^2)\times L^1(B_R;\R^{2\times 2})$ which satisfies $(\mu_j,\beta_j)\in\mathcal X_{\alpha_j}\times \mathcal A_{\alpha_j}(\mu_j)$ for all $j\in\N$. Let $(\mu_j)_j\subset\mathcal M(\R^2;\R^2)$ be the sequence of measures given by Lemma~\ref{lem:mu-appr} with $N_j = \frac{1}{2\alpha_j}$ for $j\in\N$ (cf.~\eqref{eq:a-da-j-1}). Using the notation of Remark~\ref{rem:mu-appr}(i), we have 
\begin{equation}\label{eq:muj-mukj}
\mu_j= \sum_{k=1}^M\xi_k\mu_{k,j}=\sum_{i=1}^{M_j}\xi_{i,j}\delta_{x_{i,j}}\quad\text{for all $j$}, 
\end{equation}
with
\begin{align}\label{in A}
B_{r_j}(x_{i,j})\subset A\quad\text{for all $i,j$},\qquad |x_{i,j}-x_{i',j}|\ge 2r_j\quad\text{if $i\neq i'$, for all $j$},
\end{align}
and, as $j\to\infty$,
\begin{align}\label{mumumu}
2\alpha_j|\mu_{k,j}|\xrightharpoonup{*} \lambda_k\chi_A\,\de x\quad\text{in $\mathcal M(\R^2)$ for all $k$},\qquad 2\alpha_j\mu_j \xrightharpoonup{*}\mu = \xi\chi_A\,\de x \quad\text{in $\mathcal M(\R^2;\R^2)$ }.
\end{align}
Next, we will modify ${\beta}^{\rm reg}_j$ to obtain a function whose curl is $\mu_j$. To this end, let $\phi\in C_c^\infty (B_1)$ be such that $0\le \phi\le 1$ in $\R^2$ and $\phi=1$ in $B_{\frac{1}{2}}$. Recalling the function $\eta_\xi$ for $\xi \in \R^2$ given in~\eqref{eq:eta-sys} and using the fact that $\rho_j < r_j$ for $j$ large enough (see Remark~\ref{rem:mu-appr}(iii)), we define
\begin{equation}\label{eq:zetaij-cr}
\zeta_{i,j}(x)\coloneqq \eta_{\xi_{i,j}}(x-x_{i,j})\phi\left(\frac{x-x_{i,j}}{r_j-\rho_j}\right)\quad\text{for all $x\in\R^2\setminus\{x_{i,j}\}$, $i\in\{1,\dots,M_j\}$, and $j\in\N$},
\end{equation}
and
\begin{equation}\label{eq:zetaj-cr}
\zeta_j(x)\coloneqq\sum_{i=1}^{M_j}\zeta_{i,j}(x)\quad\text{for all $x\in\R^2\setminus\{x_{1,j},\dots,x_{M_j,j}\}$ and $j\in\N$}.
\end{equation}
We have that $\zeta_j\in L^p_{\loc}(\R^2;\R^{2\times 2})$ for all $p\in [1,2)$, see~\eqref{eq:eta-decay}--\eqref{eq:eta-est}, and $(\supp\zeta_j)_{\rho_j}\subset A$ by~\eqref{in A}. Moreover, 
\begin{equation}\label{eq:nuj-cr}
\Curl\zeta_j=\mu_j+\nu_j\coloneqq\sum_{i=1}^{M_j}\xi_{i,j}\delta_{x_{i,j}}+\sum_{i=1}^{M_j}\eta_{\xi_{i,j}}(\,\cdot\,-x_{i,j})J\nabla\phi\left(\frac{\,\cdot\,-x_{i,j}}{r_j-\rho_j}\right)\frac{1}{r_j-\rho_j}\quad\text{in $\mathcal D'(\R^2;\R^2)$}.
\end{equation}
It turns out that the curl of the function $\frac{\beta^{\rm reg}_j}{2\alpha_j}+\zeta_j$ is $\mu_j$, up to an asymptotically vanishing term. As a final step of the construction, we remove this remainder term: let $w^{\rm rem}_j\in H^1_0(B_R;\R^2)$ be the solution to 
\begin{align}\label{hatw}
\begin{cases}
\Delta w^{\rm rem}_j=-\frac{\mu}{2\alpha_j}-\nu_j &\text{in $B_R$},\\
 w^{\rm rem}_j=0 &\text{on $\partial B_R$},
\end{cases}
\end{align}
and define the function $\beta^{\rm rem}_j\coloneqq \nabla w^{\rm rem}_jJ^T\in L^2(B_R;\R^{2\times 2})$. We set
\begin{equation}\label{betabtetabta}
\beta_j\coloneqq \frac{\beta^{\rm reg}_j}{2\alpha_j}+\zeta_j+\beta^{\rm rem}_j\quad\text{in $B_R$}.
\end{equation}
Recalling that $\Curl\beta^{\rm reg}_j=\mu$ in $\mathcal D'(B_R;\R^2)$, by construction (use~\eqref{curlicurl}) we get 
\begin{align}\label{curlrechnung}
\Curl\beta_j=\frac{\mu}{2\alpha_j}+\mu_j+\nu_j-\frac{\mu}{2\alpha_j}-\nu_j=\mu_j\quad\text{in $\mathcal D'(B_R;\R^2)$},
\end{align}
which yields $\beta_j\in\mathcal A_{\alpha_j}(\mu_j)\quad\text{for all $j\in\N$}.$ 

{\bf Step~1.3:} \emph{Convergence of recovery sequence.} 
We now show that $(\mu_j,\beta_j)_j$ satisfies~\eqref{eq:con-cri-1}--\eqref{eq:con-cri-2}. To this aim, in view of~\eqref{mumumu}, we just need to prove that
\begin{equation}\label{eq:aj-bj-con}
2\alpha_j\mathcal I^{\alpha_j}_{\rho_j}\beta_j \rightharpoonup \beta\quad\text{in $L^2(\Omega;\R^{2\times 2})$ as $j\to\infty$}.
\end{equation}
Notice that $\supp(\mathcal I^{\alpha_j}_{\rho_j}\zeta_{i,j})\subset B_{r_j}(x_{i,j})$ by~\eqref{eq:Qsd-p} and~\eqref{eq:zetaij-cr}. By Remark~\ref{rem:gamma-a}, Proposition~\ref{prop:Riesz-comp}, and~\eqref{eq:eta-decay}--\eqref{eq:eta-est} we have
\begin{equation}\label{eq:zetaij-est}
|\mathcal I^{\alpha_j}_{\rho_j}\zeta_{i,j}(x)|\le \frac{K|\xi_{i,j}|}{\gamma_{{\alpha_j}}}\int_{\R^2}\frac{1}{|y-x_{i,j}|}\frac{1}{|x-y|^{2-{\alpha_j}}}\,\de y=\frac{2\pi K|\xi_{i,j}|}{\gamma_{1+{\alpha_j}}}\frac{1}{|x-x_{i,j}|^{1-{\alpha_j}}}.
\end{equation}
Thus, by~\eqref{eq inremark} in Remark~\ref{rem:mu-appr} and~\eqref{eq:a-da-j-1} we can find a constant $C>0$ (independent of $j$) such that, for all $j\in\N$,
\begin{align}\label{eq:zetaij-est-new}
\int_\Omega|2\alpha_j\mathcal I^{\alpha_j}_{\rho_j}\zeta_j(x)|^2\,\de x=4\alpha_j^2\sum_{i=1}^{M_j}\int_{B_{r_j}(x_{i,j})}|\mathcal I^{\alpha_j}_{\rho_j}\zeta_{i,j}(x)|^2\,\de x\le Cr_j^{2\alpha_j}\alpha_jM_j \le C\alpha_jN_j \le C.
\end{align}
Moreover, we have
\begin{align*}
\int_\Omega|2\alpha_j\mathcal I^{\alpha_j}_{\rho_j}\zeta_j(x)|\,\de x=2\alpha_j\sum_{i=1}^{M_j}\int_{B_{r_j}(x_{i,j})}|\mathcal I^{\alpha_j}_{\rho_j}\zeta_{i,j}(x)|\,\de x\le C\alpha_j M_j r_j^{1+{\alpha_j}}\le C r_j\quad\text{for all $j\in\N$}.
\end{align*}
Hence, we derive that 
\begin{align}\label{eq:ls-con-1}
2\alpha_j\mathcal I^{\alpha_j}_{\rho_j}\zeta_j \rightharpoonup 0\quad\text{in $L^2(\Omega;\R^{2\times 2})$ as $j\to\infty$}.
\end{align}
As $\tilde\beta\in C(\overline B_R;\R^{2\times 2})$, by~\eqref{tildebetaj0}--\eqref{tildebetaj}, Proposition~\ref{prop:Psd}, and Corollary~\ref{coro:Iad-conv} we infer that 
\begin{align}\label{eq:ls-con-2}
\mathcal I^{\alpha_j}_{\rho_j}\beta^{\rm reg}_j=\mathcal I^{\alpha_j}_{\rho_j}\tilde \beta+\nabla\Phi_j\to \tilde\beta+\nabla v=\beta\quad\text{in $L^2(\Omega;\R^{2\times 2})$ as $j\to\infty$}.
\end{align}
Recalling the definition of $\nu_j$ in~\eqref{eq:nuj-cr}, we claim that 
\begin{equation}\label{eq:ajnuj-lim}
2\alpha_j\nu_j+\mu\xrightharpoonup{*}0\quad\text{in $L^\infty(B_R;\R^2)$ as $j\to\infty$}.
\end{equation}
First, since $\zeta_{i,j}$ have disjoint supports, there exists a constant $C>0$ (independent of $j$) such that
\begin{align}\label{nununu}
\|2\alpha_j\nu_j\|_{L^\infty(B_R)}\le C\frac{2\alpha_j}{(r_j-\rho_j)^2}\quad\text{for all $j\in\N$},
\end{align}
where we used~\eqref{eq:eta-decay}--\eqref{eq:eta-est}. Moreover, by Remark~\ref{rem:mu-appr}(iii),~\eqref{eq:a-da-j-1}, and~\eqref{eq:rj} we have
$${
\lim_{j\to\infty}\frac{2\alpha_j}{(r_j-\rho_j)^2}=4\Lambda. }
$$
 Hence, $(2\alpha_j\nu_j+\mu)_j\subset L^\infty(B_R;\R^2)$ is uniformly bounded. Therefore, to prove~\eqref{eq:ajnuj-lim}, it is enough to observe that for all $\Phi\in C_c^\infty (B_R;\R^2)$ we have
$$\langle 2\alpha_j\nu_j+\mu,\Phi\rangle_{\mathcal D'(B_R)}\to 0\quad\text{as $j\to\infty$}.$$
Since $2\alpha_j\mu_j\xrightharpoonup{*}\mu$ in $\mathcal M(\R^2;\R^2)$, we just need to prove that 
$2\alpha_j\langle \nu_j+\mu_j,\Phi\rangle_{\mathcal D'(B_R)}\to 0$ as $j\to\infty$ for all $\Phi\in C_c^\infty (B_R;\R^2)$. We recall $ \nu_j+\mu_j = \Curl \zeta_j$ and, as $j \to \infty$, using~\eqref{curlicurl} and~\eqref{eq:eta-decay}--\eqref{eq:eta-est}, we estimate 
\begin{align*}
2\alpha_j\left|\langle \Curl \zeta_j,\Phi\rangle_{\mathcal D'(B_R)}\right|&=2\alpha_j\left|\int_{B_R}\zeta_j(x)J\colon\nabla\Phi(x)\,\de x\right|\\
&\le C\alpha_j \sum_{i=1}^{M_j} \int_{B_{r_j-\rho_j}}\left|\eta_{\xi_{i,j}}(y)\right|\,\de y\le C\alpha_jM_j(r_j-\rho_j)\le C(r_j-\rho_j)\to 0,
\end{align*}
where we also used~\eqref{eq:a-da-j-1} and~\eqref{eq inremark}. 
This gives~\eqref{eq:ajnuj-lim}. Therefore, by the compact embedding of $L^2(B_R;\R^2)$ into $H^{-1}(B_R;\R^2)$ we conclude that
$2\alpha_j\nu_j+\mu\to 0$ in $H^{-1}(B_R;\R^2)$ as $j\to\infty$. For $ w^{\rm rem}_j$ introduced in~\eqref{hatw}, this implies that 
$$2\alpha_j w^{\rm rem}_j\to 0\quad\text{in $H^1_0(B_R;\R^2)$ as $j\to\infty$}.$$
Recall the definition $\beta^{\rm rem}_j = \nabla w_j^{\rm rem}J^T$. By Remark~\ref{rem:gamma-a},~\eqref{eq:Qsd-est}, and Young's convolution inequality, as $j\to\infty$, we have 
\begin{equation}\label{eq:ls-con-3}
\|2\alpha_j\mathcal I^{\alpha_j}_{\rho_j}\beta^{\rm rem}_j\|_{L^2(\Omega)}\le \|Q_{\rho_j}^{1-{\alpha_j}}\|_{L^1(\R^2)}\|2\alpha_j\beta^{\rm rem}_j\|_{L^2(B_R)}\le\frac{2\pi\rho_j^{{\alpha_j}}}{\alpha_j\gamma_{{\alpha_j}}}\|2\alpha_j\nabla w^{\rm rem}_j\|_{L^2(B_R)}\to 0.
\end{equation}
This together with~\eqref{betabtetabta},~\eqref{eq:ls-con-1}, and~\eqref{eq:ls-con-2} gives~\eqref{eq:aj-bj-con}. 

{\bf Step~1.4:} \emph{Convergence of energies.} It remains to prove~\eqref{eq:gamma-limsup}. As in the proof of the $\Gamma$-liminf inequality, we define
$$
\Omega_j\coloneqq \bigcup_{i=1}^{M_j}B_{\rho_j}(x_{i,j})\subset\Omega\quad\text{for all $j\in\N$},
$$
and by~\eqref{all convergences},~\eqref{eq:a-da-j-1}, and~\eqref{eq inremark} we derive that
\begin{equation}\label{eq:Omegaj-con}
|\Omega_j|=\pi M_j\rho_j^2\le C N_j\rho_j^2 \to 0\quad\text{as $j\to\infty$}.
\end{equation}
We split the energy as 
\begin{align*}
\mathcal F_{\alpha_j}(\mu_j,\beta_j)&=2\alpha_j^2\int_{\Omega\setminus \Omega_j}\C\mathcal I^{\alpha_j}_{\rho_j}\beta_j(x):\mathcal I^{\alpha_j}_{\rho_j}\beta_j(x)\,\de x+2\alpha_j^2\int_{\Omega_j}\C\mathcal I^{\alpha_j}_{\rho_j}\beta_j(x):\mathcal I^{\alpha_j}_{\rho_j}\beta_j(x)\,\de x,
\end{align*}
and we study the two terms separately. 

By~\eqref{eq:ls-con-2},~\eqref{eq:ls-con-3}, and~\eqref{eq:Omegaj-con} we derive that
\begin{align*}
{
\mathcal I^{\alpha_j}_{\rho_j}\beta^{\rm reg}_j\chi_{\Omega_j}\to 0 \ \ \text{and} \ \ 2\alpha_j\mathcal I^{\alpha_j}_{\rho_j}\beta^{\rm rem}_j\chi_{\Omega_j}\to 0 \quad\text{in $L^2(\Omega;\R^{2\times 2})$ as $j\to\infty$},}
\end{align*}
which by~\eqref{betabtetabta} gives
\begin{align}\label{LLL1}
\limsup_{j\to\infty}2\alpha_j^2\int_{\Omega_j}\C\mathcal I^{\alpha_j}_{\rho_j}\beta_j(x):\mathcal I^{\alpha_j}_{\rho_j}\beta_j(x)\,\de x=\limsup_{j\to\infty} 2\alpha_j^2\int_{\Omega_j}\C\mathcal I^{\alpha_j}_{\rho_j}\zeta_j(x):\mathcal I^{\alpha_j}_{\rho_j}\zeta_j(x)\,\de x.
\end{align}
Since $\zeta_{i,j}=\eta_{\xi_{i,j}}(\,\cdot\,-x_{i,j})$ in $B_{(r_j-\rho_j)/2}(x_{i,j})$ and $\rho_j/r_j \to 0$, see Remark~\ref{rem:mu-appr}(iii), we infer that, for $j$ sufficiently large, 
$$
\mathcal I^{\alpha_j}_{\rho_j}\zeta_{i,j}=\mathcal I^{\alpha_j}_{\rho_j}\eta_{\xi_{i,j}}(\,\cdot\,-x_{i,j})\quad\text{in $B_{\rho_j}(x_{i,j})$}.
$$
Together with Lemma~\ref{lem:lower-est},~\eqref{eq:a-da-j-1},~\eqref{eq:unif-conv2},~\eqref{eq inremark},~\eqref{phiuse}, and~\eqref{eq:muj-mukj}--\eqref{mumumu}, we conclude that
\begin{align}\label{never-remove!}
&\limsup_{j\to\infty} 2\alpha_j^2\int_{\Omega_j}\C\mathcal I^{\alpha_j}_{\rho_j}\zeta_j(x):\mathcal I^{\alpha_j}_{\rho_j}\zeta_j(x)\,\de x\notag\\
&=\limsup_{j\to\infty} 2\alpha_j^2\sum_{i=1}^{M_j}\int_{B_{\rho_j}}\C\mathcal I^{\alpha_j}_{\rho_j}\eta_{\xi_{i,j}}(y):\mathcal I^{\alpha_j}_{\rho_j}\eta_{\xi_{i,j}}(y)\,\de y\notag\\
&\le \limsup_{j\to\infty}\left(4\alpha_j^2\sum_{i=1}^{M_j}\Psi(\xi_{i,j},\alpha_j,\rho_j)+4C\alpha_j^2\sum_{i=1}^{M_j}|\xi_{i,j}|^2\right)\\
&=\lim_{j\to\infty}\sum_{k=1}^M2\alpha_j|\mu_{k,j}|(\Omega)2\alpha_j\Psi(\xi_k,\alpha_j,\rho_j)=\sum_{k=1}^M\lambda_k|A| \psi(\xi_k)=|A|\varphi(\xi)=\int_\Omega\varphi\left(\frac{\de \mu}{\de|\mu|}\right)\de|\mu|. \notag
\end{align}
Here, we also used that $\sum_{i=1}^{M_j}\Psi(\xi_{i,j},\alpha_j,\rho_j) = \sum_{k=1}^M |\mu_{k,j}|(\Omega)\Psi (\xi_k,\alpha_j,\rho_j)$ by Remark~\ref{rem:mu-appr}(i). Next, by~\eqref{all convergences},~\eqref{eq:a-da-j-1},~\eqref{eq inremark}, and~\eqref{eq:zetaij-est}, as $j\to\infty$ we have 
\begin{align*}
\int_{\Omega \setminus \Omega_j}|2\alpha_j\mathcal I^{\alpha_j}_{\rho_j}\zeta_j(x)|^2\,\de x&=4\alpha_j^2\sum_{i=1}^{M_j}\int_{B_{r_j}(x_{i,j})\setminus B_{\rho_j}(x_{i,j})}|\mathcal I^{\alpha_j}_{\rho_j}\zeta_{i,j}(x)|^2\,\de x\\
&\le \frac{8\pi^3 K^2}{\gamma_{1+{\alpha_j}}^2} 2 \alpha_j\sum_{i=1}^{M_j}|\xi_{i,j}|^2(r_j^{2\alpha_j}-\rho_j^{2\alpha_j})\to 0,
\end{align*}
where in the last step we used $\rho_j^{\alpha_j} \to 1$ as $j \to \infty$. Hence, together with~\eqref{betabtetabta},~\eqref{eq:ls-con-2},~\eqref{eq:ls-con-3}, and~\eqref{eq:Omegaj-con} we derive that
\begin{align*}
&2\alpha_j\mathcal I^{\alpha_j}_{\rho_j}\beta_j\chi_{\Omega\setminus\Omega_j}\to \beta\quad\text{in $L^2(\Omega;\R^{2\times 2})$ as $j\to\infty$},
\end{align*}
which gives
\begin{align}\label{LLL2}
\lim_{j\to\infty}2\alpha_j^2\int_{\Omega\setminus \Omega_j}\C\mathcal I^{\alpha_j}_{\rho_j}\beta_j(x):\mathcal I^{\alpha_j}_{\rho_j}\beta_j(x)\,\de x=\frac{1}{2}\int_\Omega\C\beta(x):\beta(x)\,\de x.
\end{align}
Combining~\eqref{LLL1}--\eqref{LLL2} yields~\eqref{eq:gamma-limsup}, and concludes the proof in the case $\mu=\xi\chi_A\,\de x$ and $\beta\in L^2(\Omega;\R^{2\times 2})$.

{\bf Step~2.} Let $\beta\in L^2(\Omega;\R^{2\times 2})$ be such that 
$${\Curl\beta=\mu\coloneqq \sum_{l=1}^L\mu^l,\qquad \text{ where } \ \mu^l\coloneqq \xi^l\chi_{A^l} \,\de x \quad\text{for all $l\in\{1,\dots,L\}$},}$$
where $(\xi^l)_l\in\R^2$, $(A^l)_l$ are open, bounded, simply connected sets with Lipschitz boundary, and $\{A^l\}_{l=1}^L$ is a partition of $\Omega$. By arguing as in Step~1, see in particular~\eqref{tildebetaj} and~\eqref{eq:ls-con-2}, we can find $(\beta^{\rm reg}_j)_j\subset C(\overline B_R;\R^{2\times 2})$ such that 
\begin{align*}
\Curl\beta^{\rm reg}_j=\mu \quad\text{in $\mathcal D'(B_R;\R^2)$},\qquad \mathcal I^{\alpha_j}_{\rho_j}\beta^{\rm reg}_j\to \beta\quad\text{in $L^2(\Omega;\R^{2\times 2})$ as $j\to\infty$}.
\end{align*}
Moreover, for all $l\in\{1,\dots,L\}$ there exist $(\mu_j^l)_j\subset \mathcal M(\R^2;\R^2)$, $(\nu_j^l)_j\subset L^\infty(\R^2;\R^2)$, $(\zeta_j^l)_j\subset L^1_{\loc}(\R^2;\R^{2\times 2})$, and $((\beta^{\rm rem}_j)^l)_j\subset L^2(B_R;\R^{2\times 2})$ such that 
\begin{align}\label{eq: zeta supp}
&(\supp\zeta_j^l)_{\rho_j}\subset A^l,\qquad\text{for all $j\in\N$}
\end{align}
and, see~\eqref{eq:nuj-cr}--\eqref{hatw},
\begin{align*}
&\Curl\zeta_j^l=\mu_j^l+\nu_j^l\quad\text{in $\mathcal D'(\R^2;\R^2)$}, \qquad \Curl(\beta^{\rm rem}_j)^l=-\frac{\xi^l\chi_{A^l}\,\de x}{2\alpha_j}-\nu_j^l\quad\text{in $\mathcal D'(B_R;\R^2)$ for all $j\in\N$}
\end{align*}
as well as (see~\eqref{mumumu},~\eqref{eq:ls-con-1}, and~\eqref{eq:ls-con-3}) 
\begin{align*}
2\alpha_j\mu^l_j \xrightharpoonup{*} \xi^l\chi_A^l\,\de x \quad\text{in $\mathcal M(\R^2;\R^2)$}, \qquad 2\alpha_j\mathcal I^{\alpha_j}_{\rho_j} \big( \zeta_j^l + (\beta^{\rm rem}_j)^l\big) \rightharpoonup 0\quad \text{in $L^2(\Omega;\R^{2\times 2})$ as $j\to\infty$}.
\end{align*}
We define
$$\mu_j\coloneqq \sum_{l=1}^L \mu_j^l\qquad\text{ and } \qquad\beta_j\coloneqq \frac{\beta^{\rm reg}_j}{2\alpha_j}+\sum_{l=1}^L(\zeta_j^l+(\beta^{\rm rem}_j)^l)\quad\text{for all $j\in\N$}.$$
Then, we have that $(\mu_j,\beta_j)\in\mathcal X_{\alpha_j}\times\mathcal A_{\alpha_j}(\mu_j)$ for all $j\in\N$ and
$$2\alpha_j\mu_j\xrightharpoonup{*}\mu\quad\text{in $\mathcal M(\Omega;\R^2)$ as $j\to\infty$},\qquad 2\alpha_j\mathcal I^{\alpha_j}_{\rho_j}\beta_j\rightharpoonup\beta\quad\text{in $L^2(\Omega;\R^{2\times 2})$ as $j\to\infty$}.$$
Moreover, as in Step 1, in view of~\eqref{eq: zeta supp}, for all $l\in\{1,\dots,L\}$ we derive
$$\limsup_{j\to\infty}2\alpha_j^2\int_{A^l}\C\mathcal I^{\alpha_j}_{\rho_j}\beta_j(x):\mathcal I^{\alpha_j}_{\rho_j}\beta_j(x)\,\de x\le \frac{1}{2}\int_{A^l}\C\beta(x):\beta(x)\,\de x+\int_{A^l}\varphi\left(\frac{\de\mu}{\de|\mu|}\right)\,\de|\mu|.$$
Therefore, since $|\Omega\setminus \bigcup_{l=1}^L A^l|=0$, we conclude
$$\limsup_{j\to\infty}2\alpha_j\int_\Omega\C\mathcal I^{\alpha_j}_{\rho_j}\beta_j(x):\mathcal I^{\alpha_j}_{\rho_j}\beta_j(x)\,\de x\le \frac{1}{2}\int_\Omega\C\beta(x):\beta(x)\,\de x+\int_\Omega\varphi\left(\frac{\de\mu}{\de|\mu|}\right)\,\de|\mu|.$$

{\bf Step~3.} Let $\mu\in\mathcal M(\Omega;\R^2)\cap H^{-1}(\Omega;\R^2)$ and $\beta\in L^2(\Omega;\R^{2\times 2})$ be such that $\Curl\beta=\mu$ in $\mathcal D'(\Omega;\R^2)$. As shown in~\cite[Proof of Theorem~12, Step 3 of $\Gamma$-limsup]{GLP}, we can find a sequence $(\mu^n,\beta^n)_n\subset (\mathcal M(\Omega;\R^2)\cap H^{-1}(\Omega;\R^2))\times L^2(\Omega;\R^{2\times 2})$ such that 
\begin{align}
&\mu^n\xrightharpoonup{*} \mu\quad\text{in $\mathcal M(\Omega;\R^2)$ as $n\to\infty$},& &\beta^n\to \beta\quad\text{in $L^2(\Omega;\R^{2\times 2})$ as $n\to\infty$},\label{eq:mun-appr1}\\
&|\mu^n|(\Omega)\to |\mu|(\Omega)\quad\text{as $n\to\infty$},& &\Curl\beta^n=\mu^n\quad\text{in $\mathcal D'(\Omega;\R^2)$ for all $n\in\N$},\label{eq:mun-appr2} 
\end{align}
where $\mu^n$ is as in Step~2 for all $n\in\N$. In view of~\eqref{eq:F-limit} and~\eqref{eq:mun-appr1}--\eqref{eq:mun-appr2}, by Reshetnyak's continuity theorem it follows that 
\begin{equation*}
\lim_{n\to\infty}\mathcal F(\mu^n,\beta^n)=\mathcal F(\mu,\beta).
\end{equation*}
By Step~2, for all $n\in\N$ there exists a sequence $(\mu_j^n,\beta_j^n)_j$ with $(\mu_j^n,\beta_j^n)\in \mathcal X_{\alpha_j}\times \mathcal A_{\alpha_j}(\mu_j^n)$ for all $j\in\N$ such that, as $j \to \infty$,
\begin{align*}
&2\alpha_j\mu_j^n\xrightharpoonup{*} \mu^n\ \text{in $\mathcal M(\Omega;\R^2)$},& &2\alpha_j\mathcal I^{\alpha_j}_{\rho_j}\beta_j^n\rightharpoonup \beta^n\ \text{in $L^2(\Omega;\R^{2\times 2})$}, & & \limsup_{j\to\infty}\mathcal F_{\alpha_j}(\mu_j^n,\beta_j^n)\le \mathcal F(\mu^n,\beta^n).
\end{align*}
Thus, to obtain~\eqref{eq:gamma-limsup} it is enough to use a standard diagonal argument.
\end{proof}

\section{The subcritical and supercritical regimes}\label{s:proofs2}

This section is devoted to the proofs of the $\Gamma$-limit of $\mathcal F_\alpha$ in the subcritical and supercritical regime. Since the proofs are similar to those of the critical regime, we only highlight the essential differences.

\subsection{The subcritical regime}

In the subcritical regime $(N_\alpha\ll \frac{1}{\alpha}$ as $\alpha\to 0$), we fix three sequences $(\alpha_j)_j$, $(\rho_j)_j$, and $(N_j)_j$ satisfying~\eqref{all convergences} and~\eqref{eq:a-da-j-4}. We recall the functionals defined in~\eqref{Falphasub} and~\eqref{subcritlimit}. 

\begin{proof}[Proof of Theorem~\ref{t:subcritical}]
{\bf Compactness and $\Gamma$-liminf inequality.} The proofs are similar to the ones of Theorem~\ref{t:critical}(i),(ii), see Remark~\ref{remark:forsub}. We only need to check that the convergences in~\eqref{eq:con-sub-1}--\eqref{eq:con-sub-2} imply $\Curl\beta=0$ in $\mathcal D'(\Omega;\R^2)$. To this end, we fix $\Phi\in C_c^\infty(\Omega;\R^2)$. Then, by~\eqref{eq:con-sub-2}, 
\begin{align*}
\langle \Curl\beta,\Phi\rangle_{\mathcal D'(\Omega)}&=\lim_{j\to\infty}\frac{\sqrt{2\alpha_j}}{\sqrt{N_j}}\langle \Curl (\mathcal I^{\alpha_j}_{\rho_j}\beta_j-S_j),\Phi\rangle_{\mathcal D'(\Omega)}=\lim_{j\to\infty}\frac{\sqrt{2\alpha_j}}{\sqrt{N_j}}\langle \Curl \mathcal I^{\alpha_j}_{\rho_j} \beta_j,\Phi\rangle_{\mathcal D'(\Omega)},
\end{align*} 
and by Lemma~\ref{lem:curl-Iadxi}, see particularly~\eqref{lastli}, we have
\begin{align*}
\frac{\sqrt{2\alpha_j}}{\sqrt{N_j}}\langle \Curl \mathcal I^{\alpha_j}_{\rho_j} \beta_j,\Phi\rangle_{\mathcal D'(\Omega)}=\sqrt{2\alpha_jN_j}\int_{\R^2}\mathcal I^{\alpha_j}_{\rho_j}\Phi(x)\cdot\de\frac{1}{N_j}\mu_j.
\end{align*}
For $j$ sufficiently large, we derive that $\mathcal I^{\alpha_j}_{\rho_j}\Phi\in C_c^\infty(\Omega;\R^2)$ and $\mathcal I^{\alpha_j}_{\rho_j}\Phi\to \Phi$ in $C(\overline\Omega;\R^2)$ as $j\to\infty$, see Corollary~\ref{coro:Iad-conv}. Therefore, by~\eqref{eq:a-da-j-4} and~\eqref{eq:con-sub-1} 
\begin{align*}
\langle \Curl\beta,\Phi\rangle_{\mathcal D'(\Omega)}&=\lim_{j\to\infty}\sqrt{2\alpha_jN_j}\int_\Omega\mathcal I^{\alpha_j}_{\rho_j}\Phi(x)\cdot\de\frac{1}{N_j}\mu_j=\lim_{j\to\infty}\sqrt{2\alpha_jN_j}\int_\Omega\Phi(x)\cdot\,\de \mu=0,
\end{align*}
which gives $\Curl\beta=0$ in $\mathcal D'(\Omega;\R^2)$. 

{\bf $\Gamma$-limsup inequality.} As in the proof of Theorem~\ref{t:critical}(iii), we proceed in three steps. We briefly explain only the first one, as the other two are analogous to the critical regime with the only difference that in Step~3 we approximate only the measure $\mu$, given that $\beta$ is independent of $\mu$.

Let $\beta\in L^2(\Omega;\R^{2\times 2})$ with $\Curl\beta=0$ in $\mathcal D'(\Omega;\R^2)$, and let $\mu\coloneqq \xi\chi_A\,\de x$, where $\xi\in\R^2$ and $A\subset\Omega$ is an open, bounded, simply connected set with Lipschitz boundary. We write $\xi=\sum_{k=1}^M\lambda_k\xi_k$, where $\xi_k\in\mathbb S$ and $\lambda_k\ge 0$, such that $\varphi(\xi)=\sum_{k=1}^M\lambda_k\psi(\xi_k)$, see~\eqref{eq:varphi}. 

Since $\Curl\beta=0$ in $\mathcal D'(\Omega;\R^2)$, there exists $v\in H^1(\Omega;\R^2)$ such that 
$$
\beta=\nabla v\quad\text{in $\Omega$}.
$$
Let $(\Phi_j)_j\subset C^\infty_c(\R^2;\R^2)$ be such that $\nabla\Phi_j\to\nabla v$ in $L^2(\Omega;\R^2)$ as $j\to\infty$, and consider the linear operator $P_{\rho_j}^{1-{\alpha_j}}\colon\mathcal S(\R^2;\R^2)\to \mathcal S(\R^2;\R^2)$ given by Proposition~\ref{prop:Psd}. We define
$$
\beta^{\rm reg}_j\coloneqq \nabla (P_{\rho_j}^{1-{\alpha_j}}\Phi_j)\quad\text{in $\R^2$ for all $j\in\N$},
$$
and by construction $\beta^{\rm reg}_j\in \mathcal S(\R^2;\R^{2\times 2})$ with $\Curl\beta^{\rm reg}_j=0$ in $\mathcal D'(\R^2;\R^2)$. 

Let $(\mu_j)_j\subset\mathcal M(\R^2;\R^2)$ be the sequence of measures given by Lemma~\ref{lem:mu-appr} associated to $(N_j)_j$. We consider the functions $\zeta_{i,j}$, $\zeta_j$, and $\nu_j$ defined in~\eqref{eq:zetaij-cr},~\eqref{eq:zetaj-cr}, and~\eqref{eq:nuj-cr}, respectively. Let $R>0$ be such that $\Omega\subset\subset B_R$. We consider the solution $ w^{\rm rem}_j\in H^1_0(B_R;\R^2)$ to 
\begin{equation*}
\begin{cases}
\Delta w^{\rm rem}_j=-\nu_j &\text{in $B_R$},\\
 w^{\rm rem}_j=0 &\text{on $\partial B_R$},
\end{cases}
\end{equation*}
and we set $\beta^{\rm rem}_j\coloneqq \nabla w_j^{\rm rem}J^T\in L^2(B_R;\R^{2\times 2})$. We define
\begin{equation*}
\beta_j\coloneqq \frac{\sqrt{N_j}}{\sqrt{2\alpha_j}}\beta^{\rm reg}_j+\zeta_j+\beta^{\rm rem}_j\quad\text{in $B_R$}.
\end{equation*}
Arguing as in~\eqref{curlrechnung}, we find by construction that $\beta_j\in\mathcal A_{\alpha_j}(\mu_j)$ for all $j\in\N$. Moreover, by proceeding along the lines of~\eqref{eq:ls-con-1},~\eqref{eq:ls-con-2}, and~\eqref{eq:ls-con-3}, the sequence $(\mu_j,\beta_j)_j$ satisfies~\eqref{eq:con-sub-1}--\eqref{eq:con-sub-2}. Only the derivation of~\eqref{eq:ls-con-3} is slightly different and relies on the estimate (see~\eqref{nununu}) 
$$
\frac{\sqrt{2\alpha_j}}{\sqrt{N_j}}\|\nu_j\|_{L^\infty(B_R)}\le C\frac{\sqrt{2\alpha_jN_j}}{N_j(r_j-\rho_j)^2}\to 0\quad\text{as $j\to\infty$},
$$
where we used~\eqref{eq:a-da-j-4},~\eqref{eq:rj}, and Remark~\ref{rem:mu-appr}(iii). 
Then, we proceed as in Step 1 of the proof of Theorem~\ref{t:critical}(iii) to obtain~\eqref{eq:gamma-limsup-sub}.
\end{proof}

\subsection{The supercritical regime}

In the supercritical regime ($N_\alpha\gg \frac{1}{\alpha}$ as $\alpha\to 0$), we fix three sequences $(\alpha_j)_j$, $(\rho_j)_j$, and $(N_j)_j$ satisfying~\eqref{all convergences} and~\eqref{eq:a-da-j-6}. We recall the functionals defined in~\eqref{Falphasuper} and~\eqref{supercritlimit}. 

\begin{proof}[Proof of Theorem~\ref{t:supercritical}]

{\bf Compactness and $\Gamma$-liminf inequality.} By (C1)--(C3) we have
\begin{align*}
\frac{1}{N_j^2}\|\mathcal I^{\alpha_j}_{\rho_j}\beta_j^{\sym}\|_{L^2(\Omega)}^2\le \frac{1}{\nu_1 N_j^2}\int_\Omega \C\mathcal I^{\alpha_j}_{\rho_j}\beta_j(x):\mathcal I^{\alpha_j}_{\rho_j}\beta_j(x)\,\de x\le C. 
\end{align*}
Hence, there is $\beta\in L^2(\Omega;\R^{2\times2}_{\sym})$ such that, up to a not relabeled subsequence, 
\begin{equation*}
\frac{1}{N_j}\mathcal I^{\alpha_j}_{\rho_j}\beta_j^{\sym}\rightharpoonup\beta\quad\text{in $L^2(\Omega;\R^{2\times 2}_{\sym})$}\quad\text{ as $j\to\infty$}.
\end{equation*}
For the $\Gamma$-liminf inequality, it is enough to observe that the convergence in~\eqref{eq:con-sup} implies
\begin{align*}
&\liminf_{j\to\infty}\frac{1}{2 N_j^2}\int_\Omega\C\mathcal I^{\alpha_j}_{\rho_j}\beta_j(x):\mathcal I^{\alpha_j}_{\rho_j}\beta_j(x)\,\de x\\
&=\liminf_{j\to\infty}\frac{1}{2 N_j^2}\int_\Omega\C\mathcal I^{\alpha_j}_{\rho_j}\beta_j^{\sym}(x):\mathcal I^{\alpha_j}_{\rho_j}\beta_j^{\sym}(x)\,\de x\ge \frac{1}{2}\int_\Omega\C\beta(x):\beta(x)\,\de x.
\end{align*}

{\bf $\Gamma$-limsup inequality.} The proof of the $\Gamma$-limsup inequality is divided into two steps.

{\bf Step~1.} We fix $\beta\in C_c^\infty(\Omega;\R^{2\times 2}_{\sym})$ and we define $\mu\coloneqq \Curl\beta\,\de x\in\mathcal M(\Omega;\R^2)$. By arguing as in~\cite[Theorem~18]{GLP} there exists a constant $C>0$ (depending on $\|\Curl\beta\|_{L^\infty(\Omega)}$) and a sequence of measures $(\mu_j)_j\subset\mathcal M(\R^2;\R^2)$ such that 
\begin{align}
\mu_j= \sum_{i=1}^{M_j}\xi_{i,j}\delta_{x_{i,j}}\quad\text{for all $j$},\qquad |\xi_{i,j}|\le C\quad\text{for all $j$},\qquad \frac{M_j}{N_j}\le C\quad\text{for all $j$},
\label{eq:muj-est-sup}
\end{align}
with
$$
B_{r_j}(x_{i,j})\subset\Omega\quad\text{for all $i,j$},\qquad |x_{i,j}-x_{k,j}|\ge 2r_j\quad\text{if $i\neq k$, for all $j$},\qquad r_j\coloneqq \frac{C}{\sqrt{N_j}}\quad\text{for all $j$},
$$
and satisfying
\begin{align*}
&\frac{1}{N_j}\mu_j \xrightharpoonup{*}\mu\quad\text{in $\mathcal M(\R^2;\R^2)$ as $j\to\infty$}.
\end{align*}
We consider the functions $\zeta_{i,j}$, $\zeta_j$, and $\nu_j$ defined in~\eqref{eq:zetaij-cr},~\eqref{eq:zetaj-cr}, and~\eqref{eq:nuj-cr}, respectively. Moreover, let $R>0$ be such that $\Omega\subset\subset B_R$ and let $ w^{\rm rem}_j\in H^1_0(B_R;\R^2)$ be the solution to 
\begin{equation*}
\begin{cases}
\Delta w^{\rm rem}_j=-N_j\mu-\nu_j &\text{in $B_R$},\\
 w^{\rm rem}_j=0 &\text{on $\partial B_R$}.
\end{cases}
\end{equation*}
We define $\beta^{\rm rem}_j\coloneqq \nabla w_j^{\rm rem}J^T\in L^2(B_R;\R^{2\times 2})$ and we set
\begin{equation*}
\beta_j\coloneqq N_j\beta+\zeta_j+\beta^{\rm rem}_j\quad\text{in $B_R$}.
\end{equation*}
Similar to~\eqref{curlrechnung}, by construction $\beta_j\in\mathcal A_{\alpha_j}(\mu_j)$ for all $j\in\N$, and we claim that 
\begin{equation}\label{eq:aj-bj-con-sup}
\frac{1}{N_j}\mathcal I^{\alpha_j}_{\rho_j}\beta_j\to \beta\quad\text{in $L^2(\Omega;\R^{2\times 2})$ as $j\to\infty$}.
\end{equation}
By~\eqref{eq:a-da-j-6},~\eqref{eq:zetaij-est-new}, and~\eqref{eq:muj-est-sup}, we can find a constant $C>0$ (independent of $j$) such that
\begin{align}\label{eq:ls-con-sup-1}
\frac{1}{N_j^2}\int_\Omega|\mathcal I^{\alpha_j}_{\rho_j}\zeta_j(x)|^2\,\de x\le \frac{C}{(\alpha_j N_j)^2}\to 0\quad\text{as $j\to\infty$},
\end{align}
i.e., $\frac{1}{N_j}\mathcal I^{\alpha_j}_{\rho_j}\zeta_j\to 0$ in $L^2(\Omega;\R^{2\times 2})$ as $j\to\infty$. Moreover, thanks to Corollary~\ref{coro:Iad-conv}, we derive that
\begin{align}\label{eq:ls-con-sup-2}
\mathcal I^{\alpha_j}_{\rho_j}\beta\to\beta\quad\text{in $L^2(\Omega;\R^{2\times 2})$ as $j\to\infty$},
\end{align}
and, by arguing as in~\eqref{eq:ls-con-3} (note that we can identify $2\alpha_j$ with $1/N_j$ in~\eqref{hatw} and~\eqref{eq:ls-con-3}) we have $
\frac{1}{N_j}\mathcal I^{\alpha_j}_{\rho_j}\beta^{\rm rem}_j\to 0$ in $L^2(\Omega;\R^{2\times 2})$ as $j\to\infty$. By combining this with~\eqref{eq:ls-con-sup-1} and~\eqref{eq:ls-con-sup-2}, we obtain~\eqref{eq:aj-bj-con-sup}.

Finally, we use~\eqref{eq:aj-bj-con-sup} to get
\begin{align*}
&\lim_{j\to\infty}\frac{1}{2N_j^2}\int_\Omega\C\mathcal I^{\alpha_j}_{\rho_j}\beta_j(x):\mathcal I^{\alpha_j}_{\rho_j}\beta_j(x)\,\de x=\frac{1}{2}\int_\Omega\C\beta(x):\beta(x)\,\de x,
\end{align*}
which proves~\eqref{eq:gamma-limsup-sup} in the case $\beta\in C_c^\infty(\Omega;\R^{2\times 2}_{\sym})$.

{\bf Step~2.} Let $\beta \in L^2(\Omega;\R^{2\times 2}_{\sym})$. Then, there exists a sequence $(\beta^n)_n\subset C_c^\infty(\Omega;\R^{2\times 2}_{\sym})$ such that 
\begin{align*}
&\beta^n\to\beta\quad\text{in $L^2(\Omega;\R^{2\times 2}_{\sym})$ as $n\to\infty$},
\end{align*}
which implies
\begin{equation*}
\lim_{n\to\infty}\mathcal F^{\rm super}(\beta^n)=\mathcal F^{\rm super}(\beta).
\end{equation*}
By Step~1, for all $n\in\N$ there exists a sequence $(\mu_j^n,\beta_j^n)_j$ with $(\mu_j^n,\beta_j^n)\in \mathcal X_{\alpha_j}\times \mathcal A_{\alpha_j}(\mu_j^n)$ for all $j\in\N$ such that 
\begin{align*}
&\frac{1}{N_j}\mathcal I^{\alpha_j}_{\rho_j}\beta_j^n\to \beta^n\quad\text{in $L^2(\Omega;\R^{2\times 2})$ as $j\to\infty$}, \quad \quad 
\lim_{j\to\infty}\mathcal F_{\alpha_j}(\mu_j^n,\beta_j^n)= \mathcal F^{\rm super}(\beta^n).
\end{align*}
Therefore, to obtain~\eqref{eq:gamma-limsup-sup}, it is enough to use a standard diagonal argument.
\end{proof}

\section*{Acknowledgments}
The work of S.A.\ was supported by the Austrian Science Fund through the project 10.55776/P353\\59 and by the University of Naples Federico II through the FRA project ``ReSinApas''. The work of S.A., M.C., and F.S.\ is part of the MUR - PRIN 2022, project Variational Analysis of Complex Systems in Materials Science, Physics and Biology, No. 2022HKBF5C, funded by European Union NextGenerationEU. The work of M.C.\ and F.S.\ is part of the project STAR PLUS 2020 - Linea 1 (21-UNINA-EPIG-172) ``New perspectives in the Variational modeling of Continuum Mechanics'' from the University of Naples Federico II and Compagnia di San Paolo. The work of M.C.\ was supported by the INdAM-GNAMPA Project ``Pairing e div-curl lemma: estensioni a campi debolmente derivabili e differenziazione non locale'' (CUP E53C23001670001). The work of M.F.\ was funded by the DFG project FR 4083/5-1 and by the Deutsche Forschungsgemeinschaft (DFG, German Research Foundation) under Germany's Excellence Strategy EXC 2044 -390685587, Mathematics M\"unster: Dynamics--Geometry--Structure. 

\noindent Finally, S.A., M.C., and F.S. are part of the Gruppo Nazionale per l'Analisi Matematica, la Probabilit\`a e le loro Applicazioni (INdAM-GNAMPA).

\appendix

\section{Riesz potentials}

In the appendix, we collect some results regarding the Riesz potentials used in this paper. We start with the classical Riesz potential $\mathcal I^\alpha$.

\begin{proposition}\label{prop:Riesz}
Let $\alpha\in(0,2)$ and let $f\colon\R^2\to\R$ be a measurable function satisfying~\eqref{eq:Ia-hyp}. Then $\mathcal I^\alpha f(x)$ is well-defined for a.e.\ $x\in\R^2$ and $\mathcal I^\alpha f\in L^1_{\loc}(\R^2)$. 
\end{proposition}

\begin{proof}
Let $f\colon\R^2\to\R$ be satisfying~\eqref{eq:Ia-hyp}. The fact that $\mathcal I^\alpha f(x)$ is well-defined for a.e.\ $x\in\R^2$ is a consequence of~\cite[Theorem~1.1, Chapter~2]{Mizuta}. Moreover, in this case, $\mathcal I^\alpha f$ is locally integrable on $\R^2$, as observed in~\cite[Section~4.2]{Mizuta}. 
\end{proof}

We also recall the following composition formula for the Riesz potential.

\begin{proposition}[{\cite[Theorem~1.6, Chapter~2]{Mizuta}}]\label{prop:Riesz-comp}
Let $\alpha,\beta\in (0,2)$ be such that $\alpha+\beta\in (0,2)$. Then,
\begin{equation*}
\int_{\R^2}\frac{1}{|x-y|^{2-\alpha}}\frac{1}{|y-z|^{2-\beta}}\,\de y=\frac{\gamma_\alpha\gamma_\beta}{\gamma_{\alpha+\beta}}\frac{1}{|x-z|^{2-\alpha-\beta}}\quad\text{for all $x,z\in\R^2$ with $x\neq z$}.
\end{equation*}
\end{proposition}

Let us now consider the Riesz potential with finite horizon $\mathcal I^\alpha_\rho$. We first show that $\mathcal I^\alpha_\rho$ is well-defined. Recall~\eqref{neigh}.

\begin{lemma}\label{lem:Iad-int}
Let $\alpha\in(0,1)$, $\rho>0$, and $f\in L^1_{\loc}(\R^2)$. Then $\mathcal I^\alpha_\rho f(x)$ is well-defined for a.e.\ $x\in\R^2$ and $\mathcal I^\alpha_\rho f\in L^1_{\loc}(\R^2)$. Moreover, if $E\subseteq \R^2$ is a measurable set and $f\in L^p(E_\rho)$ for some $p\in \left[1,\frac{2}{\alpha}\right)$, then $\mathcal I^\alpha_\rho f\in L^q(E)$ for all $q\in \left[p,\frac{2p}{2-\alpha p}\right)$ and 
$$
\|\mathcal I^\alpha_\rho f\|_{L^q(E)}\le \|Q_\rho^{1-\alpha}\|_{L^\frac{pq}{pq+p-q}(\R^2)}\|f\|_{L^p(E_\rho)}.
$$
In particular, if $f\in L^p(E_\rho)$ for all $p\in [1,2)$, then $\mathcal I^\alpha_\rho f \in L^2(E)$ for all $\alpha\in(0,1)$. 
\end{lemma}

\begin{proof}
Let $\alpha\in (0,1)$ and $\rho>0$ be fixed. By~\eqref{eq:Qsd} and the fact that $0 \le \overline w_\rho \le 1$ we can derive the estimate 
\begin{equation}\label{eq:Qsd-est}
0\le Q^{1-\alpha}_\rho(x)\le\frac{1}{\gamma_\alpha}\frac{1}{|x|^{2-\alpha}}\quad\text{for all $x\in\R^2\setminus\{0\}$, $\alpha\in (0,2)$, and $\rho>0$}. 
\end{equation}
Hence, as $\supp (Q^{1-\alpha}_\rho)\subset B_\rho$ by~\eqref{eq:Qsd-p}, we derive
\begin{equation*}
Q^{1-\alpha}_\rho\in L^r (\R^2)\quad\text{for all $r\in \left[1,\frac{2}{2-\alpha}\right)$ and $\rho>0$},
\end{equation*}
and both statements follow from Young's convolution inequality.
\end{proof}

Similarly to classical Riesz potential $\mathcal I^\alpha$, we can prove the following asymptotic result for $\mathcal I^\alpha_\rho$.

\begin{lemma}\label{lem:Iad-con}
Let $\alpha\in (0,1)$, $\rho>0$, and $x\in\R^2$. Let $f\colon B_\rho(x)\to\R$ be continuous in $x$ and bounded. Then,
\begin{equation}\label{eq:Iad-lim}
\lim_{\alpha\to 0}\mathcal I^\alpha_\rho f(x)=f(x),\qquad|\mathcal I^\alpha_\rho f(x)|\le \frac{2\pi\rho^\alpha\|f\|_{L^\infty(B_\rho(x))}}{\alpha\gamma_\alpha}.
\end{equation}
\end{lemma}

\begin{proof}
For all $\varepsilon>0$ there exists $\ell=\ell(\varepsilon)>0$ such that 
\begin{align}\label{eeps}
|f(y)-f(x)|<\varepsilon \quad\text{for all $y\in B_\ell(x)$}.
\end{align}
We set $\sigma\coloneqq \min\{\ell,\frac{\rho}{2}\}$ and write
\begin{align*}
&\mathcal I^\alpha_\rho f(x)=g^1_{\alpha,\rho}(x)+g^2_{\alpha,\rho}(x)\coloneqq\int_{B_\sigma(x)}f(y)Q^{1-\alpha}_\rho(x-y)\,\de y+\int_{\R^2\setminus B_\sigma(x)}f(y)Q^{1-\alpha}_\rho(x-y)\,\de y.
\end{align*}
By Remark~\ref{rem:gamma-a},~\eqref{eq:Qsd-p}, and~\eqref{eq:Qsd-est}, as $\alpha\to 0$, we have
\begin{align}\label{combi}
|g^2_{\alpha,\rho}(x)|&\le \frac{1}{\gamma_\alpha}\int_{B_\rho(x)\setminus B_\sigma(x)}\frac{|f(y)|}{|x-y|^{2-\alpha}}\,\de y\notag\\
&\le\frac{2\pi\|f\|_{L^\infty(B_\rho(x))}}{\gamma_\alpha}\int_\sigma^\rho r^{\alpha-1}\,\de r\le\frac{2\pi\|f\|_{L^\infty(B_\rho(x))}}{\alpha\gamma_\alpha}(\rho^\alpha-\sigma^\alpha)\to 0.
\end{align}
Since by the definition of $\overline{w}$ it holds that
$$1-\overline w_\rho(s)=0\quad\text{for all $s\in\left[0,\frac{\rho}{2}\right]$},$$
by using~\eqref{eq:Qsd}, Remark~\ref{rem:gamma-a}, and the fact that $\sigma=\min\{\ell,\frac{\rho}{2}\}$, we derive that
\begin{align*}
\left|\frac{2\pi \sigma^\alpha}{\alpha\gamma_\alpha}-\int_{B_\sigma}Q^{1-\alpha}_\rho(y)\,\de y\right|&=\frac{2\pi(2-\alpha)}{\gamma_\alpha}\int_0^\sigma r\int_r^\infty\frac{1-\overline w_\rho(s)}{s^{3-\alpha}}\,\de s\,\de r\\
&\le \frac{2\pi(2-\alpha)}{\gamma_\alpha}\int_0^\sigma r\int_{\frac{\rho}{2}}^\infty\frac{1}{s^{3-\alpha}}\,\de s\,\de r=\frac{2^{2-\alpha}\pi\sigma^2}{\gamma_\alpha\rho^{2-\alpha}}\le \frac{\pi \rho^{\alpha}}{\gamma_\alpha 2^\alpha}.
\end{align*}
Moreover, we can write
$$f(x)=\int_{B_\sigma(x)}f(x)Q^{1-\alpha}_\rho(x-y)\,\de y+f(x)\left(\frac{2\pi \sigma^\alpha}{\alpha\gamma_\alpha}-\int_{B_\sigma}Q^{1-\alpha}_\rho(y)\,\de y\right)+f(x)\left(1-\frac{2\pi \sigma^\alpha}{\alpha\gamma_\alpha}\right).$$
Therefore, by Remark~\ref{rem:gamma-a},~\eqref{eq:Qsd-est}, and~\eqref{eeps}, we have, as $\alpha\to 0$,
\begin{align*}
|g^1_{\alpha,\rho}(x)-f(x)|&\le |f(x)|\left|1-\frac{2\pi \sigma^\alpha}{\alpha\gamma_\alpha}\right|+|f(x)|\frac{\pi \rho^{\alpha}}{\gamma_\alpha 2^\alpha}+\frac{1}{\gamma_\alpha}\int_{B_\sigma(x)}\frac{|f(y)-f(x)|}{|x-y|^{2-\alpha}}\,\de y\\
&\le|f(x)|\left|1-\frac{2\pi \sigma^\alpha}{\alpha\gamma_\alpha}\right|+|f(x)|\frac{\pi \rho^{\alpha}}{\gamma_\alpha 2^\alpha}+\varepsilon\frac{2\pi \sigma^\alpha}{\alpha\gamma_\alpha}\to \varepsilon.
\end{align*}
Hence, combining with~\eqref{combi} we get
\begin{equation*}
\limsup_{\alpha\to 0}\left|\mathcal I^\alpha_\rho f(x)-f(x)\right|\le \varepsilon\quad\text{for all $\varepsilon>0$},
\end{equation*}
which gives the limit in~\eqref{eq:Iad-lim}. Finally, the estimate in~\eqref{eq:Iad-lim} can be obtained by using~\eqref{eq:Qsd-est}.
\end{proof}

Actually, the same result holds also true in the case of a sequence of horizons $(\rho_\alpha)_{\alpha\in (0,1)}$, as long as $\alpha\log\rho_\alpha\to 0$ as $\alpha\to 0$. 

\begin{corollary}\label{coro:Iad-conv}
Let $(\rho_\alpha)_{\alpha\in (0,1)}\subset (0,\infty)$ be such that 
$$
\alpha\log\rho_\alpha\to 0\quad\text{as $\alpha\to 0$}.
$$
Let $R\coloneqq \sup\{\rho_\alpha:\alpha\in (0,1)\}$, $x\in\R^2$, and $f\colon B_R(x)\to \R$ be continuous in $x$ and bounded. Then
\begin{align*}
\lim_{\alpha\to 0}\mathcal I^\alpha_{\rho_\alpha} f(x)=f(x).
\end{align*}
Moreover, if $f\in C_c(\R^n)$, then 
\begin{equation}\label{eq:Iad-unif-lim}
\mathcal I^\alpha_{\rho_\alpha} f\to f\quad\text{in $C(\R^2)$ as $\alpha\to 0$}.
\end{equation}
\end{corollary}

\begin{proof}
We proceed as in Lemma~\ref{lem:Iad-con}: given $\varepsilon>0$, we choose $\ell$ as in~\eqref{eeps}, and then $\sigma_\alpha\coloneqq\min\{\ell,\frac{\rho_\alpha}{2}\}$ for all $\alpha\in (0,1)$. Therefore, as $\alpha\to 0$, we get
\begin{align*}
&|\mathcal I^\alpha_{\rho_\alpha}f(x)-f(x)| \le\frac{2\pi\|f\|_{L^\infty(B_R(x))}}{\alpha\gamma_\alpha}(\rho_\alpha^\alpha-\sigma_\alpha^\alpha)+|f(x)|\left|1-\frac{2\pi \sigma_\alpha^\alpha}{\alpha\gamma_\alpha }\right|+|f(x)|\frac{\pi \rho_\alpha^{\alpha}}{\gamma_\alpha 2^\alpha}+\varepsilon\frac{2\pi \sigma_\alpha^\alpha}{\alpha\gamma_\alpha}\to \varepsilon,
\end{align*}
since $\alpha\gamma_\alpha \to 2\pi$ by Remark~\ref{rem:gamma-a} and $\rho_\alpha^\alpha=\e^{\alpha\log\rho_\alpha}\to 1$ as $\alpha\to 0$.

Now, consider $f\in C_c(\R^2)$. Given $\varepsilon>0$, we can find $\ell>0$, independently of $x$, such that~\eqref{eeps} holds for all $x\in\R^2$. By arguing as before, we get
$$
\|\mathcal I^\alpha_{\rho_\alpha} f-f\|_{L^\infty(\R^2)}\le\frac{2\pi\|f\|_{L^\infty(\R^2)}}{\alpha\gamma_\alpha}(\rho_\alpha^\alpha-\sigma_\alpha^\alpha)+\|f\|_{L^\infty(\R^2)}\left|1-\frac{2\pi \sigma_\alpha^\alpha}{\alpha\gamma_\alpha }\right|+\|f\|_{L^\infty(\R^2)}\frac{\pi \rho_\alpha^{\alpha}}{\gamma_\alpha 2^\alpha}+\varepsilon\frac{2\pi \sigma_\alpha^\alpha}{\alpha\gamma_\alpha},
$$
which implies~\eqref{eq:Iad-unif-lim} by the arbitrariness of $\varepsilon$.
\end{proof}

We conclude with the following lemma, which allows us to compute $\Curl\mathcal I^\alpha_\rho f$ on $\Omega$ by means of $\Curl f$ on the enlarged domain $\Omega_\rho$.

\begin{lemma}\label{lem:curl-Iadxi}
Let $\alpha\in (0,1)$, $\rho>0$, and $\xi\in\R^2$ be fixed. Let $\Omega\subseteq\R^2$ be an open set with $0 \in \Omega$and let $f\in L^1(\Omega_\rho;\R^{2\times 2})$ be such that $$\Curl f=\xi\delta_0\quad\text{in $\mathcal D'(\Omega_\rho;\R^2)$}.$$ 
Then, $\mathcal I^\alpha_\rho f \in L^1(\Omega;\R^{2\times 2})$, $\Curl \mathcal I^\alpha_\rho f\in L^1(\Omega;\R^2)$, and 
\begin{equation}\label{eq:curl-Iadxi}
\Curl \mathcal I^\alpha_\rho f=\xi Q^{1-\alpha}_\rho\quad\text{in $\mathcal D'(\Omega;\R^2)$.} 
\end{equation}
\end{lemma}

\begin{proof}
Since $f\in L^1(\Omega_\rho;\R^{2\times 2})$, we get that $\mathcal I^\alpha_\rho f=Q_\rho^{1-\alpha}*f\in L^1(\Omega;\R^{2\times 2})$ by Lemma~\ref{lem:Iad-int}. Since $Q^{1-\alpha}_\rho \in L^1(\Omega)$, it remains to prove formula~\eqref{eq:curl-Iadxi}. To this end, we fix $\Phi\in C_c^\infty(\Omega;\R^2)$. By~\eqref{curlicurl}, Fubini's theorem, and the symmetry of $Q^{1-\alpha}_\rho$ we have
\begin{align*}
\langle\Curl \mathcal I^\alpha_\rho f,\Phi\rangle_{\mathcal D'(\Omega)}&=\langle \Div(\mathcal I^\alpha_\rho f J),\Phi\rangle_{\mathcal D'(\Omega)}=-\int_\Omega\mathcal I^\alpha_\rho f(x) J:\nabla\Phi(x)\,\de x\\
&=-\int_\Omega\left(\int_{B_\rho(x)}f(y)JQ^{1-\alpha}_\rho( x-y )\,\de y\right):\nabla \Phi(x)\,\de x\\
&=-\int_{\Omega_\rho}f(y)J:\left(\int_{B_\rho(y)}\nabla \Phi(x) Q^{1-\alpha}_\rho(y-x)\,\de x\right)\,\de y\\
&=-\int_{\Omega_\rho}f(y)J: \mathcal I^\alpha_\rho \nabla\Phi(y)\,\de y.
\end{align*}
By Proposition~\ref{prop:Psd} we derive that $\mathcal I^\alpha_\rho \Phi\in C_c^\infty(\Omega_\rho; \R^2)$ and $\mathcal I^\alpha_\rho \nabla\Phi=\nabla \mathcal I^\alpha_\rho \Phi$. Hence, again by~\eqref{curlicurl},
\begin{align}\label{lastli}
\langle\Curl \mathcal I^\alpha_\rho f,\Phi\rangle_{\mathcal D'(\Omega)}&=-\int_{\Omega_\rho}f (y)J: \nabla \mathcal I^\alpha_\rho \Phi(y)\,\de y\notag\\
&=\langle \Curl f,\mathcal I^\alpha_\rho \Phi\rangle_{\mathcal D'(\Omega_\rho)}=\xi\cdot \mathcal I^\alpha_\rho \Phi(0)=\int_\Omega\Phi(y)\cdot\xi Q^{1-\alpha}_\rho( y )\,\de y.
\end{align}
Therefore, $\Curl \mathcal I^\alpha_\rho f\in L^1(\Omega;\R^2)$ and formula~\eqref{eq:curl-Iadxi} is satisfied. 
\end{proof}

%
%
%
%
%


\end{document}